\documentclass[10pt]{amsart}
\usepackage{amsmath,amscd}
\usepackage{graphicx}
\usepackage[all]{xy}
\usepackage{amsthm}
\usepackage{amssymb,url,extarrows}
\usepackage{color}
\usepackage{subfigure}
\usepackage{marginnote}
\usepackage[bookmarks=true]{hyperref}  

\newtheorem{thm}{Theorem}[section]
\newtheorem{thm*}{Theorem}
\newtheorem{lem}[thm]{Lemma}
\newtheorem{cor}[thm]{Corollary}

\newtheorem{prop}[thm]{Proposition}

\theoremstyle{remark}
\newtheorem{remark}[thm]{Remark}
\newtheorem{example}[thm]{Example}

\theoremstyle{definition}
\newtheorem{deef}[thm]{Definition}

\newcommand{\R}{\mathbb{R}}
\newcommand{\C}{\mathbb{C}}

\newcommand{\Z}{\mathbb{Z}}
\newcommand{\Q}{\mathbb{Q}}
\newcommand{\N}{\mathbb{N}}

\newcommand{\rd}{\mathrm{d}}

\newcommand{\Fg}{\mathcal{F}}

\renewcommand{\epsilon}{\varepsilon}

\newcommand{\e}{\mathrm{e}}

\title[Semiclassical analysis of a nonlocal boundary value problem]{Semiclassical analysis of a nonlocal boundary value problem related to magnitude}
\author{Heiko Gimperlein, Magnus Goffeng, Nikoletta Louca}
\address{Heiko Gimperlein\newline
\indent Leopold-Franzens-Universit\"{a}t Innsbruck\newline
\indent Technikerstraße 13 \newline
\indent 6020 Innsbruck\newline
\indent Austria \newline\newline
\indent Magnus Goffeng,\newline
\indent Centre for Mathematical Sciences\newline 
\indent University of Lund\newline 
\indent Box 118, SE-221 00 Lund\newline 
\indent Sweden\newline
\newline
\indent Nikoletta Louca\newline
\indent Maxwell Institute for Mathematical Sciences and \newline
\indent Department of Mathematics, Heriot-Watt University\newline
\indent Edinburgh EH14 4AS\newline
\indent United Kingdom\newline
}
\subjclass[2010]{}
\email{heiko.gimperlein@uibk.ac.at, magnus.goffeng@math.lth.se, nl24@hw.ac.uk}

\begin{document}

\begin{abstract}
We study a Dirichlet boundary problem related to the fractional Laplacian in a manifold. Its variational formulation arises in the study of magnitude, an invariant of compact metric spaces given by the reciprocal of the ground state energy. Using recent techniques developed for pseudodifferential boundary problems we  discuss the structure of the solution operator and resulting properties of the magnitude. In a semiclassical limit we obtain an asymptotic expansion of the magnitude in terms of curvature invariants of the manifold and the boundary, similar to the invariants arising in short-time expansions for heat kernels. 
\end{abstract}

\maketitle

\section{Introduction}
The analysis of boundary problems for nonlocal operators has attracted much interest in recent years, including Dirichlet and Neumann problems for fractional Laplacians. In this article we  initiate the semiclassical analysis of related boundary problems motivated by applications to the Leinster-Willerton conjecture for the magnitude invariant of compact metric spaces. 

To be specific, we consider the integral equation with parameter $R>0$
\begin{equation}
\label{magequa}
\int_X \e^{-R\rd(x,y)}u(y)\rd y=f(x).
\end{equation}
Here $(X,\rd)$ is a compact metric space, and we focus on when $X$ is a manifold with boundary and $\rd$ is a distance function satisfying additional regularity assumptions. Already when $X\subseteq \R^2$ is the unit disc, close to nothing was known for the solutions to \eqref{magequa}. We shall prove in this paper that for $X$ a compact $n$-dimensional manifold with boundary, Equation \eqref{magequa} is well posed for $f$ in the Sobolev space $\overline{H}^{(n+1)/2}(X)$ and admits a unique solution $u_R\in \dot{H}^{-(n+1)/2}(X)$ for sufficiently large $R\gg 0$ (for notation, see page \pageref{qxondomains}). We relate the integral equation \eqref{magequa} to a pseudodifferential boundary value problem which is elliptic with parameter. 

Our main results concern structural properties of solutions to Equation \eqref{magequa} such as asymptotic behavior as $R\to \infty$ and meromorphic extensions in the parameter $R$ to sectors $\Gamma\subseteq \C$. The methods for pseudodifferential boundary value problems that we use date far back, see the work of Gregory Eskin \cite{eskinbook} and Lars Hörmander \cite{hornotes}, but have in recent years seen much development in work of Gerd Grubb \cite{g2,g3,g4,grubb1,grubbibp}.

The solution to \eqref{magequa} for the right hand side $f=1$ enters in the so called magnitude function of $(X,\rd)$, studied extensively in for instance \cite{barcarbs,gimpgoff,gimpgoff2,leinster,leinmeck,leinwill,meckes1,meckes,willna,will}. The empirical properties of the solution to \eqref{magequa} have recently found applications in data science, leading to precise conjectures for its structural properties  \cite{bunch1,bunch2}. 

The case  $f=1$ can be considered as a minimizing problem relating to the ground state energy
\begin{align*}
\mathcal{E}(R;X,\rd):=&\inf\left\{\int_X\int_X \e^{-R\rd(x,y)}\rd\mu(x)\rd\mu(y): \; \mu\;\mbox{a signed Borel measure with $\mu(X)=1$}\right\}=\\
=&\inf\left\{\sum_{x,y\in X} \e^{-R\rd(x,y)}c(x)c(y): \; \mbox{$c:X\to \R$ has finite support and $\sum_{x\in X} c(x)=1$}\right\}.
\end{align*}
More precisely, if $R$ is such that $(X,R\rd)$ is positive definite (i.e. the matrix $(\e^{-R\rd(x,y)})_{x,y\in F}$ is positive definite for any finite $F\subseteq X$), then by \cite{meckes} a solution $u_R$ to Equation \eqref{magequa} with $f=1$ satisfies 
$$\int_X u_R(x)\rd x=\frac{1}{\mathcal{E}(R;X,\rd)}.$$

Let us digest on the problem of finding $\mathcal{E}(R;X,\rd)$ and studying its semiclassical limit, as it has been broadly studied in various mathematical communities. The ground state energy $\mathcal{E}(R;X,\rd)$ is that of a signed distribution of finitely many particles on $X$ where a particle in $x$ interacts with a particle in $y$ under the potential $\e^{-R\rd(x,y)}$. As such, the scaling parameter $R>0$ should be thought of as an order parameter with $R\to \infty$ corresponding to a semiclassical limit. The non-locality of Equation \eqref{magequa} and the ground state energy $\mathcal{E}(R;X,\rd)$ makes the problem of explicit computation an impossibility, however in the semiclassical limit  the problem localizes and is asymptotically described in terms of geometric invariants. 
Related problems concerning ground state energies with nonlocal interaction potentials arise in the mean field description of interacting particle systems, such as \cite{franklieb}. Specifically for log gases the dependence of the ground state energy on the geometry has been of interest \cite{loggas}. In complex geometry, Berman has studied similar minimization problems from which geometric structures emerged \cite{bermanwitt,bermandet}. Related operators also appear in image processing \cite{antil}.

The integral equation \eqref{magequa} is, as mentioned above, related to magnitude -- an invariant that has been extensively studied since it was  introduced by Tom Leinster \cite{leinster}. We presume no prerequisites from the reader on magnitude in this paper, but for the convenience of the reader we summarize the implications to magnitude here and expand on this relation in the follow up paper \cite{gimpgofflouc}. From its category-theoretic origin, magnitude has found unexpected applications from algebraic topology \cite{govchep,ls,summerv} and applied category theory \cite{chocho,ottermag} to data science \cite{bunch1,bunch2} and mathematical biology \cite{leincobb}. 

For a metric space $(X,\rd)$ this invariant leads to a function $\mathcal{M}_X : (0,\infty) \to \mathbb{R}\cup \{\infty\}$. When the metric space $(X,R\rd)$ is positive definite, and in particular for compact sets $X \subset \mathbb{R}^n$ \cite{meckes1}, $\mathcal{M}_X(R)$ is defined as $\mathcal{M}_X(R)=\int_X u_R(x)\mathrm{d}x$, where $u_R$ satisfies the magnitude equation 
\begin{equation}
\label{magequation}
\int_X \e^{-R\rd(x,y)}u(y)\rd y=1.
\end{equation}
The work \cite{meckes} provides an abstract Hilbert space framework in which to pose this equation.
 
In the case of compact sets $X \subset \mathbb{R}^n$, Meckes \cite{meckes} gives an interpretation of the magnitude in potential theoretic terms, as a generalized capacity:
\begin{equation}
\label{magsobolev}
\mathcal{M}_X(R) = \frac{1}{Rn! \omega_n} \inf \left\{\|(R^2+\Delta)^{(n+1)/4} h_R\|_{L^2(\mathbb{R}^n)}^2 : h_R\in H^{(n+1)/2}(\mathbb{R}^n),\ h_R =1 \ \text{on}\ X\right\}\ .
\end{equation}
Here $\omega_n$ denotes the volume of the unit ball in $\R^n$. The minimizer of \eqref{magsobolev} is attained at a function $h_R\in H^{(n+1)/2}(\mathbb{R}^n)$ solving the non-local exterior problem 
\begin{equation}
\label{magsobolevexterior}
\begin{cases}
(R^2+\Delta)^{(n+1)/2} h_R=0, \; &\mbox{in $\R^n\setminus X$},\\
h_R=1, \; &\mbox{in $X$}.
\end{cases}
\end{equation}
If $n$ is odd, this was studied as a boundary value problem for the five-dimensional unit ball in \cite{barcarbs}, and extended to odd-dimensional unit balls in \cite{willoddb}. Few explicit computations of magnitude are known outside the realm of compact domains in odd-dimensional Euclidean space, and even there the state-of-the art \cite{gimpgoff,gimpgoff2} can only provide asymptotic results in the semiclassical limit and ensure existence of meromorphic extensions. In particular, nothing was previously known about magnitude even in such a simple case as the unit disk $X=B_2\subseteq \R^2$.

We provide a framework for a refined analysis and explicit computations for solutions to Equation \eqref{magequa} when $X$ is a smooth, $n$-dimensional, compact manifold with boundary, independent of the parity of $n$. The framework relies on recent advances for pseudodifferential boundary problems and initiates their semiclassical analysis. We work under certain regularity assumptions on the distance function $\rd$, firstly that its square is \emph{regular at the diagonal} (see Definition \ref{regularlaododal}) and secondly that it has \emph{property (MR)} (see Definition \ref{sovodlwwowm} and \ref{sovodlwwowmbodu}). Our first assumption ensures that the distance function behaves to leading term as a Euclidean distance, and is satisfied by any geodesic distance function or a pullback thereof under an embedding. The first assumption ensures that the diagonal behavior in  Equation \eqref{magequa} is governed by a pseudodifferential operator of order $-n-1$ which is elliptic with parameter. Our second assumption -- property (MR) -- is a technical condition to ensure that the off-diagonal behavior in  Equation \eqref{magequa} is negligible. Property (MR) is satisfied for subspace distances in manifolds whose distance functions squared are smooth, but it in fact fails for higher dimensional tori and real projective spaces.

{ The reader can note that there is an extended arXiv version of this paper \cite{gimpgoffloucaarXiv} containing more details and overview.}

\subsection{Main results}
Let us summarize the main results of this paper. The results all circle around the family of integral operators
$$\mathcal{Z}_X(R)u(x):=\frac{1}{R} \int_X \e^{-R\rd(x,y)}u(y)\rd y, \quad R\in \C\setminus \{0\}.$$
Here $X$ is a compact manifold with boundary equipped with a distance function $\rd$ and a volume density $\rd y$. We assume that $\rd^2$ is smooth in a small neighborhood of the diagonal $x=y$ and there in local coordinates admits a Taylor expansion (for any $N>0$)
\begin{equation}
\label{talaldadladldaladlda}
\rd(x,y)^2=H_{\rd^2,x}(v)+\sum_{j=3}^\infty C^{j}_{\rd^2}(x;v)+O(|v|^{N+1}).
\end{equation}
Here $v=x-y$, and where $H_{\rd^2}$ is a Riemannian metric on $X$ and $C^{j}_{\rd^2}$ in local coordinates is a symmetric $j$-form in $v$. This condition can be summarized in the terminology that \emph{$\rd^2$ is regular at the diagonal}, see Definition \ref{regularlaododal} and for more details on the Taylor expansion, see Equation \eqref{taylorexpamdmd}. We fix a function $\chi\in C^\infty(X\times X)$ such that $\chi=1$ on a neighborhood of the diagonal $x=y$ and $\rd^2$ is smooth on the support of $\chi$. The localization of $\mathcal{Z}$ to near the diagonal is the integral operator 
$$Q_X(R)u(x):=\frac{1}{R} \int_X \chi(x,y)\e^{-R\rd(x,y)}u(y)\rd y, \quad R\in \C\setminus \{0\}.$$
We remark that if $\rd^2$ is smooth on all of $X\times X$, e.g.~for a domain  or a submanifold in $\R^n$ with the induced metric, it holds that $\mathcal{Z}_X-Q_X$ is smoothing with parameter on any sector $\Gamma\subseteq \C_+$ with opening angle $<\pi/2$.

\begin{thm}
\label{techmain}
Let $X$ be a compact $n$-dimensional manifold with boundary and $\rd$ such that $\rd^2$ is regular at the diagonal (see Definition \ref{regularlaododal}). The family of integral operators $Q_X$ is an elliptic pseudodifferential operator with parameter $R\in \C_+$ of order $-n-1$, and its principal symbol is 
$$\overline{\sigma}_{-n-1}(Q_{X})(x,\xi,R)=n!\omega_n(R^2+g_{\rd^2}(\xi,\xi))^{-(n+1)/2},$$
where $g_{\rd^2}$ is the dual metric to $H_{\rd^2}$ from the Taylor expansion \eqref{talaldadladldaladlda}. The properties of $Q_X$ can be summarized as follows:
\begin{enumerate}
\item In each coordinate chart, the full symbol of $Q_X$ can be computed by an iterative scheme as in Theorem \ref{firstfirstofz}. 
\item There exists an $R_0$ such that 
$$Q_X(R): \dot{H}^{-\frac{n+1}{2}}(X)\to \bar{H}^{\frac{n+1}{2}}(X),$$
is invertible for $\mathrm{Re}(R)>R_0$ and $\mathrm{arg}(R)<\pi/(n+1)$. Here $\dot{H}^{-\frac{n+1}{2}}(X)$, respectively $\bar{H}^{\frac{n+1}{2}}(X)$, denote the Sobolev spaces of supported, respectively extendable distributions in $X$ (see Section \ref{qxondomains}).  
\item If $\partial X=\emptyset$, then $Q_{X}(R)^{-1}$ is an elliptic pseudodifferential operator of order $n+1$ whenever it exists. The full symbol of $Q_X^{-1}$ can be computed by an iterative scheme as in Theorem \ref{symbolstructureinversepsido}.  
\end{enumerate}
Moreover, if $\rd^2$ is smooth then all the properties above hold also for $\mathcal{Z}_X$.
\end{thm}

Theorem \ref{techmain} is found in the bulk of the text as follows. The first statement and item (1) is found in Theorem \ref{firstfirstofz}. Item (2) is proven in Theorem \ref{symbcorboundaryq}, see also Corollary \ref{cortombolstructureinversepsido} of Theorem \ref{symbcor} for the simpler case of $\partial X=\emptyset$. Item (3) follows from Theorem \ref{symbolstructureinversepsido} and Corollary \ref{cortombolstructureinversepsido}. 

The operator $Q_X$ is generally more well behaved than $\mathcal{Z}_X$; the off-diagonal singularities of $\rd$ can create problems in considering $\mathcal{Z}_X$ as a map between Sobolev spaces. For examples of such phenomena, see Subsection \ref{countexsobrefl}. We impose one of two conditions on $\rd$; property (MR) and property (SMR) respectively to ensure that  $Q_X$ and $\mathcal{Z}_X$ share common functional analytic features as operators between Sobolev space. { The precise definition of property (MR) and property (SMR) may be found }in Definition \ref{sovodlwwowm} (for $\partial X=\emptyset$) and Definition \ref{sovodlwwowmbodu} (for $\partial X\neq \emptyset$). We note that property (MR) and property (SMR) hold on any sector $\Gamma\subseteq \C\setminus \{0\}$ as soon as  $\rd^2$ is smooth on all of $X\times X$, e.g. for a domain in $\R^n$ or more generally for a compact submanifold with boundary in a manifold with $\rd^2$ smooth.

\begin{thm}
\label{techmain2}
Let $X$ be a compact $n$-dimensional manifold with boundary and let $\rd$ be a distance function such that $\rd^2$ is regular at the diagonal (see Definition \ref{regularlaododal}). The family of operators
$$Q_X(R): \dot{H}^{-\frac{n+1}{2}}(X)\to \bar{H}^{\frac{n+1}{2}}(X),\; R\in \C\setminus \{0\},$$
is a holomorphic family of Fredholm operators that are invertible on a sector. The inverse $Q_X(R)^{-1}: \bar{H}^{\frac{n+1}{2}}(X)\to \dot{H}^{-\frac{n+1}{2}}(X)$, $R\in \C\setminus \{0\}$, is a meromorphic family of Fredholm operators. 

If $\rd$ satisfies property (SMR) on a sector $\Gamma$, then also 
$$\mathcal{Z}_X(R): \dot{H}^{-\frac{n+1}{2}}(X)\to \bar{H}^{\frac{n+1}{2}}(X),\; R\in \Gamma,$$
is a holomorphic family of Fredholm operators  whose inverse family $\mathcal{Z}_X(R)^{-1}: \bar{H}^{\frac{n+1}{2}}(X)\to \dot{H}^{-\frac{n+1}{2}}(X)$, $R\in \Gamma$,
is a meromorphic family of Fredholm operators. 
\end{thm}

For the purposes described above, we are interested in precise asymptotic information about solutions to $\mathcal{Z}_X(R)u=f$. We describe the inverse operator $\mathcal{Z}_X^{-1}$ via Wiener-Hopf factorizations.

\begin{thm}
\label{techmain3}
Let $X$ be a compact $n$-dimensional manifold with boundary and $\rd$ a distance function whose square is regular at the diagonal. For some $R_0\geq 0$ and any $R\in \Gamma_{\pi/(n+1)}(R_0)$, we can write 
$$Q_X^{-1}=\chi_1A\chi_1'+\chi_2 (\varphi^{-1})^*W_+W_-\varphi^*\chi_2'+S,$$
where { $A$ is a pseudodifferential parametrix of $Q_X$,} $\chi_1,\chi_1'\in C^\infty_c(X^\circ)$, and $\chi_2,\chi_2'\in C^\infty(X)$ are functions supported in a collar neighborhood $U_0$ of $\partial X$ in $X$ such that 
$$\chi_1+\chi_2=1\quad\mbox{and}\quad \chi_j'|_{\mathrm{supp}(\chi_j)}=1, \; j=1,2,$$ 
$\varphi:\partial X\times [0,1)\to U_0$ is a collar identification, and the operators $S$, $W_-$ and $W_+$ satisfying the following as $R\to \infty$:
\begin{enumerate}
\item $S:\overline{H}^{\mu}(X)\to \dot{H}^{-\mu}(X)$ is a continuous operator with $\|S\|_{\overline{H}^{\mu}(X)\to \dot{H}^{-\mu}(X)}=O(R^{-\infty})$. 
\item $W_+:L^2(\partial X\times [0,\infty))\to \dot{H}^{-\mu}(\partial X\times [0,\infty))$ is the properly supported pseudodifferential operator of mixed-regularity $(\mu,0)$ from Definition \ref{thewopsdef} which is invertible for large $R>0$ and in local coordinates has an asymptotic expansion modulo $S^{\mu,-\infty}$ as in Lemma \ref{thewops} and preserves support in $\partial X\times [0,\infty)\subseteq \partial X\times \R$. Moreover, for $\chi,\chi'\in \C+C^\infty_c(\partial X\times [0,\infty))$ with $\chi \chi'=0$, it holds that $\|\chi W_+ \chi'\|_{L^2(\partial X\times [0,\infty))\to H^{-\mu}(\partial X\times \R)}=O(R^{-\infty})$.
\item $W_-:\overline{H}^{\mu}(\partial X\times [0,\infty))\to L^2(\partial X\times [0,\infty))$ is the properly supported pseudodifferential operator of mixed-regularity $(\mu,0)$ from Definition \ref{thewopsdef} which is invertible for large $R>0$ and in local coordinates has an asymptotic expansion modulo $S^{\mu,-\infty}$ as in Lemma \ref{thewops} and preserves support in $\partial X\times (-\infty,0]\subseteq \partial X\times \R$. Moreover, for $\chi,\chi'\in \C+C^\infty_c(\partial X\times \R)$ with $\chi \chi'=0$, it holds that $\|\chi W_- \chi'\|_{H^{\mu}(\partial X\times \R)\to L^2(\partial X\times \R)}=O(R^{-\infty})$.
\end{enumerate}
\end{thm}

{ Theorem \ref{techmain3} is} stated as Theorem \ref{resolventstructure} in the body of the text. A key feature of the construction in Theorem \ref{techmain3} is that it provides us with a method to compute the symbolic structure of the inverse $Q_X^{-1}$.

\begin{thm}
\label{techmain4}
Let $X$ be a compact $n$-dimensional manifold with boundary and $\rd$ a distance function whose square is regular at the diagonal. In the sector $\mathrm{Re}(R)>R_0$ and $\mathrm{arg}(R)<\pi/(n+1)$, we { have a complete asymptotic expansion
$$\langle Q(R)^{-1}1_X,1_X\rangle\sim \sum_{k=0}^\infty c_k(X,\rd)R^{n-k}+O(\mathrm{Re}(R)^{-\infty}), \quad \mbox{as $\mathrm{Re}(R)\to \infty$ in $\Gamma$},$$}
where the coefficients $c_k(X,\rd)$ are given as 
$$c_k(X,\rd)=\int_Xa_{k,0}(x,1)\rd x+\int_{\partial X}B_{\rd^2,k}(x)\rd x',$$
where 
\begin{enumerate}
\item $a_{k,0}(\cdot,1)\in C^\infty(X)$ is an invariant polynomial in the entries of the Taylor expansion \eqref{talaldadladldaladlda} as described in Theorem \ref{evaluationsofinterioraxizero} and can be computed inductively using Lemma \ref{evaluationsofinterioraxizeroind}, with $a_{k,0}=0$ if $k$ is odd; and
\item $B_{\rd^2,k}\in C^\infty(\partial X)$ is an invariant polynomial in the entries of the Taylor coefficients of $\rd^2$ at the diagonal in $X$ near $\partial X$ as described in Proposition \ref{evaluationsbbbbbbb} and can be inductively computed using Lemma \ref{wsymdexcpsps}.
\end{enumerate} 
In particular, we have that 
\begin{align*}
c_0(X,\rd)=&\frac{\mathrm{vol}(X)}{n!\omega_n},\quad c_{1}(X,\rd)=\frac{(n+1) \mathrm{vol}(\partial X)}{2n!\omega_n},\\
c_{2}(X,\rd)=&\frac{n+1}{6\cdot n!\omega_n}\int_X s_{\rd^2}\rd x+\frac{(n-1)(n+1)^2}{8\cdot n!\omega_n}\int_{\partial X}H_{\rd^2}\rd x'.
\end{align*}
where the scalar curvature $s_{\rd^2}$ is defined as in Theorem \ref{magcompsclosedeld} and the mean curvature $H_{\rd^2}$ of the distance function is defined as in Theorem \ref{asyofqsexoeod}.
\end{thm}

{ Theorem \ref{techmain4} is }stated as Theorem \ref{asyofqsexoeod} in the body of the text. 

\begin{cor}
Let $X$ be a compact $n$-dimensional manifold with boundary and $\rd$ a distance function whose square is regular at the diagonal. Write $\mathcal{E}(R;X,\rd)$ for the ground state energy.
\begin{enumerate}
\item If $\rd$ has property (MR) on a sector $\Gamma$, the ground state energy function $\mathcal{E}(R;X,\rd)$ is a well defined meromorphic function of $R\in \Gamma$.
\item If $\rd$ has property (SMR) on a sector $\Gamma$,  the ground state energy function $\mathcal{E}(R;X,\rd)$ has { the complete semiclassical asymptotic expansion
$$\mathcal{E}(R;X,\rd)\sim \sum_{k=0}^\infty \epsilon_k(X,\rd)R^{-n-k}+O(\mathrm{Re}(R)^{-\infty}), \quad \mbox{as $\mathrm{Re}(R)\to \infty$ in $\Gamma$},$$}
where 
\begin{align*}
\epsilon_0(X,\rd)=&\frac{n!\omega_n}{\mathrm{vol}_n(X)},\quad\epsilon_1(X,\rd)=-\frac{(n+1)\mathrm{vol}_{n-1}(X)}{2\mathrm{vol}_n(X)}\\
\epsilon_2(X,\rd)=&\frac{(n+1)^2\mathrm{vol}_{n-1}(X)^2}{4\mathrm{vol}_n(X)^2}-\frac{n!\omega_nc_2(X,\rd)}{\mathrm{vol}_n(X)},
\end{align*}
and more generally 
$$\epsilon_k(X,\rd)=p_{k,n}\left( \frac{c_1(X,\rd)}{\mathrm{vol}_n(X)}, \ldots,  \frac{c_k(X,\rd)}{\mathrm{vol}_n(X)}\right),$$ 
for a universal polynomial $p_{k,n}$ of total degree $k$ (where each  $c_j(X,\rd)$ is declared to be of degree $j$).
\end{enumerate}

\end{cor}

\subsection{Notational conventions}

To avoid confusion, we will use the terms Riemannian metrics and distance function to separate the notions of metrics that appear in Riemannian geometry and metric geometry, respectively.

The Fourier transform on $\R^n$ is denoted by $\mathcal{F}$, where we use the convention 
$$\mathcal{F}f(\xi) = \int_{\R^n} e^{-i x\cdot \xi} f(x) \rd x$$ for $f \in \mathcal{S}(\R^n)$. We further write
$$D_x=-i\frac{\partial}{\partial x}.$$
For $\alpha=(\alpha_1,\ldots,\alpha_n)\in \N^n$, we write $|\alpha|=\sum_j \alpha_j$, $D^\alpha_x=D_{x_1}^{\alpha_1}\cdots D_{x_n}^{\alpha_n}$ and $x^\alpha=x_1^{\alpha_1}\cdots x_n^{\alpha_n}$. In this convention, for $f \in \mathcal{S}(\R^n)$,
$$\mathcal{F}(D_x^\alpha f)(\xi)=\xi^\alpha\mathcal{F}f(\xi)\quad\mbox{and}\quad \mathcal{F}(x^\alpha f)(\xi)=(-D_\xi)^\alpha\mathcal{F}f(\xi),$$
and the product of pseudodifferential symbols is up to smoothing operators defined from a symbol of the form 
$$p\#q\sim\sum_\alpha \frac{1}{\alpha !} \partial_\xi^\alpha pD_x^\alpha q=\sum_\alpha \frac{1}{\alpha !} D_\xi^\alpha p\partial_x^\alpha q,$$ 
where $\alpha! = \prod_{j=1}^n \alpha_j!$.

We write $M$ for a manifold and $X$ for a compact manifold with boundary, or occasionally a general compact metric space. We let $n$ denote the dimension of $M$ or $X$ and use the notation 
$$\textstyle{\mu:=\frac{n+1}{2}}.$$
We write $\mathrm{Diag}_M:=\{(x,x): x\in M\}$ for the diagonal in $M\times M$. If $(X,\rd)$ is a compact metric space such that the matrix $(\e^{-R\rd(x,y)})_{x,y\in F}$ is positive definite for any finite subset $F\subseteq X$, we say that $(X,\rd)$ is positive definite. If $(X,R\rd)$ is positive definite for all $R>0$, we say that $(X,\rd)$ is stably positive definite. This terminology follows \cite{meckes1}. 

For a manifold $M$, we write $C^\infty_c(M)$ for the Fr\'{e}chet space of compactly supported smooth functions and $\mathcal{D}'(M)$ for its topological dual -- the distributions on $M$. { A domain $X\subseteq M$ is a subset which coincides with the closure of its interior points}. If $X$ is a compact manifold with boundary, it can always be embedded as a domain in a manifold $M$ and we write $C^\infty(X)\subseteq C(X)$ for the restrictions to $X$ of elements in $C^\infty(M)$.

A sector $\Gamma\subseteq \C$ is a conical subset, i.e., $\lambda \Gamma \subset \Gamma$ for all $\lambda > 1$. Standard examples we use throughout the paper are $\C_+=\{z\in \C: \mathrm{Re}(z)>0\}$ and 
$$\Gamma_\alpha(R_0):=\{z\in \C: |\mathrm{Arg}(z)|<\alpha, \mathrm{Re}(z)>R_0\}.$$ 
If we make a claim concerning $R\to +\infty$, it is implicitly assumed to be a limit along the real line. We also note that for sectors $\Gamma\subseteq \C_+$ of opening angle $\alpha<\pi/2$, there is a $C_\alpha>0$ with $C_\alpha^{-1}|R|\leq \mathrm{Re}(R)\leq C_\alpha |R|$. 

For two Banach spaces $V_1$ and $V_2$, we write $\mathbb{B}(V_1,V_2)$ for the space of bounded operators $V_1\to V_2$ and $\mathbb{K}(V_1,V_2)$ for the space of compact operators $V_1\to V_2$. Both form Banach spaces in the norm topology.

We write $\N=\{0,1,2,3,\ldots\}$ for the set of natural numbers. The closed positive, respectively negative half-spaces are denoted by $\R^n_\pm = \{x \in \R^n : \pm x_n \geq 0 \}$.

\subsection{Acknowledgments}
We are most grateful to the referees for their careful reading that helped us improve the paper. We thank Tony Carbery, Daniel Grieser, Gerd Grubb, Tom Leinster, Rafe Mazzeo, Mark Meckes, Niels Martin M\o ller, Grigori Rozenblum, Jan-Philip Solovej and Simon Willerton for fruitful and encouraging discussions.

MG was supported by the Swedish Research Council Grant VR 2018-0350, NL by The Maxwell Institute Graduate School in Analysis and its Applications, a Centre for Doctoral Training funded by the UK Engineering and Physical Sciences Research Council (Grant EP/L016508/01), the Scottish Funding Council, Heriot-Watt University and the University of Edinburgh.

\section{The symbolic structure near the diagonal}
\label{seconsymbolofq}

To better understand the operator $\mathcal{Z}$ we first consider the case of a manifold $M$. This analysis describes compact manifolds (see Subsection \ref{qoncomapalala} and Section \ref{remadinedlsosec}) and we carry it over to the interior of a compact manifold with boundary below in Section \ref{qxondomains} and \ref{structurofinversesec}. We formulate our results in terms of the operator 
\begin{equation}
\label{definedindoz}
\mathcal{Z}(R)f(x):=\frac{1}{R}\int_M \e^{-R\rd(x,y)}f(y)\rd y,
\end{equation}
whose dependence on $R\neq 0$ is holomorphic under suitable assumptions studied in Section \ref{remadinedlsosec} below. Here we are implicitly using a volume density on $M$, and the operator depends on this choice. We shall later fix a certain choice adapted to the distance function. We shall specify the domain and codomain of this operator more precisely later on. For now, we can consider $\mathcal{Z}$ an operator $C^\infty_c(M)\to\mathcal{D}'(M)$ by setting 
$$\langle \mathcal{Z}\varphi,\psi\rangle=\frac{1}{R}\int_{M\times M} \e^{-R\rd(x,y)}\varphi(y)\psi(x)\rd x\rd y, \quad\mbox{for}\quad \varphi,\psi \in C^\infty_c(M).$$

\subsection{On a class of pseudodifferential operators with parameter}

We pick a function $\chi\in C^\infty(M\times M)$ such that $\chi=1$ near $\mathrm{Diag}_M$ and is supported in a small neighborhood of $\mathrm{Diag}_M$. The precise choice of $\chi$ will not play an important role, but we shall later specify conditions on its support. The operator $\mathcal{Z}$ decomposes as 
\begin{equation}
\label{definiqododkq}
\mathcal{Z}=Q+L, \quad \mbox{where}\quad Q(R)f(x):=\frac{1}{R}\int_M \chi(x,y)\e^{-R\rd(x,y)}f(y)\rd y.
\end{equation}
We call the operator $Q$ the \emph{localization of $\mathcal{Z}$ near the diagonal}. In this section we focus our attention to $Q$. Distance functions might be non-smooth away from the diagonal despite being quite regular at the diagonal and this off-diagonal behavior of the distance function dictates whether or not $L$ is negligible. The remainder $L$ will be studied further in Section \ref{remadinedlsosec}.

This subsection is devoted to proving that the localized part $Q$ of $\mathcal{Z}$ is a parameter-dependent pseudodifferential operator, see Appendix \ref{subsec:ftofcodnsos}. We will treat a slightly more general class of operators than $Q$. Consider a family of operators that take the form
\begin{equation}
\label{arforgo}
Q_{G,\chi}(R)f(x):=\frac{1}{|R|_{\rm Mc}}\int_M\chi(x,y)\e^{-|R|_{\rm Mc}\sqrt{G(x,y)}}f(y)\mathrm{d}V(y).
\end{equation}
Here we have written 
$$|R|_{\rm Mc}:=\begin{cases}
R, \; &\mathrm{Re}(R)>0,\\
-R, \; &\mathrm{Re}(R)<0,\end{cases}$$
for the McIntosh modulus which extends the absolute value to a holomorphic function in $\C\setminus i\R$. We shall mainly be concerned with the cases $R\in \R$ and $R\in \C_+$. The cut-off function $\chi$ is as above with the additional constraint that $G$ is smooth on $\mathrm{supp}(\chi)$. The function $G:M\times M\to [0,\infty)$ should be \emph{regular at the diagonal} as made precise in Definition \ref{regularlaododal} below: to define this notion we first introduce further terminology. Note that $T(M\times M)=p_1^*TM\oplus p_2^*TM$ where $p_j:M\times M\to M$, $j=1,2$, denotes the projection onto the $j$:th factor. Over the diagonal $\mathrm{Diag}_M$, the map $Dp_1\oplus Dp_2:T(M\times M)|_{\mathrm{Diag}_M}\to TM\oplus TM$ is an isomorphism. We define the transversal tangent bundle to the diagonal to be 
$$T_{\rm tra}\mathrm{Diag}_M:=\ker (Dp_1+ Dp_2:T(M\times M)|_{\mathrm{Diag}_M}\to TM)\subseteq T(M\times M)|_{\mathrm{Diag}_M}.$$
The restriction $Dp_1|:T_{\rm tra}\mathrm{Diag}_M\to TM$ is an isomorphism. 

\begin{deef}
\label{transversehessdefi}
For a smooth function $G$ defined in a neighborhood of $\mathrm{Diag}_M$, we define the transversal Hessian $H_{G}$ as the quadratic form on $T_{\rm tra}\mathrm{Diag}_M$ obtained from restricting the Hessian of $G$ over the diagonal to the transversal tangent bundle. 
\end{deef}

\begin{deef}
\label{regularlaododal}
A function $G:M\times M\to [0,\infty)$ is said to be \emph{regular at the diagonal} if there is a tubular neighborhood $U$ of the diagonal $\mathrm{Diag}_M$ such that $G|_U\in C^\infty(U)$ is a smooth function satisfying that
\begin{itemize}
\item $G|_U$ and $\rd G|_{U}$ vanish on $\mathrm{Diag}_M\subseteq U$; 
\item $G(x)> 0$ for $x\in U\setminus \mathrm{Diag}_M$; and
\item the transversal Hessian $H_{G}$ is positive definite in all points of $\mathrm{Diag}_M$.
\end{itemize}
\end{deef}

\begin{remark}
The prototypical example of a function $G$ regular at the diagonal is $G=\rd^2$ for suitable distance functions $\rd$. The function $\rd^2$ is regular at the diagonal for the Euclidean distance, or when $\rd$ is the geodesic distance on a Riemannian manifold (see Example \ref{geodesciexamokad1} below) or more generally a distance function induced from pulling back a geodesic distance along an embedding of $M$ (see Example \ref{subsmandiodladexu} below for an example).
\end{remark}

To show that $Q_{G,\chi}(R)$ is an elliptic pseudodifferential operator with parameter, and to describe its full symbol, we shall use a slight detour. The basic idea used in computing the full symbol of $Q_{G,\chi}(R)$ is to do an inverse Fourier transform in $R$, and then Fourier transform all conormal variables. When we Fourier transform in the $R$-variable the Schwartz kernel -- depending on $(x,y,R)$ -- transforms to a conormal distribution on $U\times \R$ (conormal to $\mathrm{Diag}_M\times \{0\}$) -- depending on $(x,y,\eta)$ -- that we then Fourier transform in all the transversal directions $(v,\eta)$ where $v=x-y$, thus producing the full symbol depending on $(x,\xi,R)$. To compute the Fourier transform in the $R$-direction, we use the following elementary lemma.

\begin{prop}
\label{ftprop}
For a parameter $a\geq 0$, and $F_a(R):=\mathrm{F.P.}\frac{\e^{-a|R|}}{|R|}$, we have that 
$$\Fg F_a(\eta)= -\log(\eta^2+a^2)+\log(2)-2\gamma,$$
where $\gamma$ is the Euler-Mascheroni constant.
\end{prop}

\begin{proof}
By taking a derivative in the parameter $a$ and using that the Fourier transform of $R\mapsto\e^{-a|R|}$ is $2a(\eta^2+a^2)^{-1}$, we see that 
$$\frac{\partial }{\partial a}\Fg F_a(\eta)=-2a(\eta^2+a^2)^{-1}.$$
As such, $\Fg F_a(\eta)=-\log(\eta^2+a^2)+c_0(\eta)$ for some tempered distribution $c_0$. Setting $a=0$ and using Proposition \ref{ftoffpa}, we see that $c_0$ is a constant. By Proposition \ref{ftoffpa} we have that
$$c_0=\beta_{0,1}=2\log(2)+\frac{1}{2}\psi(1/2)-\gamma=\log(2)-2\gamma.$$
\end{proof}

\begin{prop}
\label{canonicalmetrix}
Let $G:M\times M\to [0,\infty)$ be a function which is regular at the diagonal, see Definition \ref{regularlaododal}. For a neighborhood $U$ of $\mathrm{Diag}_M$ on which $G$ is smooth, define $\tilde{G}\in C^\infty(U\times \R)$ by 
$$\tilde{G}(x,y,\eta):=\eta^2+G(x,y),$$
and the conormal distribution $\log(\tilde{G})\in I^{-n-1}(U\times \R; \mathrm{Diag}_M\times \{0\})$ as in Proposition \ref{oneonfandlog}. Then there exists a canonical metric $g_G$ on $T^*M$ such that 
$$\sigma_{-n-1}(\log(\tilde{G}))(x,\xi,R)=-2\pi n!\omega_n(R^2+g_G(\xi,\xi))^{-(n+1)/2}.$$
The canonical metric $g_G$ is dual to the transversal Hessian of $G$ under the isomorphism $Dp_1|:T_{\rm tra}\mathrm{Diag}_M\to TM$.
\end{prop}
In the following, unless specified otherwise we shall consider $M$ endowed with the Riemannian metric $g_G$. This allows, in particular, to define a Laplace operator on $M$.  
\begin{proof}
Since $G$ is regular at the diagonal, the function $\tilde{G}$ satisfies the assumptions of Proposition \ref{oneonfandlog} and the result follows therefrom. 
\end{proof}

To compute the full symbol of $Q_{G,\chi}$ we use a Taylor expansion in the direction transversal to the diagonal. Consider a function $G:M\times M\to [0,\infty)$ which is regular at the diagonal, see Definition \ref{regularlaododal}. Consider a coordinate chart $U_0\subseteq M$. The coordinates on $U_0$ induce coordinates $(x,y)$ on $U_0\times U_0$ and we can identify a neighborhood of $\mathrm{Diag}_M\cap (U_0\times U_0)$ with a neighborhood of the zero section in $T_{\rm tra}\mathrm{Diag}_M|_{U_0}$ via the map $(x,y)\mapsto (x,x-y)$. If the coordinate chart $U_0$ on $M$ satisfies that $G$ is smooth on $U_0\times U_0$, Taylor's theorem implies that for any $N\in \N$ we can on $U_0\times U_0$ write 
\begin{equation}
\label{taylorexpamdmd}
G(x,y)=H_{G,x}(x-y)+\sum_{j=3}^NC_{G}^{(j)}(x;x-y)+r_N(x,x-y),
\end{equation}
for $|x-y|$ small enough, where $r_N$ is a smooth function with $r_N(x,v)=O(|v|^{N+1})$ as $v\to 0$, $H_{G}$ is the transversal Hessian of $G$, and $C_{G}^{(j)}:U_0\to \mathrm{Sym}^j(T_{\rm tra}\mathrm{Diag}_M|_{U_0})$ takes values in the symmetric $j$-forms on the transversal tangent bundle $T_{\rm tra}\mathrm{Diag}_M|_{U_0}$. A short computation shows that $H_G$ indeed is a Riemannian metric on $M$ under the isomorphism $Dp_1|:T_{\rm tra}\mathrm{Diag}_M\to TM$. However, each $C_{G}^{(j)}$ depends on the choice of coordinates, we nevertheless suppress this dependence in the notation.

Since there is a canonical isomorphism $T_{\rm tra}\mathrm{Diag}_M|_{U_0}\cong TM|_{U_0}$, the symmetric $j$-form $C_{G}^{(j)}:U_0\to \mathrm{Sym}^j(T_{\rm tra}\mathrm{Diag}_M|_{U_0})$ appearing in the Taylor expansion \eqref{taylorexpamdmd} of $G$ defines a $j$:th order differential operator
$$C^{(j)}_{G}(x,-D_\xi):C^\infty(T^*M|_{U_0})\to C^\infty(T^*M|_{U_0}),$$
obtained by quantizing the coordinate functions, i.e. $C^{(j)}_{G}(x,-D_\xi)$ acts as multiplication operators by $C^{(j)}_{G}(x,v)$ under the fiberwise inverse Fourier transform (in the $v$-direction). For a $k\in \N_+$ and a multiindex $\gamma\in \N^k_{\geq 3}$, we can define a differential operator $C^{(\gamma)}_{G}(x,-D_\xi)$ on $T^*M|_{U_0}$ by 
$$C^{(\gamma)}_{G}(x,-D_\xi):=\prod_{l=1}^kC^{(\gamma_l)}_{G}(x,-D_\xi).$$
Since each $C^{(\gamma_l)}_{G}(x,-D_\xi)$ acts as multiplication operators under the inverse Fourier transform, the differential operators $C^{(\gamma_l)}_{G}(x,-D_\xi)$, $l=1,\ldots, k$, commute. The order of $C^{(\gamma)}_{G}(x,-D_\xi)$ is $|\gamma|:=\sum_{l=1}^k\gamma_l$. For $j\in \N$, define the finite set
$$I_j:=\{\gamma\in \cup_{k=1}^\infty \N^k_{\geq 3}: |\gamma|=j+2k\}.$$
For instance, we have that 
$$I_1=\{3\}, \; I_2=\{(3,3),4\}, \; \mbox{and}\; I_3=\{(3,3,3),(4,3),(3,4),5\}.$$
The role of $I_j$ will become clear in Theorem \ref{firstfirstofz} below describing the full symbol of $Q_{G,\chi}$ from Equation \eqref{arforgo} in a coordinate chart. For $\gamma\in \N^k$, we set $\mathrm{rk}(\gamma):=k$. In other words, $\gamma\in \cup_{k=1}^\infty \N^k_{\geq 3}$ belongs to $I_j$ if and only if $j=|\gamma|-2\mathrm{rk}(\gamma)$. We remark that $|\gamma|\geq 3$ and $\mathrm{rk}(\gamma)>0$ is implicit for $\gamma \in I_j$ since $I_j\subseteq \cup_{k=1}^\infty \N^k_{\geq 3}$. The number of elements in $I_j\cap \N^k_{\geq 3}$ is the same as the number of ways to write $j-k$ as a sum of $k$ natural numbers, and so 
$$\#(I_j\cap \N^k_{\geq 3})=\begin{pmatrix} j-1\\ k-1\end{pmatrix}.$$
The following properties of $I_j$ follows.

\begin{prop}
\label{maxesforidkd}
Let $j>0$. The set $I_j\subseteq \cup_{k>0}\N^k_{\geq 3}$ satisfies the following 
\begin{itemize}
\item $\max\{|\gamma|: \gamma\in I_j\}=3j$ and is attained at $\gamma=\vec{3}\in \N^j$.
\item $\max\{\gamma_i: \gamma\in I_j\}=j+2$ and is attained at $\gamma=j+2\in \N^1$.
\item $\#I_j= 2^{j-1}$.
\end{itemize}
\end{prop}

For notational purposes, we introduce the following notation.

\begin{deef}
For an integer {$n\in \N \cup -2\N -1$}, we introduce the notation 
$$\omega_n:=\frac{\pi^{n/2}}{\Gamma\left(\frac{n}{2}+1\right)}.$$
If $n>0$, then $\omega_n$ is the volume of the unit ball in $n$-dimensions.
\end{deef}

To simplify the computations in the subsequent theorem, we note the following relations for the $\Gamma$-function. 

\begin{prop}
\label{agammacompallad}
For natural numbers $n,k\in \N$ {such that $n>2k+1$}, we have that 
$$\Gamma\left(\frac{n+1}{2}-k\right)\Gamma\left(\frac{n-2k}{2}+1\right)=\sqrt{\pi}\frac{\Gamma(n-2k+1)}{2^{n-2k}} {=\sqrt{\pi}\frac{(n-2k)!}{2^{n-2k}}}.$$
Moreover, we have the identities
\begin{align*} 
\frac{(-1)^{k+1}}{k!}\pi^{(n-1)/2}2^{n-2k}\Gamma\left(\frac{n+1}{2}-k\right)=&(-1)^{k+1}(n-2k)!\omega_{n-2k}\omega_{2k}, \; &\mbox{{for $2k<n$}}\\
\frac{(-1)^{k+1}}{k!}\pi^{(n-1)/2}2^{n-2k}\Gamma\left(\frac{n+1}{2}-k\right)=&\frac{(-1)^{n/2+1} \omega_{2k}}{(2k-n)!\omega_{2k-n}}= &\mbox{{for $2k- n \in 2\N$}}\\
=&\frac{(-1)^{n/2+1}(2k-n+1)}{2(2\pi)^{2k-n}}\omega_{2k-n+1}\omega_{2k}, \\
\frac{(-1)^{\frac{n+1}{2}}\pi^{(n-1)/2}}{2^{2k-n}\left(k-\frac{n+1}{2}\right)!k!}=&\frac{(-1)^{\frac{n+1}{2}}}{(2\pi)^{2k-n}}\omega_{2k}\omega_{2k-n-1} , \; &\mbox{for $2k- n\in 2\N+1$}
\end{align*}
\end{prop}

\begin{proof}
The Legendre duplication formula $\Gamma(\zeta)\Gamma(\zeta+1/2)=2^{1-2\zeta}\sqrt{\pi}\Gamma(2\zeta)$ applied to $\zeta=n-2k+1$ implies the first stated identity, and the second one follows from the identity $\Gamma(1/2-m) = (-4)^mm!/(2m)!$. Combining these identities with the definition of $\omega_n$ produces the first and second identities. The third identity follows from the definition of $\omega_n$.
\end{proof}

We now arrive at the main result of this subsection, describing the full symbol of $Q_{G,\chi}$. The reader should keep in mind that we are primarily interested in the function $G(x,y):=\rd(x,y)^2$ for a distance function $\rd$ such that $\rd^2$ is regular at the diagonal. In this case $Q_{G,\chi}(R)=Q(R)$ is the localization of $\mathcal{Z}$ at the diagonal for $\mathrm{Re}(R)>0$. We therefore formulate our results on the sector $\C_+$, albeit for $Q_{G,\chi}$ they hold in the sector $\C\setminus i\R$.

\begin{thm}
\label{firstfirstofz}
Let $M$ be an $n$-dimensional manifold and $G:M\times M\to [0,\infty)$ a function which is regular at the diagonal. We denote the Riemannian metric on $T^*M$ dual to the transversal Hessian $H_{G}$ by $g_G$, as in Proposition \ref{canonicalmetrix}.  Consider the operator 
$$Q_{G,\chi}(R)f:=\frac{1}{|R|_{\rm Mc}}\int_M\chi(x,y)\e^{-|R|_{\rm Mc}\sqrt{G(x,y)}}f(y)\mathrm{d}y,$$
where $\chi \in C^\infty(M\times M)$ a function with $\chi=1$ near $\mathrm{Diag}_M$ and supported only where $G$ is smooth, and we use the Riemannian volume density defined from $g_G$.

We have that $Q_{G,\chi}\in \Psi^{-n-1}_{\rm cl}(M;\C_+)$ is a classical elliptic pseudodifferential operator with parameter of order $-n-1$  with principal symbol 
$$\overline{\sigma}_{-n-1}(Q_{G,\chi})(x,\xi,R)=n!\omega_n(R^2+g_G(\xi,\xi))^{-(n+1)/2}.$$
In a coordinate chart $U_0$ on $M$, the full symbol $q$ of $Q_{G,\chi}$ has a classical asymptotic expansion $q\sim \sum_{j=0}^\infty q_j$ computed from the Taylor expansion \eqref{taylorexpamdmd} and each $q_j$ is the homogeneous symbol of degree $-n-1-j$ which for $j>0$ and for $n$ odd is given by
\small
\begin{align*}
q_j(x,\xi,R)=&\sum_{\gamma\in I_j, \mathrm{rk}(\gamma)<(n+1)/2} \mathfrak{c}_{\mathrm{rk}(\gamma),n}C_{G}^{(\gamma)}(x,-D_\xi)(R^2+g_G(\xi,\xi))^{-(n+1)/2+\mathrm{rk}(\gamma)}-\\
&-\sum_{\gamma\in I_j, \mathrm{rk}(\gamma)\geq (n+1)/2} \mathfrak{c}_{\mathrm{rk}(\gamma),n}C_{G}^{(\gamma)}(x,-D_\xi)\left[(R^2+g_G(\xi,\xi))^{-(n+1)/2+\mathrm{rk}(\gamma)}\log(R^2+g_G(\xi,\xi))\right],
\end{align*}
\normalsize
and for $n$ even given as
\begin{align*}
q_j(x,\xi,R)=&\sum_{\gamma\in I_j} \mathfrak{c}_{\mathrm{rk}(\gamma),n}C_{G}^{(\gamma)}(x,-D_\xi)(R^2+g_G(\xi,\xi))^{-(n+1)/2+\mathrm{rk}(\gamma)}
\end{align*}
Here the coefficients are computed as
{ $$\mathfrak{c}_{k,n}:= 
\begin{cases}
(-1)^k(n-2k)!\omega_{n-2k}\omega_{2k}, \; &\mbox{for $2k<n$}\\
\frac{(-1)^{1-n/2} \omega_{2k}}{(2k-n)!\omega_{2k-n}}, \; & \mbox{for $2k-n\in 2\N$}\\
\frac{(-1)^{\frac{n+1}{2}}}{(2\pi)^{2k-n}}\omega_{2k}\omega_{2k-n-1} , \; &\mbox{for $2k- n\in 2\N+1$}
\end{cases}$$}
\end{thm}

\begin{remark}
Exact expressions for $q_1$ and $q_2$ are given below in Proposition \ref{computq1} and \ref{computq2}, respectively.
\end{remark}

\begin{remark}
From the expression for $q_j$ in Theorem \ref{firstfirstofz}, it is not clear that $q_j$ is homogeneous of degree $-n-1-j$ when $n$ is odd. We shall see in the proof that this is in fact the case.
\end{remark}

\begin{proof}[Proof of Theorem \ref{firstfirstofz}]
By Proposition \ref{ftprop}, the Schwartz kernel $\frac{1}{|R|}\chi(x,y)\e^{-|R|\sqrt{G(x,y)}}$ of $Q_{G,\chi}$ is the Fourier transform in the $\eta$-direction of 
$$K(x,y,\eta)=-\frac{1}{2\pi}\chi(x,y)\log(\eta^2+G(x,y)).$$
The extra $2\pi$ is coming from Fourier inversion in one dimension. 

Let $U$ denote a neighborhood of the diagonal $\mathrm{Diag}_M$ on which $G$ is smooth. It follows from Proposition \ref{canonicalmetrix} (cf. Proposition \ref{oneonfandlog}) that $K\in I^{-n-1}(U\times \R;\mathrm{Diag}_M\times \{0\})$ with principal symbol 
$$\overline{\sigma}_{-n-1}(K)(x,\xi,R)=n!\omega_n(R^2+g_G(\xi,\xi))^{-(n+1)/2}.$$
Define $K_0(x,y,\eta):=\log(\eta^2+G(x,y))\in I^{-n-1}(U\times \R;\mathrm{Diag}_M\times \{0\})$. We compute in a coordinate chart $U_0$ that $K_0\in CI^{-n-1}(U\times \R;\mathrm{Diag}_M\times \{0\})$ and using a uniform asymptotic expansion we use  Proposition \ref{compsuppfotkr} to show that the Fourier transform in the $\eta$-direction of $K$ is the Schwartz kernel of a pseudodifferential operator with parameter.

In a coordinate chart $U_0$, we introduce the coordinates $(x,v)=(x,x-y)$ on $U_0\times U_0$. Using Equation \eqref{taylorexpamdmd}, we can write 
$$K_0(x,y,\eta)=-\log(\eta^2+H_{G}(v,v))-\log\left(1+\frac{\sum_{j=3}^NC_{G}^{(j)}(v,v)+r_N(x,v)}{\eta^2+H_{G}(v,v)}\right).$$

For small $v$, we can Taylor expand 
$$K_0(x,y,\eta)=-\log(\eta^2+H_{G}(v,v))-\sum_{j=1}^N\sum_{\gamma\in I_{j}} \frac{(-1)^{\mathrm{rk}(\gamma)}}{\mathrm{rk}(\gamma)}\frac{C^{(\gamma)}_{G}(v)}{(\eta^2+H_{G}(v,v))^{\mathrm{rk}(\gamma)}}+\tilde{r}_N(x,v,\eta).$$
We note that, by the definition of $I_j$, each term in the second sum $\sum_{\gamma\in I_{j}} \frac{(-1)^{\mathrm{rk}(\gamma)}}{\mathrm{rk}(\gamma)}\frac{C^{(\gamma)}_{G}(v)}{(\eta^2+H_{G}(v,v))^{\mathrm{rk}(\gamma)}}$ is homogeneous of degree $j$. We also note that $\tilde{r}_N(x,v,\eta)=O((|\eta|+|v|)^{N+1})$ and a short computation gives that $\partial^\alpha_x\partial^\beta_v\partial_\eta^k\tilde{r}_N=O((|\eta|+|v|)^{N+1-|\beta|-k})$ for any multiindices $\alpha$, $\beta$ and $k$. As such, we have a uniform asymptotic expansion $K\sim \sum_{j=0}^\infty K_j$ (cf. Definition \ref{unidoasissma}) where 
\begin{equation}
\label{fullasexp}
K_j(x,v,\eta)=
\begin{cases}
-\frac{1}{2\pi}\log(\eta^2+H_{G}(v,v)), \quad &j=0,\\
\frac{1}{2\pi}\sum_{\gamma\in I_{j}} \frac{(-1)^{\mathrm{rk}(\gamma)+1}}{\mathrm{rk}(\gamma)}\frac{C^{(\gamma)}_{G}(v)}{(\eta^2+H_{G}(v,v))^{\mathrm{rk}(\gamma)}}, \quad &j>0.
\end{cases}
\end{equation}
We conclude from Proposition \ref{compsuppfotkr} that $Q_{G,\chi}\in \Psi^{-n-1}_{\rm cl}(M;\R)$ is a pseudodifferential operator with parameter. Proposition \ref{compsuppfotkr} implies that 
$$\overline{\sigma}_{-n-1}(Q_{G,\chi})(x,\xi,R)=\overline{\sigma}_{-n-1}(K)(x,\xi,R)=n!\omega_n(R^2+g_G(\xi,\xi))^{-(n+1)/2}.$$
This is invertible in $C^\infty(S(T^*M\oplus \R))$ so $Q_{G,\chi}$ is elliptic with parameter. It is readily seen that $\overline{\sigma}_{-n-1}(Q_{G,\chi})(x,\xi,R)\neq 0$ also for $R\in \C_+$, and the following symbol computation shows that $Q_{G,\chi}\in \Psi^{-n-1}_{\rm cl}(M;\C_+)$ is an elliptic pseudodifferential operator with parameter.

Let us turn to the full symbol of $Q_{G,\chi}$. We will compute it from Equation \eqref{fullasexp} and the Fourier transform computations of Section \ref{subsec:ftofcodnsos} of the appendix. For $k>0$, denote the Fourier transform of $(\eta^2+H_{G}(v,v))^{-k}$ in the $(v,\eta)$-direction by $F_k(x,\xi,R)$, that is
$$F_k(x,\xi,R):=\int_{T_xM\oplus \R} \frac{\e^{-i\xi.v-iR\eta}}{(\eta^2+H_{G}(v,v))^{k}} \mathrm{d}v\mathrm{d}\eta.$$
Using homogeneity and rotational symmetries, $F_k$ can be computed as in Proposition \ref{ftoffpa} to be 
$$\frac{\pi^{(n+1)/2}2^{-2k+n+1}\Gamma\left(\frac{n+1}{2}-k\right)}{(k-1)!}(R^2+g_G(\xi,\xi))^{-(n+1)/2+k},$$
for $2k- n-1\notin 2\N$, and 
\begin{equation}
\label{qgammafoform}
\frac{(-1)^{k-\frac{n+1}{2}}\pi^{(n+1)/2}}{2^{2k-n-1}\left(k-\frac{n+1}{2}\right)!(k-1)!}\left[(R^2+g_G(\xi,\xi))^{-(n+1)/2+k}\left(-\log(R^2+g_G(\xi,\xi))+\beta_{k-(n+1)/2,n+1}\right)\right],
\end{equation}
for $2k- n-1\in 2\N$. 
It follows that for $\gamma\in I_j$, the symbol of the term $C^{(\gamma)}_{G}(v)(\eta^2+H_{G}(v,v))^{-\mathrm{rk}(\gamma)}$ is given by $C_{G}^{(\gamma)}(x,-D_\xi)F_{\mathrm{rk}(\gamma)}(x,\xi,R)$
which is computed to be
$$\frac{\pi^{(n+1)/2}2^{-2\mathrm{rk}(\gamma)+n+1}\Gamma\left(\frac{n+1}{2}-\mathrm{rk}(\gamma)\right)}{(\mathrm{rk}(\gamma)-1)!}C_{G}^{(\gamma)}(x,-D_\xi)(R^2+g_G(\xi,\xi))^{-(n+1)/2+\mathrm{rk}(\gamma)},$$
for $2\mathrm{rk}(\gamma)- n-1\notin 2\N$, and 
\begin{align}
\nonumber 
\frac{(-1)^{\mathrm{rk}(\gamma)-\frac{n+1}{2}}\pi^{(n+1)/2}}{2^{2\mathrm{rk}(\gamma)-n-1}}&\left(\mathrm{rk}(\gamma)-\frac{n+1}{2}\right)!(\mathrm{rk}(\gamma)-1)!\cdot \\
\label{qgammafoformexp}
C_{G}^{(\gamma)}(x,-D_\xi)&\left[(R^2+g_G(\xi,\xi))^{-(n+1)/2+\mathrm{rk}(\gamma)}\left(-\log(R^2+g_G(\xi,\xi))+\beta_{\mathrm{rk}(\gamma)-(n+1)/2,n+1}\right)\right],
\end{align}
for $2\mathrm{rk}(\gamma)- n-1\in 2\N$. We conclude that the symbol $q_j$ of $K_j$, for $j>0$, is (up to the term $\beta_{\mathrm{rk}(\gamma),n+1}$) given by the formula in the statement of the theorem for the pre-factors $\mathfrak{c}_{k,n}$:
\begin{align*}
\mathfrak{c}_{k,n}= &
\begin{cases}
\frac{(-1)^{k+1}}{k!}\pi^{(n-1)/2}2^{n-2k}\Gamma\left(\frac{n+1}{2}-k\right), \; &\mbox{for $2k- n-1\notin 2\N$,}\\
\frac{(-1)^{\frac{n+1}{2}}\pi^{(n-1)/2}}{2^{2k-n}\left(k-\frac{n+1}{2}\right)!k!} , \; &\mbox{for $2k- n-1\in 2\N$.}
\end{cases}
\end{align*}
Therefore $\mathfrak{c}_{k,n}$ takes the form prescribed in the theorem by Proposition \ref{agammacompallad}.

We finish the proof by showing that $\beta_{\mathrm{rk}(\gamma)-(n+1)/2,n+1}$ does not contribute in Equation \eqref{qgammafoformexp} and that there is no logarithmic term when expanding the $\xi$-derivatives in Equation \eqref{qgammafoformexp}. If $2\mathrm{rk}(\gamma)- n-1\in 2\N$, then $-(n+1)/2+\mathrm{rk}(\gamma)\in \N$ and so $(R^2+g_\rd(\xi,\xi))^{-(n+1)/2+\mathrm{rk}(\gamma)}$ is a polynomial of degree $2\mathrm{rk}(\gamma)- n-1\in 2\N$ in $\xi$. For $\gamma\in I_j\subseteq \cup_k \N_{\geq 3}^k$, we have that $j+2\mathrm{rk}(\gamma)=|\gamma|\geq 3\mathrm{rk}(\gamma)$. Therefore, if $2\mathrm{rk}(\gamma)- n-1\in 2\N$ then $C_{\mathrm{d}^2}^{(\gamma)}(x,D_\xi)(R^2+g_\rd(\xi,\xi))^{-(n+1)/2+\mathrm{rk}(\gamma)}$ is a polynomial of degree $2\mathrm{rk}(\gamma)- n-1-|\gamma|\leq -\mathrm{rk}(\gamma)-n-1<0$ and therefore it must be identically zero. In other words, we have $C_{G}^{(\gamma)}(x,-D_\xi)(R^2+g_G(\xi,\xi))^{-(n+1)/2+\mathrm{rk}(\gamma)}=0$ for $2\mathrm{rk}(\gamma)- n-1\in 2\N$. This equality proves that $q_j$ can not contain a term with a logarithmic factor, and as such $q_j$ is homogeneous of degree $-n-1-j$ for all $j$, and  
\begin{align*}
C_{G}^{(\gamma)}&(x,-D_\xi)\left[(R^2+g_G(\xi,\xi))^{-(n+1)/2+\mathrm{rk}(\gamma)}\left(-\log(R^2+g_G(\xi,\xi))+\beta_{\mathrm{rk}(\gamma)-(n+1)/2,n+1}\right)\right]=\\
&=-C_{G}^{(\gamma)}(x,-D_\xi)\left[(R^2+g_G(\xi,\xi))^{-(n+1)/2+\mathrm{rk}(\gamma)}\log(R^2+g_G(\xi,\xi))\right].
\end{align*}
\end{proof}

\begin{remark}
\label{remarkaboutoddn}
The computation that $q_0(x,\xi,R)= n!\omega_n(R^2+g_G(\xi,\xi))^{-(n+1)/2}$ is compatible with known symbol computations in $\R^n$ for $G(x,y)=|x-y|^2$ in which case $q=q_0$ is the full symbol expansion when using Euclidean coordinates. Indeed, this statement follows from the fact that $\mathrm{e}^{-R|v|}$ is the Fourier transform of $n!\omega_n(R^2+|\xi|^2)^{-(n+1)/2}$, see for instance \cite[Equation (3)]{barcarbs}. We remark that the coordinate dependent symbol computations of Theorem \ref{firstfirstofz} will be used also in $\R^n$. The reason for using Theorem \ref{firstfirstofz} in $\R^n$ is that to describe the operator near the boundary of a domain with the Wiener-Hopf factorization techniques of Section \ref{structurofinversesec} below we need to ``straighten out the boundary'', i.e. choose coordinates in which the boundary locally looks like a half-space. We make such computations more precise in Subsection \ref{examplesubsesec} below.
\end{remark}

Let us give some further details in computing the symbols $q_1$ and $q_2$. The precise information contained in $q_1$ and $q_2$ will be used later to compute the first terms in the inverse of $Q$ and the asymptotics of its expectation values in Section \ref{condexpsecalald}. Before entering into the symbol computations of $q_1$ and $q_2$, let us introduce some notation. Since $g$ is a metric on $T^*M$, it can be viewed as a symmetric tensor in $TM\otimes TM$ and for a covector $\xi$, the contraction $\iota_\xi g_G$ takes values in $TM$.  In the coordinate chart, each $C^j_G$ takes values in the symmetric $j$-forms on $TM|_{U_0}$. As such, expressions such as $C^{3}_G(x,g_G\otimes \iota_\xi g_G)$ or $C^{4}_G(x,g_G\otimes g_G)$, for instance, make sense. The reader should be aware that such expressions are individually not coordinate invariant, the transformation rules for these expressions can be deduced either from the Taylor expansion \eqref{taylorexpamdmd} or from the transformation rules for pseudodifferential operators. For computational purposes, we also note that 
$$\rd_\xi g_G(\xi,\xi)=2\iota_\xi g_G.$$
Here $\rd_\xi$ denotes the fiberwise exterior differential and we are implicitly using the canonical identification $T^*(T^*M)=\pi^*TM\oplus \pi^*T^*M$ and that the $TM$-summand is where $\rd_\xi$ maps to. 
A useful tool in formulating the computations is the Pochhammer $k$-symbol. For $x\in \R$, $n\in \N$ and $k\in \Z$, we write
$$(x)_{n,k}:=\underbrace{x(x+k)(x+2k)\cdots (x+(n-1)k)}_{n \;\mbox{\tiny  factors}},$$
with the convention that $(x)_{0,k}=1$ for any $k$ and $(0)_{n,k}=1$ for any $n$ and $k$.

\begin{prop}
\label{computq1}
Let $M$ and $Q_{G,\chi}$ be as in Theorem \ref{firstfirstofz}.
In a coordinate chart on $M$, the term $q_1$ of degree $-n-2$ appearing in the full symbol $q$ of $Q_{G,\chi}$ is given by
\begin{align*}
q_1(x,\xi,R)={-}i\bigg(&6C^{3}_G(x,g_G\otimes \iota_\xi g_G)(R^2+g_G(\xi,\xi))^{-2}-\\
&-8C^{3}_G(x,\iota_\xi g_G\otimes \iota_\xi g_G\otimes \iota_\xi g_G)(R^2+g_G(\xi,\xi))^{-3}\bigg),
\end{align*}
if $n=1$, and if $n>1$ we have that
\begin{align*}
q_1(x,\xi,R)=-i(n^2-1)\mathfrak{c}_{1,n}\bigg(&3C^{3}_G(x,g_G\otimes \iota_\xi g_G)(R^2+g_G(\xi,\xi))^{-(n+1)/2-1}-\\
&-(n+3)C^{3}_G(x,\iota_\xi g_G\otimes \iota_\xi g_G\otimes \iota_\xi g_G)(R^2+g_G(\xi,\xi))^{-(n+1)/2-2}\bigg),
\end{align*}
where $\iota_\xi g_G$ denotes the contraction of the metric $g_G$ on $T^*M$ along the covector $\xi$.
\end{prop}

\begin{proof}
We note that $I_1=\{\gamma\in \cup_{k=1}^\infty \N^k_{\geq 3}: |\gamma|=1+2\mathrm{rk}(\gamma)\}=\{3\}$. For $n=1$, we have $\mathfrak{c}_{1,1}=-1/2$ and since we are in the critical case we compute as follows 
\begin{align*}
q_1(x,\xi,R)=&-i\mathfrak{c}_{1,1}C^3(x,1)\partial_\xi^3\log(R^2+g_G(\xi,\xi))=iC^3(x,1)\partial_\xi^2\frac{\iota_\xi g_G}{R^2+g_G(\xi,\xi)}=\\
=&iC^{3}_G(x,1)\bigg(-\frac{{6}g_G \iota_\xi g_G}{(R^2+g_G(\xi,\xi))^{2}}+\frac{8(\iota_\xi g_G)^3}{(R^2+g_G(\xi,\xi))^{3}}\bigg),
\end{align*}
which proves the case $n=1$.

For $n>1$, we compute that 
\begin{align*}
q_1(x,\xi,R)=i(n+1)(n-1)\mathfrak{c}_{1,n}\bigg(-&3C^{3}_G(x,g_G\otimes \iota_\xi g_G)(R^2+g_G(\xi,\xi))^{-(n+1)/2-1}+\\
&+(n+3)C^{3}_G(x,\iota_\xi g_G\otimes \iota_\xi g_G\otimes \iota_\xi g_G)(R^2+g_G(\xi,\xi))^{-(n+1)/2-2}\bigg).
\end{align*}
\end{proof}

The next proposition follows from a similar computation.

\begin{prop}
\label{computq2}
Let $M$ and $Q_{G,\chi}$ be as in Theorem \ref{firstfirstofz}.
In a coordinate chart on $M$, the term $q_2$ of degree $-n-3$ appearing in the full symbol $q$ of $Q_{G,\chi}$ is given as follows; for $n=1$ we have that 
\begin{align*}
q_2(x,\xi,R)=&\mathfrak{c}_{2,1}C^3(x,1)^2\partial_\xi^6\left((R^2+g_G(\xi,\xi))\log(R^2+g_G(\xi,\xi))\right)+\mathfrak{c}_{1,1}C^4(x,1)\partial_\xi^4\log(R^2+g_G(\xi,\xi))=\\
=& -\frac{C^3(x,1)^2}{16}\left(-\frac{2^3\cdot 15 g_G^3}{(R^2+g_G)^{2}}+\frac{2^4\cdot 90(\iota_\xi g_G)^2g_G^2}{(R^2+g_G)^3} - \frac{2^4\cdot 80(\iota_\xi g_G)^4g_G}{(R^2+g_G)^4}+ \frac{2^6\cdot 24(\iota_\xi g_G)^6}{(R^2+g_G)^5}\right)\\
& \quad \quad \quad -\frac{C^4(x,1)}{2} \left(-\frac{2^2\cdot 3g_G^2}{(R^2+g_G)^2} + \frac{2^3\cdot 12(\iota_\xi g_G)^2 g_G}{(R^2+g_G)^3} - \frac{2^4\cdot 6(\iota_\xi g_G)^4}{(R^2+g_G)^4}\right),
\end{align*}
and for $n=3$, we have that 
\begin{align*}
q_2(x,\xi,R)=&\mathfrak{c}_{2,3}C^3(x,\partial_\xi)^2\left(\log(R^2+g_G(\xi,\xi))\right)+\mathfrak{c}_{1,3}C^4(x,\partial_\xi)(R^2+g_G(\xi,\xi))^{-1}=\\
=& \mathfrak{c}_{2,3}240(C^3_G\otimes C^3_G)(g_G \otimes g_G \otimes g_G)(R^2+g_G(\xi,\xi))^{-3} - \\
&-\mathfrak{c}_{2,3}45\cdot 96C^{3}_{G}(x,g_{G}\otimes \iota_\xi g_{G})^2(R^2+g_G(\xi,\xi))^{-4}+\\
&+\mathfrak{c}_{2,3}120\cdot 96C^{3}_{G}(x,\iota_\xi g_{G}\otimes \iota_\xi g_{G}\otimes \iota_\xi g_{G})C^{3}_{G}(x,g_{G}\otimes \iota_\xi g_{G})(R^2+g_G(\xi,\xi))^{-5}- \\
&- \mathfrak{c}_{2,3}80\cdot 96C^{3}_{G}(x,g_{G}\otimes \iota_\xi g_{G})^2(R^2+g_G(\xi,\xi))^{-6}+\\
&+ \mathfrak{c}_{1,3} 24C^{4}(x,g_G\otimes g_G) (R^2+g_G(\xi,\xi))^{-3}  - \\
& {-\mathfrak{c}_{1,3}288 C^4(x,g_G \otimes \iota_\xi g_G \otimes \iota_\xi g_G) (R^2+g_G(\xi,\xi))^{-4} + }\\
&{+ \mathfrak{c}_{1,3}192C^{4}(x,\iota_\xi g_G \otimes\iota_\xi g_G \otimes\iota_\xi g_G \otimes\iota_\xi g_G )(R^2+g_G(\xi,\xi))^{-5}}
\end{align*}
and finally for $n\neq 1,3$, we have that
\begin{align*}
q_2(x,\xi,R)=&-\mathfrak{c}_{2,n}C^3(x,\partial_\xi)^2(R^2+g_G(\xi,\xi)))^{-(n-3)/2}+\mathfrak{c}_{1,n}C^4(x,\partial_\xi)(R^2+g_G(\xi,\xi))^{-(n-1)/2}=\\
=& {-24\mathfrak{c}_{2,n}(n+3)_{4,-2}C^3(x,g_G\otimes \iota_\xi g_G)^2 (R^2+g_G(\xi,\xi))^{-(n+1)/2-2}+}\\
&{+6\mathfrak{c}_{2,n}(n+5)_{5,-2}C^3(x,\iota_\xi g_G \otimes\iota_\xi g_G\otimes\iota_\xi g_G )C^3(x,g_G \otimes \iota_\xi g_G)(R^2+g_G(\xi,\xi))^{-(n+1)/2-3}-}\\
&{-\mathfrak{c}_{2,n}(n+7)_{6,-2}C^3(x,\iota_\xi g_G \otimes\iota_\xi g_G\otimes\iota_\xi g_G )^2(R^2+g_G(\xi,\xi))^{-(n+1)/2-4}{+}}\\
&+3\mathfrak{c}_{2,n}(n+5)_{5,-2}(C^3 \otimes C^3)(x,g_G \otimes g_G \otimes g_G )(R^2+g_G(\xi,\xi))^{-(n+1)/2-1} +\\
&{ +3\mathfrak{c}_{1,n}(n^2-1) C^4(x,g_G \otimes g_G)(R^2+g_G(\xi,\xi))^{-(n+1)/2-1}-}\\
& - 6\mathfrak{c}_{1,n}(n+3)_{3,-2}C^4(x,g_G \otimes \iota_\xi g_G \otimes \iota_\xi g_G)(R^2+g_G(\xi,\xi))^{-(n+1)/2-2} + \\
& + \mathfrak{c}_{1,n}(n+5)_{4,-2} C^4(x,\iota_\xi g_G\otimes \iota_\xi g_G\otimes\iota_\xi g_G\otimes\iota_\xi g_G)(R^2+g_G(\xi,\xi))^{-(n+1)/2-3},
\end{align*}
where $\iota_\xi g_G$ denotes the contraction of the metric $g_G$ on $T^*M$ along the covector $\xi$.
\end{prop}

\subsection{Examples and further structure in the symbol computations}
\label{examplesubsesec}

In the preceding subsection we saw a detailed computation of the full symbol of $Q_{\chi,G}$ in terms of the Taylor expansion. Let us consider a few important special cases where further structures can be visible in the symbol expansion, and proceed with a structural statement of for the entries in the full symbol in the general case.

\begin{example}[Symbol computations near a boundary in Euclidean space]
\label{euxcomdmoda1}
We consider the manifold $M=\R^n$ and the function $G(x,y)=|x-y|^2$ which is regular at the diagonal. This example fits into the bigger picture of the paper seeing that $G(x,y)=\rd(x,y)^2$. It will later in the paper be crucial to describe the symbol of $Q=Q_{G,\chi}$ near the boundary of a domain $X\subseteq \R^n$ with smooth boundary. Fix a point $x_0\in \partial X$. Up to a rigid motion, we can assume that $x_0=0$ and that the normal vector of $\partial X$ in $x_0$ is orthogonal to the plane $x_n=0$, where we write $x=(x',x_n)$ for $x'\in \R^{n-1}$ and $x_n\in \R$. There is a neighborhood $U_0=U_{00}\times (-\epsilon,\epsilon)$ of $0\in \R^n$ and a smooth function $\varphi\in C^\infty(U_{00})$ such that $X\cap U_0=\{x=(x',x_n)\in U_0: \varphi(x')\leq x_n\}$. Since the normal vector of $\partial X$ in $x_0$ is orthogonal to the plane $x_n=0$, $\nabla_{x'}\varphi(0)=0$. Near $x_0=0$, we use the coordinates
$$(x',x_n)\mapsto (x',x_n-\varphi(x')).$$
In these new coordinates, the domain $X$ looks like the half-space $\{(x',x_n): x_n\geq 0\}$ locally near $x_0=0$. We compute that in these coordinates $G=G(x,y)$ can be written as 
\begin{align*}
|(x',x_n-\varphi(x'))-&(y',y_n-\varphi(y'))|^2=|x'-y'|^2+(x_n-y_n-(\varphi(x')-\varphi(y')))^2=\\
=&|x-y|^2-2(x_n-y_n)(\varphi(x')-\varphi(y'))+(\varphi(x')-\varphi(y'))^2=\\
=&|v|^2+\sum_{j=2}^N\sum_{|\alpha'|=j-1} \frac{-2\partial_{x'}^{\alpha'}\varphi(x')}{\alpha'!} v_n(v')^{\alpha'}+\\
&+\sum_{j=2}^N\sum_{\substack{|\alpha'|+|\beta'|=j,\\ |\alpha'|,|\beta'|>0}} \frac{\partial_{x'}^{\alpha'}\varphi(x')\partial_{x'}^{\beta'}\varphi(x')}{\alpha'!\beta'!} (v')^{\alpha'+\beta'}+O(|v|^{N+1}),
\end{align*}
where the sums over $\alpha'$ and $\beta'$ ranges over $\alpha',\beta'\in \N^{n-1}$ and we have written $v=x-y$, $v'=x'-y'$ and $v_n=x_n-y_n$. Introducing the notation 
$$\nabla^j\varphi(x';v'):=\sum_{|\alpha'|=j} \frac{\partial_{x'}^{\alpha'}\varphi(x')}{\alpha'!} (v')^{\alpha'},$$
we have in these coordinates that 
\begin{align*}
G(x,y)=&|v|^2-2v_n\nabla\varphi(x')\cdot v'+(\nabla\varphi(x')\cdot v')^2+\sum_{j=3}^N-2v_n\nabla^{j-1}\varphi(x';v')\\
&+\sum_{j=3}^N\sum_{\substack{k+l=j,\\ k,l>0}} \nabla^{k}\varphi(x';v')\nabla^{l}\varphi(x';v')+O(|v|^{N+1})
\end{align*}
 In particular, we can conclude that 
$$H_G(v)=|v|^2-2v_n\nabla\varphi(x')\cdot v'+(\nabla\varphi(x')\cdot v')^2,$$
so $H_G$ is represented by the $n\times n$-matrix
$$H_G=\begin{pmatrix}
1_{n-1}+\nabla\varphi(x')\nabla\varphi(x')^T&-\nabla\varphi(x')\\
-\nabla\varphi(x')^T& 1\end{pmatrix}.$$
Therefore, in the same basis we have that 
$$g_G=H_G^{-1}=\begin{pmatrix}
1_{n-1}&\nabla\varphi(x')\\
\nabla\varphi(x')^T& 1+|\nabla\varphi(x')|^2\end{pmatrix}.$$
We note that $H_G|_{T\partial X}$ is the Riemannian metric on $\partial X$ induced from the Euclidean metric on $\R^n$ and the inclusion $\partial X\hookrightarrow\R^n$. We also conclude that 
\begin{equation}
\label{defecomecomcomada}
C^j_G(x;v)=v_nC^{j,1}_G(x;v')+C^{j,0}_G(x;v')
\end{equation}
where 
$$C^{j,1}_G(x;v'):=-2\nabla^{j-1}\varphi(x';v')\quad\mbox{and}\quad C^{j,0}_G(x;v'):=\sum_{k+l=j, k,l>0} \nabla^{k}\varphi(x';v)\nabla^{l}\varphi(x';v).$$
Therefore, in these coordinates near the boundary of a domain in $\R^n$, we have for $\gamma\in I_j$ that 
\begin{align*}
C_{G}^{(\gamma)}(x,-D_\xi)=(-1)^{|\gamma|}\prod_{l=1}^{\mathrm{rk}(\gamma)} \left[\sum_{|\alpha'|=|\gamma_l|-1} \frac{-2\partial_{x'}^{\alpha'}\varphi(x')}{\alpha'!} D_{\xi'}^{\alpha'}D_{\xi_n}+\sum_{\substack{|\alpha'|+|\beta'|=|\gamma_l|,\\ |\alpha'|,|\beta'|>0}} \frac{\partial_{x'}^{\alpha'}\varphi(x')\partial_{x'}^{\beta'}\varphi(x')}{\alpha'!\beta'!} D_{\xi'}^{\alpha'+\beta'}\right].
\end{align*}
This gives a method for computing the homogeneous symbol $q_j$ of degree $-n-1-j$ for any $j$ following Theorem \ref{firstfirstofz}.

The principal symbol is computed as in Theorem \ref{firstfirstofz}. By Proposition \ref{computq1} we compute for $n>1$ that
\begin{align*}
q_1(x,\xi,R)=&-i(n^2-1)\mathfrak{c}_{1,n}\bigg(6g_G(\nabla^2\varphi)g_G(\xi,\nabla\varphi-e_n)(R^2+g_G(\xi,\xi))^{-(n+1)/2-1}-\\
&\qquad\qquad\qquad-2(n+3)\nabla^2\varphi(\iota_\xi g_G,\iota_\xi g_G)g_G(\xi,\nabla\varphi-e_n)(R^2+g_G(\xi,\xi))^{-(n+1)/2-2}\bigg)=\\
=&i(n^2-1)\mathfrak{c}_{1,n}\bigg(6\xi_n g_G(\nabla^2\varphi)(R^2+g_G(\xi,\xi))^{-(n+1)/2-1}-\\
&\qquad\qquad\qquad-2(n+3)\xi_n\nabla^2\varphi((\xi'+\xi_n\nabla\varphi)^{\otimes 2})(R^2+g_G(\xi,\xi))^{-(n+1)/2-2}\bigg),
\end{align*}
By Proposition \ref{computq2} we compute for $n\neq 1,3$ that
\small
\begin{align*}
q_2(x,\xi,R)=& -24\mathfrak{c}_{2,n}(n+3)_{4,-2}4(g_G(\nabla^2\varphi))^2(g_G(\xi,\nabla\varphi-e_n))^2 (R^2+g_G(\xi,\xi))^{-(n+1)/2-2}+\\
&+6\mathfrak{c}_{2,n}(n+5)_{5,-2}4g_G(\nabla^2\varphi)(g_G(\xi,\nabla\varphi-e_n))^2\nabla^2\varphi(\iota_\xi g_G,\iota_\xi g_G)(R^2+g_G(\xi,\xi))^{-(n+1)/2-3}-\\
&-\mathfrak{c}_{2,n}(n+7)_{6,-2}4(\nabla^2\varphi(\iota_\xi g_G,\iota_\xi g_G))^2(g_G(\xi,\nabla\varphi-e_n))^2(R^2+g_G(\xi,\xi))^{-(n+1)/2-4}+\\
&+3\mathfrak{c}_{2,n}(n+5)_{5,-2}\cdot 4(g_G(\nabla^2\varphi))^2g_G(\nabla\varphi-e_n,\nabla\varphi-e_n)(R^2+g_G(\xi,\xi))^{-(n+1)/2-1} +\\
&{ +3\mathfrak{c}_{1,n}(n^2-1) \bigg(2(g_G\otimes g_G)(\nabla^3\varphi\otimes(\nabla\varphi-e_n))+(g_G(\nabla^2\varphi))^2\bigg) (R^2+g_G(\xi,\xi))^{-(n+1)/2-1}-}\\
& - 6\mathfrak{c}_{1,n}(n+3)_{3,-2}\bigg(2(g_G\otimes \iota_\xi g_G)(\nabla^3\varphi)g_G(\xi,\nabla\varphi-e_n)+\\
&\qquad\qquad\qquad\qquad\qquad+(g_G(\nabla^2\varphi))(\nabla^2\varphi(\iota_\xi g_G,\iota_\xi g_G))\bigg) (R^2+g_G(\xi,\xi))^{-(n+1)/2-2} + \\
& + \mathfrak{c}_{1,n}(n+5)_{4,-2} \bigg(2(\iota_\xi g_G)^{\otimes 3}(\nabla^3\varphi)g_G(\xi,\nabla\varphi-e_n)+\\
&\qquad\qquad\qquad\qquad\qquad\qquad+(\nabla^2\varphi(\iota_\xi g_G,\iota_\xi g_G))^2\bigg) (R^2+g_G(\xi,\xi))^{-(n+1)/2-3}=\\
\qquad\qquad=&-48\mathfrak{c}_{2,n}(n+3)_{4,-2}(g_G(\nabla^2\varphi))^2\xi_n^2 (R^2+g_G(\xi,\xi))^{-(n+1)/2-2}+\\
&+24\mathfrak{c}_{2,n}(n+5)_{5,-2}g_G(\nabla^2\varphi)\xi_n^2\nabla^2\varphi((\xi'+\xi_n\nabla\varphi)^{\otimes 2})(R^2+g_G(\xi,\xi))^{-(n+1)/2-3}-\\
&-4\mathfrak{c}_{2,n}(n+7)_{6,-2}(\nabla^2\varphi((\xi'+\xi_n\nabla\varphi)^{\otimes 2}))^2\xi_n^2(R^2+g_G(\xi,\xi))^{-(n+1)/2-4}+\\
&+(12\mathfrak{c}_{2,n}(n+5)_{5,-2}+3\mathfrak{c}_{1,n}(n^2-1) )(g_G(\nabla^2\varphi))^2(R^2+g_G(\xi,\xi))^{-(n+1)/2-1} +\\
& - 6\mathfrak{c}_{1,n}(n+3)_{3,-2}\bigg(-2\xi_n\nabla^3\varphi(1_{n-1}\otimes(\xi'+\xi_n\nabla\varphi))+\\
&\qquad\qquad\qquad\qquad\qquad+(g_G(\nabla^2\varphi))(\nabla^2\varphi((\xi'+\xi_n\nabla\varphi)^{\otimes 2})\bigg) (R^2+g_G(\xi,\xi))^{-(n+1)/2-2} + \\
& + \mathfrak{c}_{1,n}(n+5)_{4,-2} \bigg(-2\xi_n(\nabla^3\varphi)((\xi'+\xi_n\nabla\varphi)^{\otimes 3})+\\
&\qquad\qquad\qquad\qquad\qquad\qquad+(\nabla^2\varphi((\xi'+\xi_n\nabla\varphi)^{\otimes 2})^2\bigg) (R^2+g_G(\xi,\xi))^{-(n+1)/2-3},
\end{align*}
\normalsize
and for $n=3$ that 
\small
\begin{align*}
q_2(x,\xi,R)=& (\mathfrak{c}_{2,3}240\cdot 4+\mathfrak{c}_{1,3} 24)(g_G(\nabla^2\varphi))^2(R^2+g_G(\xi,\xi))^{-3} - \\
&-\mathfrak{c}_{2,3}45\cdot 96\cdot 4(g_G(\nabla^2\varphi))^2\xi_n^2(R^2+g_G(\xi,\xi))^{-4}+\\
&+\mathfrak{c}_{2,3}120\cdot 96\cdot 4g_G(\nabla^2\varphi)\xi_n^2\nabla^2\varphi((\xi'+\xi_n\nabla\varphi)^{\otimes 2})(R^2+g_G(\xi,\xi))^{-5} \\
&- \mathfrak{c}_{2,3}80\cdot 96\cdot 4(g_G(\nabla^2\varphi))^2\xi_n^2(R^2+g_G(\xi,\xi))^{-6}+\\
& -\mathfrak{c}_{1,3}288 \bigg(-2\xi_n(\nabla^3\varphi)(1_{n-1}\otimes(\xi'+\xi_n\nabla\varphi))+\\
&\qquad\qquad\qquad\qquad\qquad\qquad(g_G(\nabla^2\varphi))(\nabla^2\varphi((\xi'+\xi_n\nabla\varphi)^{\otimes 2}))\bigg)  (R^2+g_G(\xi,\xi))^{-4} + \\
&+ \mathfrak{c}_{1,3}192\bigg(-2\xi_n(\nabla^3\varphi)((\xi'+\xi_n\nabla\varphi)^{\otimes 3})+(\nabla^2\varphi((\xi'+\xi_n\nabla\varphi)^{\otimes 2}))^2\bigg) (R^2+g_G(\xi,\xi))^{-5}.
\end{align*}
\normalsize
We note that $g_G(\nabla^2\varphi)(x_0)$ is $(n-1)/2$ times the mean curvature in $x_0$.
\end{example}

\begin{example}[Symbol computations for a submanifold of Euclidean space]
\label{subsmandiodladexu}
We consider a submanifold $M\subseteq \R^N$ and the function $G(x,y)=|x-y|^2$ which is regular at the diagonal. This example fits into the bigger picture of the paper seeing that $G(x,y)=\rd(x,y)^2$ where $\rd$ is the distance function on $M$ making the inclusion $M\subseteq \R^N$ isometric. To Taylor expand $G$ as in \eqref{taylorexpamdmd}, we take coordinates around a point $x_0\in M$ such that $M$ near $x_0$ is parametrized by 
$$\begin{cases}
x_l=x_l,\; &l=1,\ldots,n\\
x_{l}=\varphi_l(x_1,\ldots, x_n), \; &l=n+1,\ldots,N\end{cases},$$
for some functions $\varphi_{n+1},\ldots, \varphi_N$. Write $x=(x_1,\ldots, x_n)$. In these coordinates, 
\begin{align*}
G(x,y)&=|x-y|^2+\sum_{l=n+1}^N |\varphi_l(x)-\varphi_l(y)|^2=\\
&=|v|^2+\sum_{j=2}^{N_0}\sum_{l=n+1}^N\sum_{\substack{|\alpha|+|\beta|=j,\\ |\alpha|,|\beta|>0}} \frac{\partial_{x}^{\alpha}\varphi_l(x)\partial_{x}^{\beta}\varphi_l(x)}{\alpha!\beta!} v^{\alpha+\beta}+O(|v|^{N_0+1}),
\end{align*}
where $v=x-y$. Therefore, we conclude that 
$$H_G(v)=|v|^2+\sum_{l=n+1}^N(\nabla \varphi_l(x)\cdot v)^2,$$
and for $j>2$, 
$$C^j_G(x,v)=\sum_{l=n+1}^N\sum_{\substack{|\alpha|+|\beta|=j,\\ |\alpha|,|\beta|>0}} \frac{\partial_{x}^{\alpha}\varphi_l(x)\partial_{x}^{\beta}\varphi_l(x)}{\alpha!\beta!} v^{\alpha+\beta}=\sum_{l=n+1}^N\sum_{\substack{i+k=j,\\ i,k>0}}(\nabla^i\varphi_l\otimes \nabla^k\varphi_l)(v).$$
Computations similar to those in Example \ref{euxcomdmoda1} can be carried out also for submanifolds. To preserve the reader's sanity, we spare the details.
\end{example}

\begin{example}[Symbol computations for geodesic distances]
\label{geodesciexamokad1}
Consider a manifold $M$ equipped with a Riemannian metric $g_M$. To avoid having to prescribe the distance between different components, we assume that $M$ is connected. The geodesic distance $\rd_{\rm geo}:M\times M\to [0,\infty)$ is defined by 
$$\rd_{\rm geo}(x,y):=\inf\{L(c): \;\mbox{$c$ is a smooth path from $x$ to $y$}\},$$
where the length $L(c)$ of a path $c:[0,1]\to M$ is defined by 
$$L(c):=\int_0^1\sqrt{g_{M,c(t)}(\dot{c}(t),\dot{c}(t))}\rd t.$$
Here we write $\dot{c}:[0,1]\to TM$ for the derivative of the path, so $\dot{c}(t)\in T_{c(t)}M$. For a suitable neighborhood $U\subseteq TM$ of the zero section, the Riemannian metric defines an exponential map $\mathrm{exp}:\pmb{U}\to M$. More precisely, for a small enough $v\in T_xM$, $\exp_x(v)=\exp(x,v)\in M$ is defined in local coordinates as $\exp_x(v)=w_x(v;1)$ where $w_x(v;\cdot):[0,1]\to TM$ is the solution to the second order ordinary differential equation 
\begin{equation}
\label{geodesicequa}
\begin{cases}
\ddot{w}_x(v,t)+\Gamma_{w_x(v;t)}(\dot{w}_x(v,t),\dot{w}_x(v,t))=0,\\
w(v;0)=x,\\
\dot{w}(v,0)=-v,\end{cases}
\end{equation}
where $\Gamma$ is the affine connection defined from $g$, which in local coordinates is a vector valued symmetric bilinear form on the tangent bundle. In these local coordinates, for $x$ and $y$ close enough we have that 
$$\rd_{\rm geo}(x,y)=|\exp_x^{-1}(y)|_g^2.$$
In particular, $\rd_{\rm geo}^2$ is smooth in a neighborhood of the diagonal. 

Let us compute the Taylor expansion of $\rd_{\rm geo}^2$ as in Equation \eqref{taylorexpamdmd} and prove that $\rd_{\rm geo}^2$ is regular at the diagonal. We use a coordinate neighborhood as above, and write $v=x-y$. We are looking for the Taylor expansion in $v$ of the coordinate function $X_x(v)=\exp_x^{-1}(x-v)$. We Taylor expand 
$$w_x(v,t)=x-vt+\sum_{k=2}^N \frac{w^{(k)}_x(v;0)}{k!}t^k+O(|tv|^{N+1}).$$
We note that $v\mapsto w^{(k)}_x(v;0)$ is a homogeneous polynomial of degree $k$, we denote this by $W_k(x;v)$. It follows from Equation \eqref{geodesicequa} that 
\begin{equation}
\label{firsttermsinw}
\begin{cases}
W_1(v)=-v,\\
W_2(v)=-\Gamma(v,v),\\
W_3(v)=-\rd\Gamma(v,v,v)-2\Gamma(v,\Gamma(v,v)).
\end{cases}
\end{equation}
The higher order terms $W_4, W_5,\ldots$ can be computed inductively from Equation \eqref{geodesicequa}. Write the Taylor expansion of the unknown function $X_x$ in $v$ as
$$X_x(v)=\sum_{k=1}^N X^{(k)}(x;v)+O(|v|^{N+1}),$$
where $X^{(k)}(x;\cdot)$ is a homogeneous polynomial of degree $k$ in $v$. 
The identity $X_x(v)=\exp_x^{-1}(x-v)$ is equivalent to $\exp_x(X_x(v))=x-v$ which implies that 
\begin{equation}
\label{detmeirneodm}
-\sum_{k=1}^N X^{(k)}(x;v)+\sum_{k=2}^N \frac{W_k(x,\sum_{l=1}^N X^{(l)}(x;v))}{k!}=-v+O(|v|^{N+1}).
\end{equation}
Considering the first order term, we see that $X^{(1)}(v)=v$. The higher order terms can be inductively determined by considering each homogeneous term separately: 
\begin{equation}
\label{detmeeodmdetmeirndm}
X^{(k)}(v)=\left[\sum_{j=2}^k \frac{W_j(x,\sum_{l=1}^{k-1} X^{(l)}(x;v))}{j!}\right]_{(k)},
\end{equation}
where $[\cdot]_{(k)}$ denotes the homogeneous term of degree $k$. Using Equation \eqref{firsttermsinw} and \eqref{detmeirneodm}, we compute the first terms to be 
$$
\begin{cases}
X^{(1)}(v)&=v,\\
X^{(2)}(v)&=\frac{1}{2}W_2(x;v)=-\frac{1}{2}\Gamma(v,v),\\
X^{(3)}(v)&=\frac{1}{6}W_3(x;v)+\frac{1}{2}\left[W_2\left(x;v-\frac{1}{2}\Gamma(v,v)\right)\right]_{(3)}=\\
&=-\frac{1}{6}\rd\Gamma(v,v,v)+\frac{1}{6}\Gamma(v,\Gamma(v,v)).
\end{cases}$$
We summarize these computations in a proposition.
\end{example}

\begin{prop}
\label{taylrofofororgeoeod}
Let $M$ be a manifold equipped with a Riemannian metric $g_M$ and let $\rd_{\rm geo}$ denote the geodesic distance. Then $\rd_{\rm geo}^2$ is regular at the diagonal and in a coordinate neighborhood the Taylor expansion as in Equation \eqref{taylorexpamdmd}  takes the form 
$$\rd_{\rm geo}^2(x,y)=|v|^2_{g_M}+C^3_{\rd_{\rm geo}^2}(x;v)+C^4_{\rd_{\rm geo}^2}(x;v)+O(|v|^5_{g_m}),$$
where 
\begin{align*}
C^3_{\rd_{\rm geo}^2}(x;v)=&-g_M(v,\Gamma(v,v)),\\
C^4_{\rd_{\rm geo}^2}(x;v)=&\frac{1}{4}|\Gamma(v,v)|^2_{g_M}+\frac{1}{3}\left( g(v,\rd \Gamma( v,v,v))-g(v,\Gamma(v,\Gamma(v,v)))\right).
\end{align*}
The higher order terms $C^j_{\rd_{\rm geo}^2}$ can be computed inductively from Equation \eqref{detmeeodmdetmeirndm}.
\end{prop}

Let us return to the general case and describe the overall structure of the terms appearing in the full symbol expansion of $Q_{G,\chi}$.

\begin{lem}
\label{firstofzcor}
Let $M$ and $G$ be as in Theorem \ref{firstfirstofz}. The entries in the full symbol expansion $q\sim \sum_{j=0}^\infty q_j$ of $Q_{G,\chi}\in \Psi^{-n-1}_{\rm cl}(M;\C_+)$ takes the form 
$$q_j(x,\xi,R)=\sum_{k=0}^{3j} P_{k,j}(x,\xi)(R^2+g_G(\xi,\xi))^{-\frac{n+1+j+k}{2}},$$
where $P_{k,j}$ are of the form 
\begin{enumerate}
\item $P_{k,j}$ is a homogeneous polynomial of degree $k$ in $\xi$ and 
$$P_{k,j}\equiv 0\quad\mbox{if $j-k\notin 2\Z$}.$$
\item If $j-k\in 2\Z$, the polynomial $P_{k,j}$ takes the form 
$$P_{k,j}(x,\xi)=\sum_{\gamma\in I_j} \rho_{\gamma,k,j} C^{(\gamma)}_G(x,\underbrace{\iota_\xi g_G\otimes \iota_\xi g_G}_{k \; times}\otimes \underbrace{g_G\otimes\cdots \otimes  g_G}_{\frac{|\gamma|-k}{2}\;  times}),$$
for some coefficients $\rho_{\gamma,k,j}\in \Q\pi^{n/2}+\Q\pi^{(n+1)/2}+\Q\pi^{(n-1)/2}$.
\end{enumerate}
\end{lem}

In the last item, we note that if $j-k\in 2\Z$, then $|\gamma|-k\in 2\Z$ for all $\gamma\in I_j$.

\begin{proof}
It follows from the computations in Theorem \ref{firstfirstofz} that $q_j$ can be written as a finite sum
$$q_j(x,\xi,R)=\sum_{k\geq 0} P_{k,j}(x,\xi)(R^2+g_G(\xi,\xi))^{-\frac{n+1+j+k}{2}},$$
where $P_{k,j}$ is a polynomial in $\xi$ whose coefficients (as a polynomial in $\xi$) are all polynomials in the Taylor coefficients of $G$ at the diagonal. For the degrees to match, $P_{k,j}$ must be of degree $k$. Since the powers $R^2+g_G(\xi,\xi)$ must differ from $-(n+1)/2$ by an integer, $P_{k,j}=0$ unless $j-k\notin 2\Z$. Finally, the largest possible degree $k$ for which $\xi^\alpha(R^2+g_G(\xi,\xi))^{-\frac{n+1+j+k}{2}}$ (with $|\alpha|=k$) can be a summand in $q_j$ is if the derivative $C^{(\gamma)}_G(x,D_\xi)$ acts only on $(R^2+g_G(\xi,\xi))^{-(n+1)/2+\mathrm{rk}(\gamma)}$ (or $(R^2+g_G(\xi,\xi))^{-(n+1)/2+\mathrm{rk}(\gamma)}\log(R^2+g_G(\xi,\xi))$ if $2\mathrm{rk}(\gamma)- n-1\in 2\N$) and in that case $\alpha=\gamma$, so the maximal degree of $P_{k,j}$ is the size of the largest index in $I_j$, i.e. $3j$. This proves item (1). To prove item (2), one notices that by Theorem \ref{firstfirstofz}, all possible entries in $q_j$ consists of terms of the form 
$$C^{(\gamma)}_G(x,\underbrace{\iota_\xi g_G\otimes \iota_\xi g_G}_{k \; times}\otimes \underbrace{g_G\otimes\cdots \otimes  g_G}_{\frac{|\gamma|-k}{2}\;  times})(R^2+g_G(\xi,\xi))^{-\frac{n+1+j+k}{2}},$$
with a coefficient being a rational number times either $\pi^{n/2}$, $\pi^{(n+1)/2}$ or $\pi^{(n-1)/2}$. 
\end{proof}

\subsection{The symbol structure of the parametrix}

The operator $Q_{G,\chi}\in \Psi^{-n-1}_{\rm cl}(M;\C_+)$ considered in the previous subsection is elliptic with parameter by Theorem \ref{firstfirstofz}. In particular, it admits a parametrix $A_{G,\chi}\in \Psi^{n+1}_{\rm cl}(M;\C_+)$, that is an operator with parameter so that $A_{G,\chi}Q_{G,\chi}-1, Q_{G,\chi}A_{G,\chi}-1\in  \Psi^{-\infty}_{\rm cl}(M;\C_+)$.

\begin{thm}
\label{symbolstructureinversepsido}
Let $M$ be an $n$-dimensional manifold and $G:M\times M\to [0,\infty)$ a function which is regular at the diagonal. The full symbol $a$ in local coordinates of the parametrix $A_{G,\chi}\in \Psi^{n+1}_{\rm cl}(M;\C_+)$ of $Q_{G,\chi}\in \Psi^{-n-1}_{\rm cl}(M;\C_+)$ has an asymptotic expansion $a\sim \sum_{j=0}^\infty a_j$ where $a_j$ is constructed inductively from the symbol expansion $q\sim \sum_j q_j$ of Theorem \ref{firstfirstofz} by
$$a_j=-a_0\sum_{k+l+|\alpha|=j, \, l<j}\frac{1}{\alpha!} \partial_\xi^\alpha q_k D_x^\alpha a_{l}.$$
Here $a_0(x,\xi,R)=\frac{1}{n!\omega_n}(R^2+g_G(\xi,\xi))^{(n+1)/2}$. The next term is given by 
\begin{align*}
n=1: \quad a_1(x,\xi,R) = & -\frac{i}{2} (\partial_\xi g_G)(\partial_x g_G)(R^2+g_G(\xi,\xi))^{-1} + \frac{3i}{2}C^3_G(x,g_G\otimes \iota_\xi g_G) - \\
&-2iC^3_G(x,\iota_\xi g_G \otimes\iota_\xi g_G \otimes\iota_\xi g_G )(R^2+g_G(\xi,\xi))^{-1}\\
n>1: \quad a_1(x,\xi,R)=&{-}\frac{(n+1)^2i}{n!\omega_n}g_{G}(\rd_x g_{G}(\xi,\xi),\xi)(R^2+g_G(\xi,\xi))^{(n+1)/2-2}+\\
&+\frac{3 i \mathfrak{c}_{1,n}(n^2-1)}{(n!)^2\omega_n^2}C^{3}_G(x,g_G\otimes \iota_\xi g_G)(R^2+g_G(\xi,\xi))^{(n+1)/2-1} - \\
&-\frac{ i \mathfrak{c}_{1,n}(n+3)_{3,-2}}{(n!)^2\omega_n^2}C^{3}_G(x,\iota_\xi g_G\otimes \iota_\xi g_G\otimes \iota_\xi g_G)(R^2+g_G(\xi,\xi))^{(n+1)/2-2}.
\end{align*}

The homogeneous terms $a_j$ each takes the form 
$$a_j(x,\xi,R)=\sum_{k=0}^{3j} \tilde{P}_{k,j}(x,\xi)(R^2+g_G(\xi,\xi))^{\frac{n+1-j-k}{2}},$$
where $\tilde{P}_{k,j}$ are homogeneous polynomials of degree $k$ in $\xi$ and 
$$\tilde{P}_{k,j}\equiv 0\quad\mbox{if $j-k\notin 2\Z$}.$$
The coefficients of $\tilde{P}_{k,j}$ as a polynomial in $\xi$ are all polynomials in derivatives of the Taylor coefficients $(C^{(\gamma)}_G)_{\gamma\in \cup_{k\leq j} I_k}$ of $G$ at the diagonal (from Equation \eqref{taylorexpamdmd}) of total degree $j$. 
\end{thm}

In the structural statement at the end of the proposition, the total degree refers to the degree of the polynomial when counting an order $i$ derivative of a Taylor coefficient $C^{(\gamma)}_G$ for $\gamma\in I_k$ to have degree $k+i$.

\begin{proof}
It follows from the parametrix construction for elliptic operators (see e.g. \cite[Chapter I.5]{shubinbook}) that $a\sim \sum_{j=0}^\infty a_j$ where $a_0=q_0^{-1}$ and $a_j:=-a_0\sum_{k+l+|\alpha|=j, \, l<j}\frac{1}{\alpha!} \partial_\xi^\alpha q_k D_x^\alpha a_{l}$ for $j>0$. 

Let us prove the structural statement about $a_j$ by induction on $j$. It is clear for $j=0$. Assume that the structural statement holds for $l<j+1$. We have that 
\small
\begin{align*}
q_0a_{j+1}=&-\sum_{k+l+|\alpha|=j+1, \, l<j+1}\frac{1}{\alpha!} \partial_\xi^\alpha q_k D_x^\alpha a_{l}=\\
=&-\sum_{k+l+|\alpha|=j+1, \, l<j+1}\sum_{i_1=0}^{N(k)}\sum_{i_2=0}^{\tilde{N}(l)}\frac{1}{\alpha!} \partial_\xi^\alpha \left[P_{i_1,k}(x,\xi)(R^2+g_{G}(\xi,\xi))^{-\frac{n+1+k+i_1}{2}}\right] \cdot\\
&\qquad\qquad\qquad\qquad\qquad\qquad\qquad\qquad\qquad\qquad\qquad \cdot D_x^\alpha \left[\tilde{P}_{i_2,l}(x,\xi)(R^2+g_{G}(\xi,\xi))^{\frac{n+1-l-i_2}{2}}\right],\\
\end{align*}
\normalsize
which proves that $a_{j+1}$ has the claimed structure.

What remains is to compute $a_1$. We write
\small
\begin{align*}
a_1(x,\xi,R)=&-a_0\sum_{|\alpha|=1} \partial^\alpha_\xi q_0D^\alpha_xa_0-a_0^2q_1=\\
=&-\frac{1}{n!\omega_n}(R^2+g_G(\xi,\xi))^{(n+1)/2}\sum_{|\alpha|=1} \partial^\alpha_\xi (R^2+g_G(\xi,\xi))^{-(n+1)/2}D^\alpha_x(R^2+g_G(\xi,\xi))^{(n+1)/2}\\
&{+\frac{i(n^2-1)\mathfrak{c}_{1,n}}{(n!)^2\omega_n^2}(R^2+g_G(\xi,\xi))^{n+1}\bigg(3C^{3}_G(x,g_G\otimes \iota_\xi g_G)(R^2+g_G(\xi,\xi))^{-(n+1)/2-1}-}\\
&{\quad\quad\quad-(n+3)C^{3}_G(x,\iota_\xi g_G\otimes \iota_\xi g_G\otimes \iota_\xi g_G)(R^2+g_G(\xi,\xi))^{-(n+1)/2-2}\bigg)=}\\
=&{-}\frac{(n+1)^2i}{n!\omega_n}g_{G}(\rd_x g_{G}(\xi,\xi),\xi)(R^2+g_G(\xi,\xi))^{(n+1)/2-2}+\\
&+\frac{3 i \mathfrak{c}_{1,n}(n^2-1)}{(n!)^2\omega_n^2}C^{3}_G(x,g_G\otimes \iota_\xi g_G)(R^2+g_G(\xi,\xi))^{(n+1)/2-1} - \\
&-\frac{ i \mathfrak{c}_{1,n}(n+3)_{3,-2}}{(n!)^2\omega_n^2}C^{3}_G(x,\iota_\xi g_G\otimes \iota_\xi g_G\otimes \iota_\xi g_G)(R^2+g_G(\xi,\xi))^{(n+1)/2-2}
\end{align*}
\normalsize
\end{proof}

We use the notation $\Gamma_\alpha(R_0):=\{z\in \C: |\mathrm{Arg}(z)|<\alpha, \mathrm{Re}(z)>R_0\}$.

\begin{cor}
\label{cortombolstructureinversepsido}
Let $M$ be a compact $n$-dimensional manifold and $G:M\times M\to [0,\infty)$ a function which is regular at the diagonal. For some $R_0>0$, the operator $Q_{G,\chi}(R)\in \Psi^{-n-1}_{\rm cl}(M)$ is invertible as an operator $H^{-(n+1)/2}(M)\to H^{(n+1)/2}(M)$ for any $R\in \Gamma_{\pi/(n+1)}(R_0)$ and 
$$Q_{G,\chi}^{-1}-A_{G,\chi}\in \Psi^{-\infty}_{\rm cl}(M;\Gamma_{\pi/(n+1)}(R_0)).$$
In particular, $Q_{G,\chi}^{-1}$ is a pseudodifferential operator with parameter in $\Gamma_{\pi/(n+1)}(R_0)$ whose full symbol $a$ in local coordinates has an asymptotic expansion $a\sim \sum_{j=0}^\infty a_j$ where $a_j$ is as in Theorem \ref{symbolstructureinversepsido}.
\end{cor}

Corollary \ref{cortombolstructureinversepsido} follows from Theorem \ref{firstfirstofz} and \ref{symbolstructureinversepsido} using standard techniques for pseudodifferential operator with parameter. For the convenience of the reader, we include its proof.

\begin{proof}
Let $\Delta$ denote a positive Laplace operator on $M$ whose principal symbol coincides with the metric $g_G$ dual to the transversal Hessian of $G$. It follows from Theorem \ref{firstfirstofz} that $\overline{\sigma}_{-n-1}((R^2+\Delta)^{-\mu})=\overline{\sigma}_{-n-1}(Q)$. Hence, the operator $\mathfrak{r}(R):=1-(R^2+\Delta)^{\mu/2} Q(R^2+\Delta)^{\mu/2} $ is a parameter-dependent pseudodifferential operator of order $-1$. We write 
$$Q_{G,\chi}=(R^2+\Delta)^{-\mu/2}(1-\mathfrak{r})(R^2+\Delta)^{-\mu/2} .$$
The order of $\mathfrak{r}$ is $-1$, so by \cite[Theorem 9.1]{shubinbook} we have that $\|\mathfrak{r}(R)\|_{L^2(M)\to L^2(M)}=O(\mathrm{Re}(R)^{-1})$ as $\mathrm{Re}(R)\to \infty$. We conclude that there exists an $R_0>1$ such that $\|\mathfrak{r}(R)\|_{L^2(M)\to L^2(M)}\leq \frac{1}{2}$ for $R\in \Gamma_{\pi/(n+1)}(R_0)$. Since $(R^2+\Delta)^{-\mu/2}:H^s(M)\to H^{s+\mu}(M)$ is invertible for $R\in \Gamma_{\pi/(n+1)}(0)$ and any $s\in \R$, we can for $R\in \Gamma_{\pi/(n+1)}(R_0)$ invert $Q_{G,\chi}$ as the absolutely convergent series of operators $H^{\mu}(M)\to H^{-\mu}(M)$ given by
$$Q_{G,\chi}^{-1}=\sum_{k=0}^\infty (R^2+\Delta)^{\mu/2}\mathfrak{r}(R)^k(R^2+\Delta)^{\mu/2}.$$

It remains to prove that $Q_{G,\chi}^{-1}-A_{G,\chi}\in \Psi^{-\infty}_{\rm cl}(M;\Gamma_{\pi/(n+1)}(R_0))$. By the construction above, $Q_{G,\chi}^{-1}\in \Psi^{n+1}_{\rm cl}(M;\Gamma_{\pi/(n+1)}(R_0))$. Therefore the classes $[Q_{G,\chi}^{-1}]$ and $[A_{G,\chi}]$ in the formal symbol algebra $\Psi^{n+1}_{\rm cl}(M;\Gamma)/\Psi^{-\infty}_{\rm cl}(M;\Gamma_{\pi/(n+1)}(R_0))$ are both inverses to 
$$[Q_{G,\chi}]\in \cup_{k\in \Z} \Psi^{k}_{\rm cl}(M;\Gamma_{\pi/(n+1)}(R_0))/\Psi^{-\infty}_{\rm cl}(M;\Gamma_{\pi/(n+1)}(R_0)).$$ 
By the uniqueness of inverses, $[Q_{G,\chi}^{-1}]=[A_{G,\chi}]\in \Psi^{n+1}_{\rm cl}(M;\Gamma_{\pi/(n+1)}(R_0))/\Psi^{-\infty}_{\rm cl}(M;\Gamma_{\pi/(n+1)}(R_0))$.
\end{proof}

\subsection{Analytic results for $Q$ on compact manifolds}
\label{qoncomapalala}

The results of the previous subsections have analytic implications in the case that $M$ is a compact manifold. We consider the scale of Hilbert spaces $H^s_R(M):=(R^2+\Delta)^{-s/2}L^2(M)$ defined for $R\in \R\setminus \{0\}$ and $s\in \R$ with the Hilbert space structure making $(R^2+\Delta)^{-s/2}:L^2(M)\to H^s_R(M)$ unitary. Here $\Delta$ could be any choice of Laplacian, but for the sake of simplicity we fix the Laplacian associated with the Riemannian metric associated with a function regular at the diagonal. By elliptic regularity, $H^s(M)=H^s_R(M)$ as vector spaces with equivalent norms independently of the choice of Laplacian, but the Hilbert space structure differs in a non-uniform way as $R$ varies. At this stage, we shall start to concern ourselves with extensions of $Q$ to the complex numbers, so we phrase our results in terms of the operator 
\begin{equation}
\label{eqoreoedo}
Q(R)f(x):=\frac{1}{R}\int_M \chi(x,y)\e^{-R\rd(x,y)}f(y)\rd y,
\end{equation}
where $\rd$ is a distance function whose square is regular at the diagonal and $\chi$ a function being $1$ near the diagonal such that $\rd^2$ is smooth on the support of $\chi$. We note that $Q(R)=Q_{\rd^2,\chi}(R)$ for $\mathrm{Re}(R)>0$. 

\begin{thm}
\label{symbcor}
Let $M$ be a compact $n$-dimensional manifold and $\rd:M\times M\to [0,\infty)$ a distance function whose square is regular at the diagonal. Set $\mu:=(n+1)/2$. The operator
$$Q(R):H^{-\mu}(M)\to H^{\mu}(M),$$ 
defined from the expression \eqref{eqoreoedo} is a well defined Fredholm operator for all $R\in \C\setminus \{0\}$ and there is an $R_0$ such that $Q(R)$ is invertible for all $R\in \Gamma_{\pi/(n+1)}(R_0)$. 
Moreover, the following holds:
\begin{enumerate}
\item[a)] For each $R\in \C\setminus \{0\}$, $Q(R)\in \Psi^{-n-1}_{\rm cl}(M)$ is an elliptic pseudodifferential operator and the family of operators 
$$(Q(R):H^{-\mu}(M)\to H^{\mu}(M))_{R\in \C\setminus \{0\}}$$ 
depends holomorphically on $R\in \C\setminus \{0\}$. Moreover, we can extend the holomorphic family $(Q(R)^{-1}:H^{\mu}(M)\to H^{-\mu}(M))_{R\in \Gamma_{\pi/(n+1)}(R_0)}$
meromorphically to $R\in \C\setminus \{0\}$.
\item[b)] There is a $C>0$ such that 
$$C^{-1}\|f\|_{H^{-\mu}_{|R|}(M)}^2\leq \mathrm{Re}\langle f,Q(R) f\rangle_{L^2}\leq C\|f\|_{H^{-\mu}_{|R|}(M)}^2,$$
for $R\in \Gamma_{\pi/(n+1)}(R_0)$ and $f\in H^{-\mu}(M)$. In particular, for $R\in \Gamma_{\pi/(n+1)}(R_0)$, the norm $\mathrm{Re}\langle \cdot, Q(R)\cdot \rangle$ is uniformly equivalent to the norm on $H^{-\mu}_{|R|}(M)$.
\end{enumerate}
\end{thm}

\begin{proof}
The first statements of the theorem follows from Corollary \ref{cortombolstructureinversepsido}. 

Part a) follows from the meromorphic Fredholm theorem (see \cite[Proposition 1.1.8]{leschhab}) upon proving that $R\mapsto Q(R)\in \mathbb{B}(H^{-\mu}(M), H^{\mu}(M))$ is holomorphic with values in the set of Fredholm operators. We can write $Q(R)$ as the integral operator with Schwartz kernel 
\begin{equation}
\label{taylroaodoafiniter}
\frac{1}{R}\chi(x,y)\mathrm{e}^{-R\rd(x,y)}=\frac{\chi(x,y)}{R}+\sum_{k=0}^\infty \frac{R^k}{(k+1)!}\chi(x,y)\rd(x,y)^{k+1}.
\end{equation}
Using the fact that $\rd^2$ is regular at the diagonal, an argument as in the proof of Theorem \ref{firstfirstofz} shows that $Q(R)$ is an elliptic pseudodifferential operator of order $-n-1$. By Proposition \ref{ftoffpa}, the principal symbol  of $Q(R)$ (for fixed $R$) is given by
$$\sigma_{-n-1}(Q(R))(x,\xi)=\pi^{(n-1)/2}2^{n}\Gamma\left(\frac{n+1}{2}\right)|\xi|_{g_G}^{-n-1}=n!\omega_n|\xi|^{-n-1}_{g_{\rd^2}}.$$
Therefore $R\mapsto Q(R)\in \mathbb{B}(H^{-\mu}(M), H^{\mu}(M))$ takes values in the set of Fredholm operators. The expression \eqref{taylroaodoafiniter}, and again an argument as in the proof of Theorem \ref{firstfirstofz}, shows that $R\mapsto Q(R)\in \mathbb{B}(H^{-\mu}(M), H^{\mu}(M))$ depends holomorphically on $R\in \C\setminus \{0\}$.

Part b) follows from the G\aa rding inequality (see Theorem \ref{gardingwithpara} on page \pageref{gardingwithpara}) using that $\mathrm{Re}(Q)$ has positive principal symbol as a pseudodifferential operator with parameter by Theorem \ref{firstfirstofz} on page \pageref{firstfirstofz}).
\end{proof}

\subsection{Evaluation at $\xi=0$ of some symbols}

For later purposes, we will be interested in knowing the value of the homogeneous component of the full symbol of the parametrix of $Q$ at $\xi=0$, constructed as in Theorem \ref{symbolstructureinversepsido}. The following lemma shows that the evaluations of symbols of operators with parameter provides coordinate invariant expressions, therefore containing invariants of a pseudodifferential operator with parameter. 

\begin{lem}
\label{restiricocoald}
Assume that $M$ is a manifold and that $A\in \Psi^m_{\rm cl}(M;\Gamma)$ is a properly supported pseudodifferential operator with parameter. Then there exists a sequence $(a_{j,0})_{j\in \N}\subseteq C^\infty(M\times \Gamma)$ such that 
\begin{enumerate}
\item[i)] Each $a_{j,0}=a_{j,0}(x,R)$ is homogeneous of degree $m-j$ in $R$.
\item[ii)] In each local coordinate chart, $ a_{j,0}(x,R)=a_j(x,0,R)$
where $a\sim \sum_j a_j$ is a homogeneous expansion of the full symbol of $A$ in that chart.
\end{enumerate}
Moreover, for any $N\in \N$, we have that
\begin{equation}
\label{expansisonandao}
[A(R)1](x)=\sum_{j=0}^N a_j(x,R)+r_{N}(x,R)=\sum_{j=0}^N a_j(x,1)R^{m-j}+r_{N}(x,R),
\end{equation}
where $r_N\in C^\infty(M\times \Gamma)$ is a function such that for any compact $K\subseteq M$ it holds that
$$\sup_{x\in K}|\partial_x^\alpha\partial_R^kr_N(x,R)|=O(\mathrm{Re}(R)^{m-N+|\alpha|+k}), \quad\mbox{as $\mathrm{Re}(R)\to +\infty$}.$$
\end{lem}

\begin{proof}
Choose a partition of unity $(\chi_j)_j\subseteq C^\infty_c(M)$ subordinate to a locally finite covering by coordinate charts, and choose $(\tilde{\chi}_j)_j\subseteq C^\infty_c(M)$ such that $\tilde{\chi}_j$ is supported in a coordinate chart and $\tilde{\chi}_j=1$ on $\mathrm{supp}(\chi_j)$. We have that $\sum_j \chi_j A\tilde{\chi}_j$ converges in weak sense to a properly supported operator, and $A-\sum_j \chi_j A\tilde{\chi}_j\in \Psi^{-\infty}_{\rm cl}(M;\Gamma)$. Therefore, we can assume that $A$ is supported in a coordinate chart. In a coordinate chart, and a homogeneous expansion $a\sim \sum_j a_j$ of the full symbol of $A$, the method of stationary phase (see for instance \cite[Chapter VII.7]{horI}) implies that $[A(R)1](x)=\sum_{j=1}^N a_j(x,0,R)+r_{N}(x,R)$ as in Equation \eqref{expansisonandao}. It is clear that that the function $A(R)1\in C^\infty(M)$ and its asymptotic expansion is independent of choice of coordinates, and the lemma follows.
\end{proof}

\begin{lem}
\label{evaluationsofinterioraxizeroind}
Let $M$ be an $n$-dimensional manifold, $G:M\times M\to [0,\infty)$ be a function regular at the diagonal, and $(a_{j,0})_{j\in \N}\subseteq C^\infty(M\times \C_+)$ the sequence defined as in Lemma \ref{restiricocoald} from $A_{G,\chi}\in \Psi^{n+1}_{\rm cl}(M;\C_+)$ (cf. Theorem \ref{symbolstructureinversepsido}). The functions $a_{j,0}$ are determined in local coordinates by the properties that 
$$a_{j,0}(x,R)=0, \quad\mbox{for all $j$ odd},$$
and for even $j$, $a_{j,0}(x,R)$ is determined inductively from $a_{0,0}(x,R)=\frac{1}{n!\omega_n}R^{n+1}$ and for $j>0$,
\begin{equation}
\label{formalforaj0}
a_{j,0}(x,R)=-\frac{1}{n!\omega_n}R^{n+1}\sum_{\substack{k+2l+p=j\\2l<j, \, 2|k+p}}i^pq_{k,p}(x,R).\nabla_x^pa_{j-2k,0}(x,R),
\end{equation}
where $q_{k,p}(x,R)$ denotes the $p$-linear form
$$q_{k,p}(x,R).v:=\sum_{|\alpha|=p}\partial_\xi^\alpha q_k(x,\xi,R)v^\alpha|_{\xi=0}=\partial_t^pq_k(x,tv,R)|_{t=0}.$$
\end{lem}

\begin{proof}
By Lemma \ref{restiricocoald}, we can perform all computations in a coordinate chart. The structural description of $a_j$ from Lemma \ref{symbolstructureinversepsido} implies that for any $(x,\xi,R)$ and $j\in \N$ it holds that
$$a_j(x,-\xi,R)=(-1)^j a_j(x,\xi,R).$$
We conclude that $a_j(x,0,R)=0$, and even that $\nabla_x^pa_{j,0}(x,R)=0$ for any $p$, when $j$ is odd. To compute $a_j$ for even $j$, we note that since $a_j=-a_0\sum_{k+l+|\alpha|=j, \, l<j}\frac{1}{\alpha!} \partial_\xi^\alpha q_k D_x^\alpha a_{l}$, the formula \eqref{formalforaj0} follows using that the only contributions are for even $l$, and evenness of $j$ implies that $2|k+p$.

\end{proof}

\begin{remark}
The formulas in Theorem \ref{firstfirstofz} shows that for $n$ odd, and $j$ even,
\small
\begin{align*}
q_{j,0}(x,R)=&R^{-n-1-j}\sum_{\gamma\in I_j, \mathrm{rk}(\gamma)<(n+1)/2}  \mathfrak{c}_{\mathrm{rk}(\gamma),n}(-1)^{\mathrm{rk}(\gamma)-1}\left(|\gamma|/2\right)!(n+1-2\mathrm{rk}(\gamma))_{|\gamma|/2,-2}C_G^{(\gamma)}(x,g_G^{\otimes |\gamma|/2})+\\
&-R^{-n-1-j}\sum_{\gamma\in I_j, \mathrm{rk}(\gamma)\geq (n+1)/2} \mathfrak{c}_{\mathrm{rk}(\gamma),n}(-1)^{\mathrm{rk}(\gamma)-1}\left(|\gamma|/2\right)!(n+1-2\mathrm{rk}(\gamma))_{|\gamma|/2-1,-2}C_{G}^{(\gamma)}(x,g_G^{\otimes |\gamma|/2}),
\end{align*}
\normalsize
and for $n$ even, and $j$ even,
\begin{align*}
q_{j,0}(x,R)=&R^{-n-1-j}\sum_{\gamma\in I_j} \mathfrak{c}_{\mathrm{rk}(\gamma),n}(-1)^{\mathrm{rk}(\gamma)-1}\left(|\gamma|/2\right)!(n+1-2\mathrm{rk}(\gamma))_{|\gamma|/2,-2}C_{G}^{(\gamma)}(x,g_G^{\otimes |\gamma|/2}),
\end{align*}
and $\mathfrak{c}_{\mathrm{rk}(\gamma),n}$ is as in Theorem \ref{firstfirstofz}.
Note that by definition of the set $I_j$, $\gamma\in I_j$ satisfies that $|\gamma|$ is even if and only if $j$ is even. 
\end{remark}

\begin{thm}
\label{evaluationsofinterioraxizero}
Let $M$ be an $n$-dimensional manifold, $G:M\times M\to [0,\infty)$ be a function regular at the diagonal, and $(a_{j,0})_{j\in \N}\subseteq C^\infty(M\times \C_+)$ the restriction to $0$ of the full symbol of $A_{G,\chi}$. Then for each $j>0$, $a_{j,0}$ is a polynomial in $(C^{(\gamma)}_G)_{\gamma\in \cup_{k\leq j}I_k}$ and its derivatives contracted by the metric $g_G$ and its derivatives of total degree $j$ where each $C^{(\gamma)}_G$, $\gamma \in I_k$, has degree $k$, the metric has degree zero and $x$-derivatives increase the order by $1$. 

In the special case $j=0$ we have 
$$a_{0,0}(x,R)=\frac{1}{n!\omega_n}R^{n+1},$$
and for $j=2$ and $n=3$, we have that 
\begin{align*}
a_{2,0}(x,R)=-\frac{24R^{2}}{(3!\omega_3)^2}\bigg(10\mathfrak{c}_{1,3}C^{4}_G(x,g_G\otimes g_G)- \mathfrak{c}_{2,3}(C^{3}_G\otimes C^{3}_G)(x,g_G\otimes g_G\otimes g_G)\bigg),
\end{align*}
while for $j=2$ and $n\neq 1,3$, we have that 
\begin{align}
\label{a20comp}
a_{2,0}(x,R)=-\frac{3R^{n-1}}{(n!\omega_n)^2}\bigg(\mathfrak{c}_{1,n}(n^2-1)&C^{4}_G(x,g_G\otimes g_G)- \\
\nonumber
&-\mathfrak{c}_{2,n}(n+5)_{5,-2}(C^{3}_G\otimes C^{3}_G)(x,g_G\otimes g_G\otimes g_G)\bigg),
\end{align}
and for for $j=4$ we have the identity 
\begin{align}
\nonumber
a_{4,0}(x,R)=-\frac{R^{2n+2}}{(n!\omega_n)^2}&q_{4,0}(x,R)-\frac{n+1}{R^{2}}g_G(\nabla^2_x a_{2,0}(x,R))-\\
\label{a30comp}
&-\frac{\mathfrak{c}_{1,n}(n^2-1)}{R^2n!\omega_n}(C^3_G\otimes \nabla_x a_{2,0})(g_G\otimes g_G)-\frac{R^{n+1}}{n!\omega_n}q_{2,0}(x,R)a_{2,0}(x,R).
\end{align}
\end{thm}

\begin{proof}
The structural statement about $a_{j,0}(x,R)$ is readily deduced from Lemma \ref{evaluationsofinterioraxizeroind} and induction by using the property that if $\gamma\in I_{j}$ and $\gamma'\in I_{j'}$ then $(\gamma,\gamma')\in I_{j+j'}$.

The equation for $a_{0,0}$ is immediate from Lemma \ref{symbolstructureinversepsido}. Using Lemma \ref{evaluationsofinterioraxizeroind}, we have that $a_{2,0}(x,R) = -a_{0,0}^2q_{2,0}$ and the formula \eqref{a20comp} follows from that
\begin{align*}
q_{2,0}(x,R)=&3\mathfrak{c}_{1,n}(n^2-1)C^{4}_G(x,g_G\otimes g_G)R^{-n-3}+{3}\mathfrak{c}_{2,n}(n+5)_{5,-2}(C^{3}_G\otimes C^{3}_G)(x,g_G\otimes g_G\otimes g_G)R^{-n-3},
\end{align*}
for $n\neq 1,3$ and a similar expression for $n=3$.

Let us compute $a_{4,0}(x,R)$. By Lemma \ref{evaluationsofinterioraxizeroind} we have that 
\begin{align*}
a_{4,0}(x,R)=&-\frac{R^{n+1}}{n!\omega_n}\bigg(\sum_{k=0}^4i^{-k}q_{k,4-k}(x,R).\nabla_x^{4-k}\frac{R^{n+1}}{n!\omega_n}-\sum_{k=0}^2i^{-k}q_{k,2-k}(x,R).\nabla_x^{2-k}a_{2,0}(x,R)\bigg)=\\
=&\frac{R^{n+1}}{n!\omega_n}\bigg(q_{0,2}(x,R).\nabla_x^2a_{2,0}(x,R)-iq_{1,1}(x,R).\nabla_xa_{2,0}(x,R)-q_{2,0}(x,R)a_{2,0}(x,R)\bigg)-\\
&\quad -\frac{R^{2n+2}}{(n!\omega_n)^2}q_{4,0}(x,R),
\end{align*}
and the computation is complete upon using Proposition \ref{computq1}.
\end{proof}

\begin{remark}
The full expression for $a_{4,0}$ can be computed from Equation \eqref{a30comp}. We omit the full details, but let us note an expression for $q_{4,0}$. Since 
$$I_4=\{6, (3,5),(4,4),(5,3),(3,3,4),(3,4,3),(4,3,3),(3,3,3,3)\},$$ 
we have that 
\begin{align*}
q_{4,0}(x,R)=&R^{-n-5}\mathfrak{c}_{1,n}3!(n-1)_{3,-2}C_{G}^{6}(x,g_G^{\otimes 3})-R^{-n-5}\mathfrak{c}_{2,n}4!(n-3)_{4,-2}C_{G}^{(3,5)}(x,g_G^{\otimes 4})-\\
&-R^{-n-5}\mathfrak{c}_{2,n}4!(n-3)_{4,-2}C_{G}^{(4,4)}(x,g_G^{\otimes 4})-R^{-n-5}\mathfrak{c}_{2,n}4!(n-3)_{4,-2}C_{G}^{(5,3)}(x,g_G^{\otimes 4})+\\
&+R^{-n-5}\mathfrak{c}_{3,n}5!(n-5)_{5,-2}C_{G}^{(3,3,4)}(x,g_G^{\otimes 5})+R^{-n-5}\mathfrak{c}_{3,n}5!(n-5)_{5,-2}C_{G}^{(3,4,3)}(x,g_G^{\otimes 5})+\\
&+R^{-n-5}\mathfrak{c}_{3,n}5!(n-5)_{5,-2}C_{G}^{(4,3,3)}(x,g_G^{\otimes 5})-R^{-n-5}\mathfrak{c}_{4,n}6!(n-7)_{6,-2}C_{G}^{(3,3,3,3)}(x,g_G^{\otimes 6}).
\end{align*}
\end{remark}

\begin{example}[Evaluations of symbols for domains in Euclidean space]
\label{euxcomdmoda2}
Let us return to the computations on Euclidean space from Example \ref{euxcomdmoda1}. We consider $G(x,y)=|x-y|^2$ -- the square of the Euclidean distance. Since $q(x,\xi,R)=n!\omega_n(R^2+|\xi|^2)^{-(n+1)/2}$ is a full symbol of $Q_{G,\chi}$ in Euclidean coordinates, $a(x,\xi,R)=\frac{1}{n!\omega_n}(R^2+|\xi|^2)^{(n+1)/2}$ is a full symbol of $A_{G,\chi}$. Therefore 
\begin{equation}
\label{intsymeuccodod}
a_{j,0}(x,R)=\begin{cases}
\frac{1}{n!\omega_n}R^{n+1}, \; &j=0,\\
0, \; &j>0.
\end{cases}
\end{equation}
By Lemma \ref{evaluationsofinterioraxizeroind} this holds in any coordinate system on Euclidean space. We remark that the bulk of computations carried out in Example \ref{euxcomdmoda1} will mainly be of interest when inverting $Q$ near the boundary of a domain, while the computation \eqref{intsymeuccodod} relates to interior terms.
\end{example}

\begin{example}[Evaluations of symbols for submanifolds of Euclidean space]
\label{subsmandiodladexu2}
We return to submanifolds $M\subseteq \R^N$ and the function $G(x,y)=|x-y|^2$ which is regular at the diagonal as in Example \ref{subsmandiodladexu} above. We take coordinates as in Example \ref{subsmandiodladexu}. By Theorem \ref{evaluationsofinterioraxizero} we have that $a_{0,0}(x,R)=\frac{1}{n!\omega_n}R^{n+1}$ and
\small
\begin{align*}
a_{2,0}(x,R)=&-\frac{3 \mathfrak{c}_{1,n}(n^2-1)}{(n!\omega_n)^2}\sum_{l=n+1}^N\sum_{\substack{i+k=4,\\ i,k>0}}(\nabla^i\varphi_l\otimes \nabla^k\varphi_l)(g_G^{\otimes 2})R^{n-1}-\\
&-\frac{3 \mathfrak{c}_{2,n}(n+5)_{3,-2}}{(n!\omega_n)^2}\sum_{l_1,l_2=n+1}^N\sum_{\substack{i_1+k_1=i_2+k_2=3,\\ i_1,i_2,k_1,k_2>0}}(\nabla^{i_1}\varphi_{l_1}\otimes \nabla^{k_1}\varphi_{l_1}\otimes \nabla^{i_2}\varphi_{l_2}\otimes \nabla^{k_2}\varphi_{l_2})(x,g_G^{\otimes 3})R^{n-1},
\end{align*}
\normalsize
where $g_G$ is the dual metric to the transversal Hessian:
$$H_G(v)=|v|^2+\sum_{l=n+1}^N(\nabla \varphi_l(x)\cdot v)^2.$$
\end{example}

\begin{example}[Evaluations of symbols for geodesic distances]
\label{geodesciexamokad2}
Consider the geodesic distance on a Riemannian manifold $M$ as in Example \ref{geodesciexamokad1}. Let $g_M$ denote the Riemannian metric, and recall from  Example \ref{geodesciexamokad1} that $g_M$ coincides with the transversal Hessian of $\rd_{\rm geo}^2$ at the diagonal. We can compute $a_{2,0}$ in this case by means of known Riemannian curvatures. We fix a point $x$ and choose coordinates so that $\Gamma$ vanishes in that point (normal coordinates). In these coordinates, $C^3_{\rd^2_{\rm geo}}(x,v)=0$ and $C^4_{\rd^2_{\rm geo}}(x,\cdot)$ is a third of the Riemannian curvature in $x$ by Proposition \ref{taylrofofororgeoeod}. By Theorem \ref{evaluationsofinterioraxizero} we conclude that for $n=3$, we have $a_{2,0}(x,R)=-\frac{80R^{2}}{(3!\omega_3)^2}s_g(x)$,
while for $n\neq 1,3$, we have that 
$$
a_{2,0}(x,R)=-\frac{R^{n-1}}{(n!\omega_n)^2}\mathfrak{c}_{1,n}(n^2-1)s_g(x),
$$
where $s_g$ denotes the scalar curvature.
\end{example}

\section{Global behavior of $\mathcal{Z}$ on compact manifolds}
\label{remadinedlsosec}

In the previous section we studied the localization $Q$ of $\mathcal{Z}$ near the diagonal. Here $\mathcal{Z}$ is the operator  with parameter $R$ defined from a distance function as in Equation \eqref{definedindoz}. We now turn to study $\mathcal{Z}$ by considering distance functions for which $L:=\mathcal{Z}-Q$ in a certain sense is negligible so that $Q$ dominates. 

\subsection{Controlling the off-diagonal part of $\mathcal{Z}$}
In this subsection we shall study the remainder $L=\mathcal{Z}-Q$. We note that 
$$L(R)f(x)=\frac{1}{R}\int_M (1-\chi(x,y))\e^{-R\rd(x,y)}f(y)\rd y.$$
We start making two initial observations concerning the remainder term $L$. The first observation concerns the behavior of the remainder term in $L^2$.

\begin{prop}
\label{obse1}
Let $M$ be a compact manifold and $\rd:M\times M\to [0,\infty)$ a distance function on $M$. Then $\C\setminus \{0\}\ni R\mapsto L(R)\in \mathcal{L}^2(L^2(M))$ 
defines a holomorphic Hilbert-Schmidt valued function. Moreover, $L$ satisfies that 
$$\|L(R)\|_{\mathcal{L}^2(L^2(M))}=O(\mathrm{Re}(R)^{-\infty})\quad\mbox{as $\mathrm{Re}(R)\to +\infty$.}$$
\end{prop}

By the standard norm estimate $\|K\|_{\mathbb{B}}\leq \|K\|_{\mathcal{L}^2}$, the statement in the proposition also holds in the operator norm. We remark that by Theorem \ref{firstfirstofz} and \cite[Theorem 9.1]{shubinbook}, the localization to the diagonal satisfies $\|Q(R)\|_{\mathcal{L}^p(L^2(M))}=O(\mathrm{Re}(R)^{\frac{n}{p}-n-1})$ as $\mathrm{Re}(R)\to +\infty$, for any $p>1+1/n$ (the bound is not uniform in $p$).

\begin{proof}
The function $\C\setminus \{0\}\ni R\mapsto \frac{1}{R}(1-\chi(\cdot,\cdot))\e^{-R\rd(\cdot,\cdot)}\in C(M\times M)$ is clearly holomorphic in the norm on $C(M\times M)$. Since $\chi=1$ in a neighborhood of the diagonal, compactness of $M$ implies that $\| \frac{1}{R}(1-\chi(\cdot,\cdot))\e^{-R\rd(\cdot,\cdot)}\|_{C(M\times M)}=O(\mathrm{Re}(R)^{-\infty})$ as $\mathrm{Re}(R)\to +\infty$. Since $M$ is compact, an integral operator $K$ with kernel $k\in C(M\times M)$ satisfies the estimate for the Hilbert-Schmidt norm $\|K\|_{\mathcal{L}^2(L^2(M))}\leq \mathrm{vol}(M)\|k\|_{C(M\times M)}$. The proposition follows.
\end{proof}

The second observation concerns how an a priori estimate with respect to the off-diagonal remainder $L$ affects distributional solutions to the magnitude equation $\mathcal{Z}u=1$.

\begin{prop}
\label{obse2}
Let $M$ be a compact manifold and $\rd:M\times M\to [0,\infty)$ a distance function on $M$ such that $\rd^2$ is regular at the diagonal. Let $N\in \N$ and $\Gamma$ be a sector. Assume that $f\in C^\infty(M)$ and that $(u_R)_{R\in \Gamma}\subseteq \mathcal{D}'(M)$ is a family of solutions to 
$$\mathcal{Z}(R)u_R=f$$
satisfying that $\langle L(R)u_R,\psi\rangle=O(\mathrm{Re}(R)^{-N})$ as $\mathrm{Re}(R)\to +\infty$ in $\Gamma$ for all $\psi\in C^\infty(M)$. Then for large enough $\mathrm{Re}(R)$,
$$u_R=Q(R)^{-1}f+v_R,$$
where $(v_R)_{R\in \Gamma}\subseteq \mathcal{D}'(M)$ is a family satisfying that $\langle v_R,\psi\rangle=O(\mathrm{Re}(R)^{-N+n+1})$ as $\mathrm{Re}(R)\to +\infty$ in $\Gamma$ for all $\psi\in C^\infty(M)$. In particular, for any $\psi\in C^\infty(M)$, 
$$\langle u_R,\psi\rangle=\langle Q(R)^{-1}f,\psi\rangle +O(\mathrm{Re}(R)^{-N+n+1}), \quad\mbox{as $\mathrm{Re}(R)\to +\infty$ in $\Gamma$}.$$
\end{prop}

\begin{proof}
The equation $\mathcal{Z}(R)u_R=f$ implies that $u_R=Q(R)^{-1}f-Q(R)^{-1}L(R)u_R$.
Consider the distribution $v_R:=-Q(R)^{-1}L(R)u_R\equiv u_R-Q(R)^{-1}f$. For $\psi\in C^\infty(M)$, we have that 
$$\langle v_R,\psi\rangle=\langle L(R)u_R,Q(R)^{-1}\psi\rangle=O(R^{-N+n+1}),$$
because of the assumption on $L(R)u_R=O(\mathrm{Re}(R)^{-N})$ in a weak sense and the fact that $Q(R)^{-1}$ is a pseudodifferential operator with parameter of order $n+1$ preserving $C^\infty(M)$ with uniform norm estimates $\|Q(R)^{-1}f\|_{C^{k}}\leq C_k(1+\mathrm{Re}(R))^{n+1}\|f\|_{C^{n+k+1}}$. 
\end{proof}

Proposition \ref{obse1} implies that $L$ is small as an operator on $L^2$, and Proposition \ref{obse2} implies that the resulting remainder term will not alter the asymptotic properties of solutions given an a priori estimate. However, as $Q$ is of negative order and acts compactly on $L^2$, well-posedness of the problem in $L^2$ is not assured. We shall circumvent this problem by imposing a regularity assumption on the distance function that forces the magnitude equation naturally into a Sobolev space framework. We discuss examples satisfying this assumption, as well as counterexamples, below in the Subsections \ref{exsobrefl} and \ref{countexsobrefl}, respectively. 

\begin{deef}[Property (MR) of distance functions]
\label{sovodlwwowm}
Let $M$ be an $n$-dimensional compact manifold and $\rd:M\times M\to [0,\infty)$ a distance function on $M$ such that $\rd^2$ is regular at the diagonal. Set $\mu:=(n+1)/2$. For a sector $[1,\infty)\subseteq \Gamma\subseteq \C$, we say that $\rd$ has \emph{property (MR)} on $\Gamma$ if for any $R\in \Gamma$, $L(R)$ extends to a continuous mapping $H^{-\mu}(M)\to H^{\mu}(M)$ with 
$$\|L(R)\|_{H^{-\mu}(M)\to H^{\mu}(M)}=O(\mathrm{Re}(R)^{-\infty}),\quad \mbox{as $\mathrm{Re}(R)\to +\infty$ in $\Gamma$}.$$
If $L(R):H^{-\mu}(M)\to H^{\mu}(M)$ is a compact operator for $R\in \Gamma$ and $\Gamma\ni R\mapsto L(R)\in \mathbb{K}(H^{-\mu}(M), H^{\mu}(M))$ 
is holomorphic in norm sense, we say that $\rd$ has \emph{property (SMR)} on $\Gamma$.
\end{deef}

The acronyms MR and SMR stand for magnitude regularity and strong magnitude regularity, respectively. Assuming these properties, the operator $\mathcal{Z}$ inherits relevant analytic and geometric properties from $Q$. Property (MR) will be used to compute asymptotic solutions to the magnitude equation $R\mathcal{Z}u=1$, while property (SMR) will be used for constructing meromorphic extensions of $\mathcal{Z}^{-1}$. If $\Gamma$ is a sector on which a distance function has property (MR), it is clear that $0\notin \Gamma$. We consider such results for compact manifolds below in Subsection \ref{analaforclosedz}.

\subsection{Examples of distance functions satisfying property (MR) and (SMR)}
\label{exsobrefl}

Let us give a method to produce distance functions with property (SMR):

\begin{prop}
\label{sobregpropdomain}
Let $M$ be a compact manifold embedded into a Riemannian manifold $i:M\to W$ such that the square of its geodesic distance $\rd_{{\rm geo},W}^2$ is smooth, for instance $W=\R^N$, $W=\mathbb{H}_{N,\R}$ or $W=\mathbb{H}_{N,\C}$, for some $N\in \N$. 

The distance function $\rd:M\times M\to [0,\infty)$, $\rd(x,y):=\rd_{{\rm geo},W}(i(x),i(y))$ satisfies 
\begin{enumerate}
\item[i)] $\rd^2$ is smooth on $M\times M$ and regular at the diagonal.
\item[ii)] $L\in \Psi^{-\infty}(M;\C_+)$ and $L(R)\in \Psi^{-\infty}(M)$ for any $R\in \C\setminus\{0\}$. 
\item[iii)] $\rd$ has property (SMR) on $\C\setminus \{0\}$. 
\end{enumerate}
\end{prop}

\begin{proof}
Since $\rd_{{\rm geo},W}^2$ is smooth, it is clear that $\rd^2$ is smooth on $M\times M$. It follows from Proposition \ref{taylrofofororgeoeod} that $\rd^2$ is regular at the diagonal, and in fact $g_{\rd^2}$ is the pullback metric $i^*g_W$ from the Riemannian metric $g_W$ on $W$. Therefore the singular support support of $\rd$ is the diagonal, and the integral kernel of $L(R)$ is smooth so $L(R)$ extends to a continuous mapping $L(R):H^s(M)\to H^t(M)$ for any $s,t\in \R$. Moreover, for any vector fields $X_1,\ldots, X_m$ on $M\times M$ we can estimate 
$$\left|X_1\cdots X_m\left((1-\chi)\e^{-R\rd}\right)\right|\leq C_m|R|^{m-1}\e^{-\epsilon \mathrm{Re}(R)}.$$
Again, we write $\epsilon:=\inf\{\rd(x,y): (x,y)\in \chi^{-1}(1)\}>0$. In particular, we readily can deduce that $L\in \Psi^{-\infty}(M;\C_+)$, and for any $s,t\in \R$,
$$\|L(R)\|_{H^{s}(M)\to H^{t}(M)}=O(\mathrm{Re}(R)^{-\infty}),\quad \mbox{as}\quad \mathrm{Re}(R)\to+ \infty.$$
Therefore $\rd$ has property (SMR) on $\C\setminus \{0\}$. 
\end{proof}

In light of Proposition \ref{sobregpropdomain}, we note the following corollary of Theorem \ref{symbcor}. 

\begin{cor}
\label{strongleledldforrn}
Let $M$ be a compact manifold equipped with a distance function $\rd:M\times M\to [0,\infty)$ such that $\rd^2$ is smooth, e.g. a subspace distance as in Proposition \ref{sobregpropdomain}. Then $\mathcal{Z}(R)\in \Psi^{-n-1}_{\rm cl}(M)$ is an elliptic pseudodifferential operator for any $R\in \C\setminus \{0\}$. Furthermore, $\mathcal{Z}\in \Psi^{-n-1}_{\rm cl}(M;\C_+)$ is elliptic with parameter and its full symbol coincides with that of $Q$ as given in Theorem \ref{firstfirstofz}. 
\end{cor}

Another example of distance functions with property (MR) arises on spheres.

\begin{prop}
\label{sobregpropsphere}
Let $\rd$ denote the geodesic distance on a sphere $S^n$ in its round metric. The distance function $\rd$ has property (MR) on $\C\setminus \{0\}$ but fails to satisfy property (SMR) on any sector.
\end{prop}

\begin{proof}
The square of the geodesic distance is regular at the diagonal by Proposition \ref{taylrofofororgeoeod}. We note that $\rd$ is smooth on $\{(x,y)\in S^n\times S^n: x\neq \pm y\}$. Consider the operator $Uf(x):=f(-x)$. The operator $U$ acts via pullback along the antipodal mapping $\varphi(x):=-x$ which acts isometrically due to $O(n)$-invariance of the Riemannian metric on $S^n$. The operator $U$ extends to a unitary on all Sobolev spaces $H^s(S^n)$, $s\in \R$. Since it holds that $\rd(x,y)=\pi-\rd(x,\varphi(y))$ we can conclude that the integral kernel of $L(R)U$ is given by $\chi_\varphi\e^{-R\rd}\e^{-\pi R}$, where $\chi_\varphi(x,y):=1-\chi(x,\varphi(y))$. In particular, since $\chi_\varphi$ satisfies that $\chi_\varphi=1$ on the diagonal and $\chi_\varphi=0$ on a neighborhood of the off-diagonal singularities of $\rd$, the operator $L(R)U\e^{\pi R}$ is an elliptic pseudodifferential operator with parameter of order $-n-1$ by the same argument as in Theorem \ref{firstfirstofz}. It follows that property (SMR) fails on any sector. Using that $U$ is unitary, $L(R)$ extends to a continuous operator $H^{-\mu}(S^n)\to H^{\mu}(S^n)$ with $\|L(R)U\e^{\pi R}\|_{H^{-\mu}(S^n)\to H^{\mu}(S^n)}=O(1)$ as $\mathrm{Re}(R)\to \infty$. Since $U$ is unitary, we deduce that 
$$\|L(R)\|_{H^{-\mu}(S^n)\to H^{\mu}(S^n)}=O(\e^{-\pi \mathrm{Re}(R)})=O(\mathrm{Re}(R)^{-\infty})\quad \mbox{as}\quad \mathrm{Re}(R)\to \infty.$$
\end{proof}

\subsection{Analytic results for the operator $\mathcal{Z}$}
\label{analaforclosedz}

We are now ready to extend the results of Section \ref{seconsymbolofq} for $Q$ to results on the operator $\mathcal{Z}$ for compact manifolds with a distance function satisfying property (MR) as defined in Definition \ref{sovodlwwowm}. 

\begin{thm}
\label{firstofz}
Let $M$ be a compact $n$-dimensional manifold and $\rd:M\times M\to [0,\infty)$ a distance function on $M$ with property (MR) on the sector $\Gamma$. Set $\mu:=(n+1)/2$. Then there is an $R_0\geq 0$ such that 
$$\mathcal{Z}(R):H^{-\mu}(M)\to H^{\mu}(M),$$ 
is invertible for all $R\in \Gamma\cap \Gamma_{\pi/(n+1)}(R_0)$. Moreover, for $R\in \Gamma\cap \Gamma_{\pi/(n+1)}(R_0)$
$$\mathcal{Z}^{-1}=Q^{-1}+\mathcal{R},$$
where $Q^{-1}\in \Psi^{n+1}_{\rm cl}(M; \Gamma_{\pi/(n+1)}(R_0))$ is the elliptic pseudodifferential operator with parameter constructed in Corollary \ref{cortombolstructureinversepsido} and $\mathcal{R}:H^{\mu}(M)\to H^{-\mu}(M)$ is a family of operators parametrized by $R\in  \Gamma\cap \Gamma_{\pi/(n+1)}(R_0)$ such that 
$$\|\mathcal{R}\|_{H^{\mu}(M)\to H^{-\mu}(M)}=O(\mathrm{Re}(R)^{-\infty}),\quad \mbox{as $\mathrm{Re}(R)\to +\infty$ in $\Gamma\cap  \Gamma_{\pi/(n+1)}(R_0)$}.$$
Moreover, there is a constant $C>0$ such that 
\begin{equation}
\label{utforzgardinbclo}  
C^{-1}\|f\|_{H^{-\mu}_{|R|}(M)}^2\leq \mathrm{Re}\langle f,\mathcal{Z}(R) f\rangle_{L^2}\leq C\|f\|_{H^{-\mu}_{|R|}(M)}^2,
\end{equation}
for $R\in \Gamma\cap \Gamma_{\pi/(n+1)}(R_0)$ and $f\in H^{-\mu}(M)$. In particular, for $R\in \Gamma\cap  \Gamma_{\pi/(n+1)}(R_0)$, $\mathcal{Z}(R)$ is coercive in form sense on $L^2(M)$ for the $H^{-\mu}$-norm. 
\end{thm}

\begin{proof}
The operator $Q$ is invertible by Corollary \ref{cortombolstructureinversepsido}. We can therefore write 
$$Q^{-1}\mathcal{Z}=1+Q^{-1}L,$$
as operators on $H^{-\mu}(M)$. Since $Q^{-1}$ is a pseudodifferential operator with parameter, we have that 
\begin{align}
\label{lefterosod}
\|Q(R)^{-1}L(R)\|_{H^{-\frac{n+1}{2}}(M)\to H^{-\frac{n+1}{2}}(M)}&=O(\mathrm{Re}(R)^{-\infty}),\\
\nonumber
& \mbox{as $\mathrm{Re}(R)\to +\infty$ in $\Gamma\cap  \Gamma_{\pi/(n+1)}(R_0)$}.
\end{align}
Therefore, for $\mathrm{Re}(R)\gg0$ the operator $(1+Q^{-1}L)^{-1}Q^{-1}$ exists and is a left inverse to $\mathcal{Z}$. By an analogous argument, 
\begin{align}
\label{righterosod}
\|L(R)Q(R)^{-1}\|_{H^{\frac{n+1}{2}}(M)\to H^{\frac{n+1}{2}}(M)}&=O(\mathrm{Re}(R)^{-\infty}),
\end{align}
as $\mathrm{Re}(R)\to +\infty$ in $\Gamma\cap  \Gamma_{\pi/(n+1)}(R_0)$, and $\mathcal{Z}$ has the right inverse $Q^{-1}(1+LQ^{-1})^{-1}$ for $\mathrm{Re}(R)\gg0$. Therefore $(1+Q^{-1}L)^{-1}Q^{-1}=Q^{-1}(1+LQ^{-1})^{-1}$ and this operator is an inverse to $\mathcal{Z}$. By the estimates \eqref{lefterosod} and \eqref{righterosod}, it follows that 
$$\mathcal{R}=Z^{-1}-Q^{-1}=Q^{-1}\left((1+LQ^{-1})^{-1}-1\right)=\sum_{k=0}^\infty (-1)^k Q^{-1}(LQ^{-1})^k,$$ 
as a norm convergent sum and has the required decay property as $\mathrm{Re}(R)\to +\infty$.

The estimate \eqref{utforzgardinbclo} follows from the decay property of $\mathcal{R}$ and Theorem \ref{symbcor}.
\end{proof}

The following result is immediate from Lemma \ref{restiricocoald} and Theorem \ref{firstofz}.

\begin{cor}
\label{asym}
Let $M$ be a compact $n$-dimensional manifold and $\rd:M\times M\to [0,\infty)$ a distance function on $M$ with property (MR) on the sector $\Gamma$. Take the sequence of homogeneous functions $(a_{j,0})_{j\in \N}\subseteq C^\infty(M\times \Gamma\cap \C_+)$ as in Lemma \ref{evaluationsofinterioraxizeroind}. Then, for any $N\in \N$, we have that
$$[\mathcal{Z}(R)^{-1}1](x)=\sum_{j=0}^N a_j(x,R)+r_{N}(x,R),$$
where $r_N\in C(\Gamma\cap \C_+,H^{-\mu}(M))$, for $\mu=(n+1)/2$, is a function such that 
$$\|r_N(\cdot,R)\|_{H^{-\mu}(M)}=O(\mathrm{Re}(R)^{n+1-N}), \quad\mbox{as $\mathrm{Re}(R)\to +\infty$ in $\Gamma$}.$$
\end{cor}

\begin{remark}
We remark that the role that property (MR) plays in Corollary \ref{asym} is to ensure existence of a distributional solution to $\mathcal{Z}(R)u_R=1$. For \emph{computing} the integrated asymptotics $\langle 1,\mathcal{Z}(R)^{-1}1\rangle$, property (MR) is not necessary. Indeed, with no assumptions of property (MR), but assuming that $\rd:M\times M\to [0,\infty)$ is a distance function on $M$ such that $\rd^2$ is regular at the diagonal and that $(u_R)_{R>R_0}\subseteq \mathcal{D}'(M)$ is a family of solutions to 
$$\mathcal{Z}(R)u_R=1$$
satisfying that $\langle L(R)u_R,\psi\rangle=O(\mathrm{Re}(R)^{-N})$ as $\mathrm{Re}(R)\to +\infty$ for all $\psi\in C^\infty(M)$, we have that 
$$\langle u_R,1\rangle=\sum_{j=1}^N \int_Ma_j(x,R) \rd x+O(\mathrm{Re}(R)^{n+1-N}),$$
for any $N$. This follows from Lemma \ref{restiricocoald} and Proposition \ref{obse2}. 

One instance where solutions to $\mathcal{Z}(R)u_R=1$ exist, yet the distance function need not satisfy property (MR), is the geodesic distance on a compact symmetric space $M=G/H$. See Proposition \ref{faidlalaoadod} below for examples of symmetric spaces failing to satisfy property (MR). The problem $\mathcal{Z}(R)u_R=1$ was studied for compact symmetric spaces in \cite{will}. For a compact symmetric space $M=G/H$, we use the normalized $G$-invariant measure induced by the Haar measure. By symmetry, the function
$$u_R(x)=\frac{1}{\int_{G/H}\e^{-R\rd(x,y)}\rd y},$$
is constant and therefore solves $\mathcal{Z}(R)u_R=1$. It is readily verified that $L(R)u_R=O(R^{-\infty})$ in distributional sense. We conclude that each $a_j(x,R)$ is constant and that
$$u_R(x)=u_R(eH)=\sum_{j=0}^N a_j(eH,R)+O(\mathrm{Re}(R)^{n+1-N}),$$
for any $N$. For examples of computations of $u_R$ for compact symmetric spaces, see \cite{will}.
\end{remark}

The next result poses an obstruction to property (MR) for distance functions and should be viewed as complementary to Theorem \ref{descososdod}. Recall the following terminology from \cite{meckes1}: a compact metric space  $(X,\rd)$  is said to be positive definite if for any finite subset $F\subseteq X$, the matrix $(\e^{-\rd(x,y)})_{x,y\in F}$ is positive definite. If $(X,R\rd)$ is positive definite for all $R>0$, we say that $(X,\rd)$ is stably positive definite. 

\begin{cor}
\label{posididiveove}
Assume that $\rd$ is a distance function on a compact manifold $M$ with property (MR) on $[1,\infty)$. Then there exists an $R_0\geq 0$ such that $(M,R\rd)$ is positive definite for all $R>R_0$.
\end{cor}

\begin{proof}
Consider the quadratic form $\mathfrak{q}_R(u)=\langle u,\mathcal{Z}(R) u\rangle_{L^2}$, $u\in H^{-\mu}(M)$. By Theorem \ref{firstofz}, $\mathfrak{q}_R$ is positive definite for $R>R_0$, for some $R_0\geq 0$. In particular, for any subspace $V\subseteq H^{-\mu}(M)$ the restriction of $\mathfrak{q}_R$ to $V$ is also positive definite for $R>R_0$. 
For a finite subset $F\subseteq M$, consider the subspace $V_F\subseteq H^{-(\mu}(M)$ spanned by $\{\delta_x: x\in F\}$. In the basis $(\delta_x)_{x\in F}$, the quadratic form $\mathfrak{q}_R|_V$ is represented by the $|F|\times|F|$-matrix $(\e^{-R\rd(x,y)})_{x,y\in F}$ and so it is positive definite for $R>R_0$.
\end{proof}

\begin{remark}
It was proven in \cite[Subsection 3.2]{meckes1} that if $M$ is a compact Riemannian manifold with $\pi_1(M)\neq 0$, the geodesic distance is not stably positive definite, i.e. there exists an $R>0$ and a finite subset $F\subseteq M$ such that $(\e^{-R\rd_{\rm geo}(x,y)})_{x,y\in F}$ fails to be positive definite. The manifold $M=S^1$ is not simply connected, and therefore fails to be stably positive definite. but nevertheless Proposition \ref{sobregpropsphere} implies that $M=S^1$ with its geodesic distance has property (MR). Therefore, property (MR) does not imply stably positive definiteness of the metric space but only an asymptotic version thereof.
\end{remark}

\begin{thm}
\label{symbcorz}
Let $M$ be an $n$-dimensional compact manifold with a distance function $\rd$ having property (SMR) on $\Gamma$ and set $\mu=(n+1)/2$. Then the operator
$$\mathcal{Z}(R):H^{-\mu}(M)\to H^{\mu}(M),$$ 
is a well defined Fredholm operator for all $R\in \Gamma$ invertible for $R\in \Gamma\cap \Gamma_{\pi/(n+1)}(R_0)$. Moreover, the operator $\mathcal{Z}(R):H^{-\mu}(M)\to H^{\mu}(M)$ depends holomorphically on $R$ in $\Gamma$ and $\mathcal{Z}(R)^{-1}:H^{\mu}(M)\to H^{-\mu}(M)$ depends holomorphically on $R\in \Gamma\cap \Gamma_{\pi/(n+1)}(R_0)$ and admits a meromorphic extension to $\Gamma$.
\end{thm}

\begin{proof}
By Theorem \ref{symbcor}, $Q$ is a holomorphic function on $\Gamma$ with values in the Fredholm operators and by property (SMR), $L$ is a compact valued holomorphic function on $\Gamma$. Therefore $\mathcal{Z}$ defines a holomorphic function on $\Gamma$ with values in the Fredholm operators. Since $\mathcal{Z}$ is invertible for a large enough $R$, see Theorem \ref{firstofz}, the theorem follows from the meromorphic Fredholm theorem, see \cite[Proposition 1.1.8]{leschhab}.
\end{proof}

\begin{remark}
To ensure holomorphicity of $\mathcal{Z}$ and $\mathcal{Z}^{-1}$ on sectors, the full property (SMR) is not needed. Indeed, if $\rd$ has property (MR) on a sector $\Gamma$ and $\Gamma\ni R\mapsto L(R)\in \mathbb{B}(H^{-\mu}(M),H^\mu(M))$ is additionally holomorphic (in norm sense) then by Theorem \ref{firstofz}, for some $R_0\geq 0$, the mapping 
$$\Gamma\cap \Gamma_{\pi/(n+1)}(R_0)\ni R\mapsto \mathcal{Z}(R)\in \mathbb{B}(H^{-\mu}(M),H^\mu(M))$$
is a holomorphic family of invertible operators. These properties are inherited by its inverse,  
$$\Gamma\cap \Gamma_{\pi/(n+1)}(R_0)\ni R\mapsto \mathcal{Z}(R)^{-1}\in \mathbb{B}(H^{\mu}(M),H^{-\mu}(M)).$$
By the proof of Proposition \ref{sobregpropsphere}, this discussion applies to $S^n$ showing that for some $R_0$, $\mathcal{Z}(R)^{-1}$ is holomorphic for $R\in \Gamma\cap \Gamma_{\pi/(n+1)}(R_0)$ in this case.
\end{remark}

The following result follows from Corollary \ref{strongleledldforrn} and Theorem \ref{symbcorz}. 

\begin{thm}
\label{symbcorzzzz}
Let $M$ be a compact manifold with a distance function $\rd$ such that $\rd^2$ is smooth on $M\times M$ and regular at the diagonal (cf. Proposition \ref{sobregpropdomain}), then $\tilde{\mathcal{Z}}(R):H^{-\mu}(M)\to H^{\mu}(M)$ depends holomorphically on $R\in \C\setminus \{0\}$ and the operator $\mathcal{Z}(R)^{-1}:H^{\mu}(M)\to H^{-\mu}(M)$ extends meromorphically to $\C\setminus \{0\}$.
\end{thm}

\subsection{Examples of distance functions that fail to satisfy property (MR)}
\label{countexsobrefl}

Property (MR) of a distance function, as defined in Definition \ref{sovodlwwowm}, is a notable restriction on the singular support and singularity structure of the distance function. We remark here that by a singular point, we mean any point in the singular support, i.e. one in which the function is not $C^\infty$. To better understand how these singularities affect the operator theoretic properties of $L$, the reader is encouraged to review the proof of Proposition \ref{sobregpropsphere} where a crucial feature used in the proof is that the geodesic distance on spheres near an off-diagonal singularity has the same singular features as it has near the antipode of the singularity. An important property used there can be stated as having control of the dimension of the off-diagonal singular support of the metric. We make an elementary observation that follows from the smoothness of the function $(0,\infty)\ni t\mapsto \sqrt{t}\in (0,\infty)$. 

\begin{prop}
Let $\rd$ be a distance function on a manifold $M$. Then it holds that 
$$\mathrm{singsupp}(\rd)\setminus \mathrm{Diag}_M=\mathrm{singsupp}(\rd^2)\setminus \mathrm{Diag}_M.$$
\end{prop}

Let us give a sufficient condition for a distance function (regular at the diagonal) \emph{not} to satisfy property (MR). We note that an additional obstruction was provided above in Proposition \ref{posididiveove}.

\begin{thm}
\label{descososdod}
Let $M$ be a compact manifold and $\rd$ a distance function such that $\rd^2$ is regular at the diagonal. Assume that there exists a submanifold $N\subseteq M\times M$ with $N\subseteq \mathrm{singsupp}(\rd^2)\setminus \mathrm{Diag}_M$ such that any point $z_0\in N$ admits a neighborhood $U_0$ in $M\times M$ and a coordinate chart 
$$\varphi:\R^{\dim(N)}_t\times \R^{\dim(M)-\dim(N)}_s\to U_0,$$
with $\varphi(0)=x_0$, $\varphi^{-1}(U_0\cap N)= \R^{\dim(N)}_t\times \{0\}$ and $(t,s)\mapsto \varphi^*\rd(t,s)-|s|$ being smooth in a neighborhood of $0$.
Then $L(R)$ does not extend to a continuous operator 
$$H^{-\mu}(M)\to H^{s}(M),$$
for any $s>(3\dim(M)+1)/2-\dim(N)$ where $\mu=(\dim(M)+1)/2$. In particular, if $\dim(N)>\dim(M)$ then $\rd$ does not have property (MR) (see Definition \ref{sovodlwwowm}) on any sector containing a half-ray $[R_0,\infty)$. 
\end{thm}

\begin{proof}
We pick a point $(x_0,y_0)\in N$ and neighborhoods $U_1$ and $U_2$ of $x_0$ and $y_0$ respectively, such that there exists a coordinate chart $\varphi$ as in item ii) above on $U_0=U_1\times U_2$. Pick functions $\chi_1\in C^\infty_c(U_1)$ and $\chi_2\in C^\infty_c(U_2)$ such that $\chi_1=1$ near $x_0$, $\chi_2=1$ near $y_0$ and such that $(t,s)\mapsto \varphi^*\rd(t,s)-|s|$ is smooth on $\varphi^{-1}(\mathrm{supp}(\chi_1)\times \mathrm{supp}(\chi_2))$. Clearly, it suffices to prove that for any $R>0$, the operator $\chi_1L(R)\chi_2$ does not extend to a continuous operator $H^{-\mu}(M)\to H^{s}(M)$ for any $s>(3\dim(M)+1)/2-\dim(N)$. 

Set $k=\dim(N)$ and $n=\dim(M)$. The Schwartz kernel of $\chi_1L(R)\chi_2$ is a distribution on $M\times M$ supported in $U_0=U_1\times U_2$ and pulling this kernel back along $\varphi$, we arrive at the distribution
$$K(t,s)=\chi_0(s,t)\e^{-R|s|},$$
where $\chi_0=\varphi^*[(\chi_1\otimes\chi_2)(1-\chi)]\e^{-R\psi}\in C^\infty_c(\R^{2n})$ for $\psi(t,s)=\varphi^*\rd(t,s)-|s|$. It follows from Proposition \ref{ftoffpa} and combining a Taylor expansion with asymptotic completeness, that $K\in CI^{k-2n-1}(\R^{2n},\R^k)$. We conclude that $\chi_1L(R)\chi_2$ is a Fourier integral operator of order $k-2n-1$ and this operator is elliptic in a neighborhood of $(x_0,y_0)$. Therefore, since $\chi_1L(R)\chi_2$ is elliptic near $(x_0,y_0)$ of order $k-2n-1$ it does not extend to a continuous operator $H^{-(n+1)/2}(M)\to H^{s}(M)$ for any $s>-(n+1)/2-(k-2n-1)=(3n+1)/2-k$.
\end{proof}

\begin{prop}
\label{faidlalaoadod}
Let $n>1$. The geodesic distance on the $n$-dimensional torus $M=\mathbb{T}^n$ or the real projective space $M=\R P^n$ fails to satisfy property (MR) (see Definition \ref{sovodlwwowm}) on any sector containing a half-ray $[R_0,\infty)$. In fact, $L(R)$ does not extend to a continuous map $H^{-\mu}(M)\to H^{s}(M)$ for $s>-n/2+3/2$.
\end{prop}

\begin{proof}
The proofs for both cases follow the same lines and rely on Theorem \ref{descososdod}. For $\R P^n$ we give a more geometric argument, and for $\mathbb{T}^n$ we give a coordinate oriented argument. 

We first prove the result for real projective space. The projective space $\R P^n$ is the quotient of $S^n$ by the antipodal map $x\mapsto -x$. This quotient map is a covering map, and it is locally isometric with respect to the geodesic distance. For $x\in \R P^n$, the equator $\tilde{E}(x)\subseteq S^n$  is defined by 
$$\tilde{E}(x)=\{v\in S^n: v\cdot x=0\}.$$
The condition $v\cdot x=0$ being invariant under the antipodal map, $\tilde{E}(x)$ only depends on $x\in \R P^n$ and not on a choice of pre-image of $x$ in $S^n$. We let $E(x)\subseteq \R P^n$ denote the image under the quotient map. 
The off-diagonal singular support of the geodesic distance on $\R P^n$ is the set 
$$\{(x,y)\in \R P^n\times \R P^n: y\in E(x)\}.$$
Indeed, any geodesic in $\R P^n$ from a point $x$ lifts uniquely up to a geodesic on $S^n$ up until the point it crosses $E(x)$ where the geodesic distance $\rd_{\rm geo}(x,\cdot)$ has a kink. The projection mapping $p_1:\mathrm{singsupp}(\rd_{\rm geo})\setminus \mathrm{Diag}_{\R P^n}\to \R P^n$ is a locally trivial $\R P^{n-1}$-bundle on $\R P^n$. Therefore, $N=\mathrm{singsupp}(\rd_{\rm geo})\setminus \mathrm{Diag}_{\R P^n}\to \R P^n$ is a manifold of dimension $2n-1>n$ and suitable local trivializations of $p_1:N\to \R P^n$ satisfy the assumptions of item ii) in Theorem \ref{descososdod}. We conclude from Theorem \ref{descososdod} that $L(R)$ has no bounded extension to an operator $H^{-\mu}(\R P^n)\to H^{s}(\R P^n)$ for any $s>-n/2+3/2$.

We now prove the result for the $n$-dimensional torus $\mathbb{T}^n$. Write $\mathbb{T}^n=\R^n/\Z^n$. While $\mathbb{T}^n$ is a symmetric space, i.e. $\mathbb{T}^n=G/H$ for $G=\mathbb{T}^n$ and $H=1$, it is instructive to consider $\rd(x,y)$ for $x=0$. We have that $\rd(0,y)=|y|$ where we represent $y$ by as an element of the fundamental domain $[-1/2,1/2)^n$, and $\mathbb{T}^n\ni y\mapsto \rd(0,y)$ as a function on $\mathbb{T}^n$ is the $\Z^n$-periodic extension of $[-1/2,1/2)^n\ni y\mapsto |y|$. Therefore $y\mapsto \rd(0,y)$ has kinks on the image of $\partial\left([-1/2,1/2)^n\right)$ in $\mathbb{T}^n$. In particular, we see that the off-diagonal singular support of $\rd$ is the set 
$$\{(x,y)\in \mathbb{T}^n\times \mathbb{T}^n: x-y\in \partial\left([-1/2,1/2)^n\right)+\Z^n\}.$$
Consider the submanifold $N_0:=\{1/2\}\times (-1/4,1/4)^{n-1}\subseteq \R^n$ and define the $2n-1$-dimensional submanifold 
$$N:=\{(x,y)\in \mathbb{T}^n\times \mathbb{T}^n: x-y\in N_0+\Z^n\}\subseteq  \mathrm{singsupp}(\rd^2)\setminus \mathrm{Diag}_{\mathbb{T}^n}.$$
Consider a point 
$$x_0=(x,1/2,y')\in N.$$
On the open ball of radius $1/4$ centred at $x_0$, we introduce the coordinates $u=x_1-y_1-1/2$ and $t=(t',y)$ where $t'=x'-y'$ in terms of standard coordinates $x=(x_1,x')$ and $y=(y_1,y')$. In these coordinates, we have that
$$\rd(x,y)=\sqrt{\left(\frac{1}{2}-|u|\right)^2+|t'|^2}=\sqrt{-|u|+|t'|^2+\frac{1}{4}+|u|^2}.$$
By shrinking the neighborhood of $x_0$, we can Taylor expand 
\begin{align*}
\sqrt{-|u|+|t'|^2+\frac{1}{4}+|u|^2}=&\sqrt{|t'|^2+\frac{1}{4}+|u|^2}\sqrt{1-\frac{|u|}{|t'|^2+\frac{1}{4}+|u|^2}}=\sum_{k=0}^\infty \pmb{\alpha}_k \frac{|u|^k}{(|t'|^2+\frac{1}{4}+|u|^2)^{k-1}}=\\
=&|u|\underbrace{\sum_{k=0}^\infty \pmb{\alpha}_{2k+1} \frac{|u|^{2k}}{(|t'|^2+\frac{1}{4}+|u|^2)^{2k}}}_{g(u,t')}+\underbrace{\sum_{k=0}^\infty \pmb{\alpha}_{2k} \frac{|u|^{2k}}{(|t'|^2+\frac{1}{4}+|u|^2)^{2k-1}}}_{\tilde{g}(u,t')}.
\end{align*}
The functions $g$ and $\tilde{g}$ are smooth near $0$. Since $g(0,0)\neq 0$, we can define the new coordinate $s:=u\tilde{g}(u,t')$. We conclude that in these coordinates $\rd(x,y)-|s|$ is smooth. We conclude from Theorem \ref{descososdod} that $L(R)$ has no bounded extension to an operator $H^{-\mu}(\mathbb{T}^n)\to H^{s}(\mathbb{T}^n)$ for any $s>-n/2+3/2$.
\end{proof}

\begin{remark}
The analytical issues arising from the remainder term $L$ reflect fundamental problems in Riemannian geometry. 
The singular support of the geodesic distance is akin to the conjugate locus of the Riemannian metric, which in general is hard to describe, see \cite{bishcut,warnerconj} and for related technical issues arising in the $X$-ray transform on a Riemannian manifold see \cite{holuhl}. Furthermore, Theorem \ref{descososdod} shows that even when the singular support is a tractable set, i.e. when it looks like a submanifold near some point, dimensional obstructions to property (MR) appear. This gives rise to the analytic problem that the operator $L$ in the decomposition $\mathcal{Z}=Q+L$ is in general of order higher than $-n-1$ in the Sobolev order, while it is infinitely decaying in the parameter $R$. In this case $L$ is of higher order than $Q$, which is elliptic with parameter of order $-n-1$ and which determines the analytic and geometric properties of $\mathcal{Z}$ in this article.
\end{remark}

\section{The operator $\mathcal{Z}$ on Sobolev spaces for a manifold with boundary}
\label{qxondomains}

We now turn to compact manifolds with boundary. For simplicity, we tacitly assume that $X$ is a compact domain in a manifold $M$ and, for the purposes of this section, it suffices to assume that $X$ has a $C^0$-boundary. { We say that $X\subseteq M$ is a domain if it coincides with the closure of its interior points}. Recall that a domain is said to have $C^0$-boundary if its boundary can be realized locally as the graph of a continuous function. We call such spaces $X$ a compact manifold with $C^0$-boundary. We may then study the operators $\mathcal{Z}$ and $Q$ in $M$ and deduce results in $X$ by restriction to distributions supported in $X$. In this section we study analytic properties and meromorphic extensions. Asymptotic properties are studied in the following section under additional regularity assumptions on the boundary. For notational clarity, we indicate the manifold on which an operator is defined by a subscript, e.g. $\mathcal{Z}_X$ and $\mathcal{Z}_M$ for the corresponding operator on $X$ and $M$, respectively.

We shall make use of the following scales of Sobolev spaces. For $s\in \R$ and $R>0$, write
$$\dot{H}^s_R(X):=\{u\in H^s_R(M): \mathrm{supp}(u)\subseteq X\},\quad\mbox{and}\quad \overline{H}^s_R(X):=H^s_R(M)/\dot{H}^s_R({ \overline{M\setminus X}}).$$
An approximation argument shows that $C^\infty_c(X^\circ)\subseteq \dot{H}^s_R(X)$ is dense for any $s\in \R$, see \cite[Theorem 3.29]{mcleanbook}. We equip these Sobolev spaces with Hilbert space structure induced from $H^s_R(M)$, i.e. $\dot{H}^s_R(X)\subseteq H^s_R(M)$ as a subspace and $\overline{H}^s_R(X)$ as a quotient. We call the quotient mapping $H^{s}_{R}(M)\to \overline{H}^{s}_R(X)$ the restriction mapping because, by duality, it identifies $\overline{H}^{s}_R(X)$ with a space of distributions in $X^\circ$. Indeed, the continuous inclusion map $H^s_R(M)\hookrightarrow \mathcal{D}'(M)$ induces a continuous map $\overline{H}^s_R(X)\to \mathcal{D}'(X^\circ)$ by mapping the equivalence class $u+\dot{H}^s_R(M\setminus X^\circ)\in \overline{H}^s_R(X)$ to the distribution $u|_{X^\circ}\in \mathcal{D}'(X^\circ)$. The map $\overline{H}^s_R(X)\to \mathcal{D}'(X^\circ)$ is a continuous embedding by the following argument. If $u\in H^s_R(M)$ satisfies that $u|_{X^\circ}=0$ then $\mathrm{supp}(u)\subseteq M\setminus X^\circ$ and $u+\dot{H}^s_R(M\setminus X^\circ)=0+\dot{H}^s_R(M\setminus X^\circ)$ defines the zero class in $ \overline{H}^s_R(X)$.

We note that for $s=0$, $\dot{H}^0_R(X)=\overline{H}^0_R(X)=L^2(X)$. The $L^2$-pairing between $\dot{H}^s_R(X)$ and $\overline{H}^{-s}_R(X)$ is a perfect pairing and induces an isomorphism $\dot{H}^s_R(X)^*\cong \overline{H}^{-s}_R(X)$ (uniformly in $R$). { By an abuse of notation, we write $\langle \cdot ,\cdot \rangle_{L^2}:\dot{H}^{s}(X)\times \overline{H}^{-s}(X)\to \C$ for the $L^2$-pairing.} For $R=1$, we omit $R$ from the notation. We remark that any pseudodifferential operator with parameter $A\in \Psi^m_{\rm cl}(M;\Gamma)$ induces a continuous operator 
$$A_X:\dot{H}^s(X)\to  \overline{H}^{s-m}(X),$$ 
defined by the composition 
$$\dot{H}^s(X)\hookrightarrow H^s_c(M)\xrightarrow{A}  H^{s-m}_{\rm loc}(M)\to  \overline{H}^{s-m}(X).$$
Here the last map is the quotient map, $H^s_c(M)$ denotes the space of compactly supported distributions that are $s$-Sobolev regular and $H^{s-m}_{\rm loc}(M)$ denotes the space of distributions that are locally $s-m$-Sobolev regular. { We use the notation $A_X$ but remark that this operator maps distributions supported in $X$ to distributions restricted to $X^\circ$.}

Let us make two remarks regarding the operator $A_X:\dot{H}^s(X)\to  \overline{H}^{s-m}(X)$. Firstly, for any $s\in \mathbb{R}$ density ensures that $A_X:\dot{H}^s(X)\to  \overline{H}^{s-m}(X)$ is determined by continuity and the restriction $A_X:C^\infty_c(X^\circ )\to  C^\infty(X)$. Secondly, if $s=-m/2$ and $A$ is formally self-adjoint, then $A_X$ is determined from the polarization identity by the continuous quadratic form $$q_{A_X}(u):=\langle u,Au\rangle_{L^2}, \quad u\in \dot{H}^{-m/2}(X),$$
defined from $A$ and the perfect $L^2$-pairing $\dot{H}^{m/2}(X)\times \overline{H}^{-m/2}(X)\to \C$.

The analogue of Theorem \ref{symbcor} for $Q_X$ is the following theorem.

\begin{thm}
\label{symbcorboundaryq}
Let $X$ be a compact $n$-dimensional manifold with $C^0$-boundary and $\rd$ a distance function on $X$ such that $\rd^2$ is regular at the diagonal. Set $\mu=(n+1)/2$. Then the family of operators
$$Q_X:=Q_M|_{ X^\circ}:\dot{H}^{-\mu}(X)\to \overline{H}^{\mu}(X),$$ 
is a well defined family of Fredholm operators for all $R\in \C\setminus \{0\}$. For some $R_0\geq 0$, $Q_X(R)$ is invertible for all $R\in \Gamma_{\pi/(n+1)}(R_0)$. Moreover, the following holds:
\begin{enumerate}
\item[a)] The family of operators 
$$(Q_X(R):\dot{H}^{-\mu}(X)\to \overline{H}^{\mu}(X))_{R\in \C\setminus \{0\}},$$ 
depends holomorphically on $R\in \C\setminus \{0\}$. Moreover, we can extend the holomorphic family $(Q_X(R)^{-1}:\overline{H}^{\mu}(X)\to \dot{H}^{-\mu}(X))_{R\in \Gamma_{\pi/(n+1)}(R_0)}$ meromorphically to $R\in \C\setminus \{0\}$.
\item[b)] There are $C,R_0>0$ such that 
$$C^{-1}\|f\|_{\dot{H}^{-\mu}_{|R|}(X)}^2\leq \mathrm{Re}\langle f,Q_X(R) f\rangle_{L^2}\leq C\|f\|_{\dot{H}^{-\mu}_{|R|}(X)}^2,$$
for $R\in \Gamma_{\pi/(n+1)}(R_0)$ and $f\in \dot{H}^{-\mu}(X)$. In particular, for $R\in \Gamma_{\pi/(n+1)}(R_0)$, the norm ${ \mathrm{Re}}\langle \cdot, Q_X(R)\cdot \rangle_{L^2}$ is uniformly equivalent to the norm on $\dot{H}^{-\mu}_R(X)$.
\end{enumerate}
\end{thm}

\begin{proof}
We first prove part b). This is a direct consequence of adapting Theorem \ref{symbcor}, part b), to compactly supported  distributions that are $-\mu$-Sobolev regular and using that $\dot{H}^s_{|R|}(X)\subseteq H^s_{|R|}(M)$ is an isometric inclusion.

To prove part a), we note that for $R \in \mathbb{C}\setminus \{0\}$, $Q_M(R)$ is a lower-order perturbation of $Q_M(R_0)$ for any $R_0\gg0$. Therefore, the Rellich lemma implies that the quadratic form 
$$\mathfrak{q}_{Q,R}(u):=\langle u,Q(R) u\rangle_{L^2}\equiv\langle u,Q_X(R) u\rangle_{L^2},\quad u\in H^{-\mu}(X),$$
is a compact perturbation of $\mathfrak{q}_{Q,R_0}$. Therefore, the difference
$$Q_X(R)-Q_X(R_0):\dot{H}^{-\mu}(X)\to \overline{H}^{\mu}(X),$$ 
is a compact operator. Since part b) implies that $Q_X(R_0)$ is invertible for a large enough $R_0\gg0$, we conclude that $(Q_X(R):\dot{H}^{-\mu}(X)\to \overline{H}^{\mu}(X))_{R\in \C\setminus \{0\}}$ is a Fredholm family. 

It remains to prove the assertion for the inverse of $Q_X(R)$ The family of operators $(Q_X(R):\dot{H}^{-\mu}(X)\to \overline{H}^{\mu}(X))_{R\in \C\setminus \{0\}}$ is obtained from the holomorphic family of operators $(Q_M(R):H^{-\mu}_c(M)\to H^{\mu}_{\rm loc}(M))_{R\in \C\setminus \{0\}}$ (see Theorem \ref{symbcor}, part a) via inclusions and projections, so it is also holomorphic. As such, $(Q_X(R)^{-1}:\overline{H}^{\mu}(X)\to \dot{H}^{-\mu}(X))_{R\in \Gamma_{\pi/(n+1)}(R_0)}$ extends meromorphically to $\C\setminus \{0\}$ by the meromorphic Fredholm theorem (see \cite[Proposition 1.1.8]{leschhab}).
\end{proof}

Similarly to the ideas in Section \ref{remadinedlsosec}, we shall transfer the results of Theorem \ref{symbcorboundaryq} to the operator $\mathcal{Z}_X$ using property (MR). For a domain, let us make the notion of property (MR) more precise.

\begin{deef}
\label{sovodlwwowmbodu}
Let $X$ be a compact manifold with $C^0$-boundary and $\rd$ a distance function on $X$ such that $\rd^2$ is regular at the diagonal. For a sector $[1,\infty)\subseteq \Gamma\subseteq \C$, we say that $\rd$ has \emph{property (MR)} on $\Gamma$ if $(X,d)$ is isometrically embedded as a domain with smooth boundary in a manifold $M$ equipped with a distance function $\rd_M$, such that $\rd^2_M$ is regular at the diagonal and and for any $R\in \Gamma$, $L(R)$ extends to a continuous mapping $H^{-\mu}_c(M)\to H^{\mu}_{\rm loc}(M)$ with 
$$\|L(R)\|_{\dot{H}^{-\mu}(K)\to \bar{H}^{\mu}(K')}=O(\mathrm{Re}(R)^{-\infty}),\quad \mbox{as $\mathrm{Re}(R)\to +\infty$ in $\Gamma$},$$
for any compact subsets $K,K'\subseteq M$. 

If, for any compact subsets $K,K'\subseteq M$, the operator $L(R):\dot{H}^{-\mu}(K)\to \overline{H}^{\mu}(K')$ is compact for $R\in \Gamma$ and $\Gamma\ni R\mapsto L(R)\in \mathbb{K}(\dot{H}^{-\mu}(K), \overline{H}^{\mu}(K))$ is holomorphic in norm sense, we say that $\rd$ has \emph{property (SMR)}  on $\Gamma$.
\end{deef}
 
In the absence of a boundary, the definition of property (MR) for a manifold with boundary (Definition \ref{sovodlwwowmbodu}) is readily seen to be equivalent to property (MR) for a compact manifold (Definition \ref{sovodlwwowm}). The reader is encouraged to think of  the definition of property (MR) for a manifold with boundary as the distance function having ``property (MR) on a neighborhood of the manifold with boundary''. Let us consider two examples of distance functions with property (SMR).

\begin{example}[Domains in Riemannian manifolds with small diameter]
\label{domainexalaslsa1}
Assume that $X\subseteq M$ is a compact domain with $C^0$-boundary in a Riemannian manifold with geodesic distance $\rd_{{\rm geo},M}$. If the diameter of $X$ is strictly smaller than the injectivity radius of $M$, $\rd_{{\rm geo},M}^2$ is smooth on a neighborhood of $X$. The same argument as in Proposition \ref{sobregpropdomain} shows that the distance function $\rd_{\rm geo}:=\rd_{{\rm geo},M}|_X$ on $X$ has property (SMR) on $\C\setminus\{0\}$. In fact, in this case, $L\in \Psi^{-\infty}(X;\C_+)$ and $L(R)\in \Psi^{-\infty}(X)$ for any $R\in \C\setminus\{0\}$. 
\end{example}

\begin{example}[Submanifolds with boundary]
\label{domainexalaslsa2}
Assume that $X$ is a compact manifold with $C^0$-boundary embedded in a manifold $i:X\to W$ and $\rd_W$ is a distance function on $W$ such that the square $\rd_{W}^2$ is smooth on $W\times W$ and regular at the diagonal. This arises for instance for $W=\R^N$, $W=\mathbb{H}_{N,\R}$ or $W=\mathbb{H}_{N,\C}$, for some $N\in \N$, with their geodesic distance. 

The subspace distance function $\rd:X\times X\to [0,\infty)$, $\rd(x,y):=\rd_{W}(i(x),i(y))$ will then satisfy 
\begin{enumerate}
\item[i)] $\rd^2$ is smooth on $X\times X$ and regular at the diagonal.
\item[ii)] $L\in \Psi^{-\infty}(X;\C_+)$ and $L(R)\in \Psi^{-\infty}(X)$ for any $R\in \C\setminus\{0\}$. 
\item[iii)] $\rd$ has property (SMR) on $\C\setminus \{0\}$. 
\end{enumerate}
This follows by the same arguments as in Proposition \ref{sobregpropdomain} by choosing a submanifold $M\subseteq W$ in which $i(X)$ is a compact domain (with $C^0$-boundary).
\end{example}
 
For a distance function with property (MR), we tacitly assume that the manifold with $C^0$-boundary is embedded into the manifold $M$ implementing property (MR). The next result is proven exactly as Theorem \ref{firstofz} but using Theorem \ref{symbcorboundaryq} instead of Corollary \ref{cortombolstructureinversepsido}.

\begin{thm}
\label{firstofzx}
Let $X$ be a compact $n$-dimensional manifold with $C^0$-boundary and $\rd$ a distance function on $X$ with property (MR) on the sector $\Gamma$. Then there is an $R_0\geq 0$ such that 
$$\mathcal{Z}_X(R):\dot{H}^{-\mu}(X)\to \overline{H}^{\mu}(X),$$ 
is invertible for all $R\in \Gamma\cap \Gamma_{\pi/(n+1)}(R_0)$. Moreover, 
$$\mathcal{Z}^{-1}_X=Q^{-1}_X+\mathcal{R}_X,$$
where $Q^{-1}_X$ is the inverse of $Q_X$ (existing by Theorem \ref{symbcorboundaryq}) and $\mathcal{R}_X:\overline{H}^{\mu}(X)\to \dot{H}^{-\mu}(X)$ is a family of operators such that 
$$\|\mathcal{R}\|_{\overline{H}^{\mu}(X)\to \dot{H}^{-\mu}(X)}=O(\mathrm{Re}(R)^{-\infty}),\quad \mbox{as $\mathrm{Re}(R)\to +\infty$ in $\Gamma\cap \Gamma_{\pi/(n+1)}(R_0)$}.$$
Moreover, there is a $C>0$ such that 
\begin{equation}
\label{utforzgardinb}
C^{-1}\|f\|_{\dot{H}^{-\mu}_{|R|}(X)}^2\leq \mathrm{Re}\langle f,\mathcal{Z}_X(R) f\rangle_{L^2}\leq C\|f\|_{\dot{H}^{-\mu}_{|R|}(X)}^2,
\end{equation}
for $R\in \Gamma\cap \Gamma_{\pi/(n+1)}(R_0)$ and $f\in \dot{H}^{-\mu}(X)$. In particular, for $R\in \Gamma\cap \Gamma_{\pi/(n+1)}(R_0)$, $\mathrm{Re}\mathcal{Z}_X(R)$ is positive in form sense on $L^2(X)$ for the $H^{-\mu}$-norm. 
\end{thm}

The next result poses an obstruction to property (MR) for distance functions on manifolds with boundary. It is proven in the same way as Corollary \ref{posididiveove} but using Theorem \ref{firstofzx} instead of Theorem \ref{firstofz}.

\begin{cor}
\label{posididiveovex}
Assume that $\rd$ is a distance function on a manifold with $C^0$-boundary $X$ with property (MR) on $[1,\infty)$. Then there exists an $R_0\geq 0$ such that for all finite subsets $F\subseteq X$ the $|F|\times|F|$-matrix $(\e^{-R\rd(x,y)})_{x,y\in F}$ is positive definite for all $R>R_0$.
\end{cor}

Using Theorem \ref{symbcorboundaryq} instead of Theorem \ref{firstofz}, the next result is proven ad verbatim as Theorem \ref{symbcorz}.

\begin{thm}
\label{symbcorzx}
Let $X$ be an $n$-dimensional compact manifold with $C^0$-boundary and assume that distance function $\rd$ has property (SMR) on $\Gamma$. There is an $R_0\geq 0$ such that the operator
$$\mathcal{Z}_X(R):\dot{H}^{-\mu}(X)\to \overline{H}^{\mu}(X),$$ 
is a well defined Fredholm operator for all $R\in \Gamma$ and invertible for $R\in \Gamma\cap\Gamma_{\pi/(n+1)}(R_0)$. Moreover, the operator 
$\mathcal{Z}_X(R):\dot{H}^{-\mu}(X)\to \overline{H}^{\mu}(X)$  depends holomorphically on $R$ in $\Gamma$ and $\mathcal{Z}_X(R)^{-1}:\overline{H}^{\mu}(X)\to \dot{H}^{-\mu}(X)$ depends meromorphically on $R\in \Gamma$.
\end{thm}

Similarly to Corollary \ref{strongleledldforrn}, we deduce the following special case of Theorem \ref{symbcorzx}. 

\begin{cor}
Let $X$ be an $n$-dimensional compact manifold with $C^0$-boundary and assume that $\rd$ is a distance function satisfying that $\rd^2$ is smooth on $X\times X$ and regular at the diagonal (e.g. as in Example \ref{domainexalaslsa1} or \ref{domainexalaslsa2}). Then the family of operators 
$$(\mathcal{Z}_X(R):\dot{H}^{-\mu}(X)\to \overline{H}^{\mu}(X))_{R\in \C\setminus \{0\}},$$ 
depends holomorphically on $R\in \C\setminus \{0\}$. Moreover, for some $R_0$, we have a holomorphic family $(\mathcal{Z}_X(R)^{-1}:\overline{H}^{\mu}(X)\to \dot{H}^{-\mu}(X))_{R\in \Gamma_{\pi/(n+1)}(R_0)}$ that extends meromorphically to $R\in \C\setminus \{0\}$.
\end{cor}

\section{Structure of the inverse operator in the presence of a boundary}
\label{structurofinversesec}

Consider a compact manifold with boundary $X$ equipped with a distance function $\rd$ whose square is regular at the diagonal. As above, we set $\mu:=(n+1)/2$ where $n:=\dim(X)$. If $\rd$ has property (MR), Theorem \ref{firstofzx} ensures that computations for $Q_X$ relate to computations for $\mathcal{Z}_X$ up to a term of infinitely low order in $R$ (as $\mathrm{Re}(R)\to \infty$), so we focus on the operator $Q_X$. As proved in Theorem \ref{symbcorboundaryq} above, the localized operator $Q_X:\dot{H}^{-\mu}(X)\to \overline{H}^\mu(X)$ is an isomorphism for large enough $R$ in the sector $\Gamma_{\pi/(n+1)}$. We shall now describe the inverse of $Q_X$ in more precise terms under the assumption that the boundary is smooth. The inverse $Q_X^{-1}:\overline{H}^\mu(X)\to \dot{H}^{-\mu}(X)$ will be computed as a sum of 
\begin{itemize}
\item a pseudodifferential operator (with parameter) in the interior;
\item a composition of two \emph{mixed-regularity} pseudodifferential operator near the boundary, where the two factors are obtained from inverting a Wiener-Hopf factorization of the magnitude operator at the boundary; as well as
\item an error term that acts as order $2\mu$, mapping $\overline{H}^\mu(X)\to \dot{H}^{-\mu+1}(X)$, whose norm is $O(|R|^{-\infty})$ as $R\to \infty$.
\end{itemize} 
Asymptotically, only the two first terms play a role, and in the next section we compute the asymptotics of conditional expectations from the symbols of these first two terms. To describe the decomposition, we shall need further terminology.

\subsection{Mixed-regularity symbols and Sobolev spaces}
\label{subsec:mexedallada}

\begin{deef}
Let $s,t\in \R$ and $n\in \N_{>0}$. Write coordinates in $\R^n$ as $\xi=(\xi',\xi_n)\in \R^{n-1}\times \R$. We define the Sobolev space of mixed-regularity $(s,t)$ as 
$$H^{s,t}(\R^n):=\left\{f\in \mathcal{S}'(\R^n): \int_{\R^n} \langle\xi\rangle^s\langle\xi'\rangle^t|\hat{f}(\xi)|^2\rd \xi<\infty\right\}.$$
If $\Omega\subseteq \R^n$ is a domain (so $\Omega=\overline{\Omega^\circ}$), we define 
$$\dot{H}^{s,t}(\Omega):=\{f\in H^{s,t}(\R^n): \mathrm{supp}(f)\subseteq \Omega\},\quad\mbox{and}\quad \overline{H}^{s,t}(\Omega):=H^{s,t}(\R^n)/\dot{H}^{s,t}(\overline{\Omega^c}).$$

We also define $\dot{H}^{s,t}_c(\Omega)$ and $\overline{H}^{s,t}_c(\Omega)$ as the elements with compact support. The local Sobolev spaces of mixed-regularity are defined as 
\begin{align*}
\dot{H}^{s,t}_{\rm loc}(\Omega)&:=\{f\in \mathcal{S}'(\R^n):\chi f\in \dot{H}^{s,t}(\Omega)\forall \chi\in C^\infty_c(\R^n)\},\\
\overline{H}^{s,t}_{\rm loc}(\Omega)&:=H^{s,t}_{\rm loc}(\R^n)/\dot{H}^{s,t}_{\rm loc}(\overline{\Omega^c}).
\end{align*}
\end{deef}

We note that $\dot{H}^{s,t}(\Omega)\subseteq H^{s,t}(\R^n)$ is a closed subspace. We call the quotient mapping ${ H}^{s,t}(\R^n)\to \overline{H}^{s,t}(\Omega)$ the restriction mapping. For large enough $s>0$, we can identify $\dot{H}^{s,t}(\Omega)$ with a subspace of $\overline{H}^{s,t}(\Omega)$ (but not for small $s<0$). A standard computation with Fourier transforms shows that the the identity operator induces continuous mappings $\overline{H}^{s,t}(\Omega)\to \overline{H}^{s',t'}(\Omega)$ and $\dot{H}^{s,t}(\Omega)\to \dot{H}^{s',t'}(\Omega)$ if and only if $s\geq s'$ and $s+t\geq s'+t'$ (which is locally compact if and only if $s> s'$ and $s+t> s'+t'$).

\begin{deef}
Let $u,m\in \R$ and $n\in \N_{>0}$. Let $\Gamma\subseteq \C$ be a sector and $U\subseteq \R^n$ an open subset. Write coordinates in $\R^n$ as $\xi=(\xi',\xi_n)\in \R^{n-1}\times \R$. We say that $a\in C^\infty(U\times \R^n\times\Gamma)$ is a symbol with parameter of mixed-regularity $(u,m)$ if for any compact $K\subseteq U$, $k\in \N$, $\alpha\in \N^n$ and $\beta\in \N^{n}$ there is a constant $C>0$ such that 
$$\sup_{x\in K}\left|\partial_{\xi_n}^k\partial_{(\xi',R)}^\beta\partial_x^\alpha a(x,\xi',\xi_n,R)\right|\leq {C}\langle (\xi,R)\rangle^{u-k}\langle (\xi',R)\rangle^{m-|\beta|},$$
for all $(\xi,R)\in \R^{n+1}$. We let $S^{u,m}(U;\Gamma)$ denote the space of symbols with parameter of mixed-regularity $(u,m)$. We set $S^{u,-\infty}(U;\Gamma):=\cap_{m\in \R}S^{u,m}(U;\Gamma)$.

For $a\in S^{u,m}(U;\Gamma)$ we define 
$$Op(a):C^\infty_c(U)\to C^\infty(U), \quad Op(a)f(x):=\frac{1}{(2\pi)^n}\int_{\R^n}a(x,\xi,R) \hat{f}(\xi)\rd \xi.$$
Let $\Psi^{u,m}(U;\Gamma)$ denote the linear space of operators $C^\infty_c(U)\to C^\infty(U)$ spanned by $\{Op(a): a\in S^{u,m}(U;\Gamma)\}$ and smoothing operators with parameter. We set $\Psi^{u,-\infty}(U;\Gamma):=\cap_{m\in \R}\Psi^{u,m}(U;\Gamma)$.
\end{deef}

\begin{example}
Assume that $0\notin \Gamma$. Suppose that $a(x,\xi,R)=b(x,\xi',R)(\xi_n-h(x,\xi',R))^u$ where $b$ is a homogeneous symbol with parameter of order $m$ and $h$ is a homogeneous symbol with parameter of order $1$. A short computation shows that $\partial_{\xi_n}^k\partial_{(\xi',R)}^\beta\partial_x^\alpha a$ is a sum of terms of the form $\tilde{b}(x,\xi',R)(\xi_n-h(x,\xi',R))^{u-k-l}$ where $\tilde{b}$ is homogeneous of order $m-|\beta|+l$. Therefore $a\in S^{u,m}(U;\Gamma)$, and in fact $a\in S^{u_0,m+u-u_0}(U;\Gamma)$ for any $u_0\leq  u$.
\end{example}

We note that since differentiation in the $(\xi',R)$-direction improves the order of decay, we have for $a\in S^{u,m}(U;\Gamma)$ and $f\in C^\infty_c(U)$ that
$$Op(a)f(x):=\frac{1}{(2\pi)^n}\int_{\R^n\times \R^{n-1}}a(x,\xi,R)\mathcal{F}_{y_n\to \xi_n}f(y',\xi_n)\e^{i\xi_nx_n+i\xi'(x'-y)'}\mathrm{d}y' \rd \xi,$$
as an oscillatory integral. Moreover, we can conclude the following result from standard techniques of oscillatory integrals (cf. \cite[Chapter I.1]{shubinbook}).  

\begin{prop}
\label{pseudolocaformfixed}
Assume that $A\in \Psi^{u,m}(U;\Gamma)$. Then for any $\chi,\chi'\in C^\infty(U)$ with $\chi \chi'=0$ it holds that $\chi A\chi'\in \Psi^{u,-\infty}(U;\Gamma)$. In particular, if $\chi,\chi'\in C^\infty_c(U)$ satisfies $\chi \chi'=0$ then $\|\chi A\chi'\|_{H^s\to H^{s'}}=O(|R|^{-\infty})$ as $R\to \infty$ for all $s\geq s'+u$.
\end{prop}

The last statement of the proposition follows from the next theorem.

\begin{thm}
\label{continuityandwhatnot}
Let $u,m,s,t,s',t'\in \R$ and $n\in \N_{>0}$. Let $A\in \Psi^{u,m}(\R^n;\Gamma)$. Then $A$ extends to a continuous operator 
$$A:H^{s,t}_c(\R^n)\to H^{s',t'}_{\rm loc}(\R^n),$$
if $s\geq s'+u$ and $t\geq t'+m$. In this case, we have for any $\chi,\chi'\in C^\infty_c(\R^n)$ that there exists a $C=C(s,s',A,\chi,\chi')>0$ 
$$\|\chi A\chi'\|_{H^{s,t}(\R^n)\to H^{s',t'}(\R^n)}\leq C(1+|R|)^{t'-t+m}.$$
\end{thm}

\begin{proof}
The first part follows from the Calder\'{o}n-Vaillancourt theorem. The second part follows from noting that Calder\'{o}n-Vaillancourt's theorem proves the case $t'+m=t$ and for $A\in \Psi^{u,m}(\R^n;\Gamma)$ compactly supported, then $(R^2+\Delta')^{(t-t'-m)/2}A\in \Psi^{u,t-t'}(\R^n;\Gamma)$ (where $\Delta'$ denotes the Laplacian in the $x'$-direction).
\end{proof}

\begin{deef}
\label{genasmdef}
Let $u\in \R$. Consider a sequence $(a_j)_{j\in \N}$ of mixed-regularity symbols with $a_j\in S^{u,m_j}(U;\Gamma)$ for a sequence $m_j\to -\infty$. Set $m:=\max_j m_j$. If $a\in S^{u,m}(U;\Gamma)$ satisfies that for any $N$, there is an $M$ such that $a-\sum_{j=0}^Ma_j\in S^{u,-N}(U;\Gamma)$, we write 
$$a\sim \sum_{j=0}^\infty a_j,$$
and call $a$ the asymptotic sum of $(a_j)_{j\in \N}$.
\end{deef}

\begin{prop}
Let $u\in \R$. For any sequence $(a_j)_{j\in \N}$ of mixed symbols with $a_j\in S^{u,m_j}(U;\Gamma)$ for a sequence $m_j\to -\infty$, the asymptotic sum $a\sim \sum_{j=0}^\infty a_j$ exists in $S^{u,m}(U;\Gamma)$, where $m:=\max_j m_j$. The asymptotic sum is uniquely determined modulo $S^{u,-\infty}(U;\Gamma)$.
\end{prop}

\begin{proof}
The proof of this proposition is carried out ad verbatim as in \cite[Proposition 18.1.3]{horIII} upon replacing the $\xi$ in \cite{horIII} with $(\xi',R)$.
\end{proof}

Again, using that differentiation in the $(\xi,R)$-direction improves the order of decay, we can conclude several results from the standard situation (cf. \cite[Chapter I]{shubinbook}). For instance, the analogue of \cite[Theorem 3.1, Chapter I.3]{shubinbook} extends modulo $\Psi^{u,-\infty}$ to $\Psi^{u,m}$ which implies asymptotic expansions of products and adjoints. 

\begin{prop}
\label{compososfofrmm}
Let $A\in \Psi^{u,m}(U;\Gamma)$ and $B\in \Psi^{u',m'}(U;\Gamma)$ be mixed-regularity pseudodifferential operators out of which at least one is properly supported. Then $AB\in \Psi^{u+u',m+m'}(U;\Gamma)$ is a mixed-regularity pseudodifferential operator. Moreover, if $a\in S^{u,m}(U;\Gamma)$ and $b\in S^{u',m'}(U;\Gamma)$ are symbols with $A-Op(a)$ and $B-Op(b)$ being smoothing with parameter, then $AB-Op(c)$ is smoothing with parameter where $c\in S^{u+u',m+m'}(U;\Gamma)$ is uniquely determined modulo $S^{u+u',-\infty}(U;\Gamma)$ as the asymptotic sum
$$c\sim \sum_{\alpha\in \N^n} \frac{1}{\alpha!} \partial_\xi^\alpha aD^\alpha_x b.$$
\end{prop}

{ 
\begin{thm}
\label{restrictingtohalfspaces}
Let $u,m,s,t,s',t'\in \R$ and $n\in \N_{>0}$. Let $a\in S^{u,m}(\R^n;\Gamma)$ be compactly supported in the $x$-direction and assume that $a$ extends to a holomorphic function in $\mathrm{Im}(\xi_n)<0$. Set $A:=Op(a)$. Then $A$ restricts to and induces, respectively, continuous operators 
$$\dot{A}:\dot{H}^{s,t}(\R^n_+)\to \dot{H}^{s',t'}(\R^n_+)\quad\mbox{and}\quad \overline{A}:\overline{H}^{s,t}(\R^n_-)\to \overline{H}^{s',t'}(\R^n_-),$$
if $s\geq s'+u$ and $t\geq t'+m$. 
\end{thm}

\begin{proof}
Note that $\overline{A}$ is well defined as the operator induced by $A:H^{s,t}(\R^n)\to H^{s',t'}(\R^n)$ as soon as $A$ preserves supports in $\R^n_+$. Because the same holds for $\dot{A}:\dot{H}^{s,t}(\R^n_+)\to \dot{H}^{s',t'}(\R^n_+)$, we only need to show that $A$ preserves supports in $\R^n_+$. By standard density arguments, it suffices to prove that $A$ preserves supports in $\R^n_+$ when applied to test functions.

Using Proposition \ref{compososfofrmm}, we note that upon replacing $A$ with the operator 
$$A\ Op((1+i\xi_n+|(\xi',R)|)^{-m_0}(1+|(\xi',\xi_n,R)|^2)^{-u_0}),$$
for sufficiently large $u_0$ and $m_0$, we can assume that $a$ is integrable in $\xi$. Set $K_R(x,z):=\mathcal{F}_{\xi\to z} a(x,\xi,R)$. Then for any test function $f$,
$$Af(x)=(2\pi)^n\int_{\R^n} K_R(x,x-y)f(y)\mathrm{d}y.$$
The Paley-Wiener theorem combined with the assumption that $a$ extends to a holomorphic function in $\mathrm{Im}(\xi_n)<0$ implies that $K_R$ is supported in $\R^n\times \R^n_+\times \Gamma$. We conclude that if $\mathrm{supp}(f)\subseteq \R^n_+$, then $Af(x)=(2\pi)^n\int_{\R^n_+} K_R(x,x-y)f(y)\mathrm{d}y$ is also supported in $\R^n_+$. 
\end{proof}

\begin{remark}
Under the assumptions of Theorem \ref{restrictingtohalfspaces}, Theorem \ref{continuityandwhatnot} implies that for any $s$ and $s'$ with $s\geq s'+u$ there is a $C>0$ such that for $t\geq t'+m$
$$\| A\|_{\dot{H}^{s,t}(\R^n_+)\to \dot{H}^{s',t'}(\R^n_+)}\leq C(1+|R|)^{t'-t+m}\quad\mbox{and}\quad\| A\|_{\overline{H}^{s,t}(\R^n_-)\to \overline{H}^{s',t'}(\R^n_-)}\leq C(1+|R|)^{t'-t+m}.$$
\end{remark}}

\subsection{Wiener-Hopf factorization of $Q_X$ near the boundary}

Using the machinery of the previous subsection, we shall now factorize the operator $Q_X$ near the boundary into factors that extend holomorphically into the lower, respectively upper half-plane, and use Theorem \ref{restrictingtohalfspaces} to (near the boundary) invert these individual factors as operators $\dot{H}^{-\mu}(X)\to L^2(X)$ and $L^2(X)\to \overline{H}^\mu(X)$, respectively. The reader should recall the structure of the full symbol of $Q$ from Theorem \ref{firstfirstofz}. We shorten the notation and write $C^j$ for $C^j_{\rd^2}$, where $C^j_{\rd^2}$ are the Taylor coefficients of $\rd^2$ as in Equation \eqref{taylorexpamdmd}.

As above, we consider a compact manifold with boundary $X$ and a distance function $\rd$ on $X$ whose square is regular at the diagonal (cf. Definition \ref{regularlaododal}). We tacitly fix a manifold $M$ containing $X$ as a smooth compact domain to which $\rd$ extends as a distance function whose square is regular at the diagonal. In particular, we can Taylor expand $\rd^2$ near any point in the diagonal as in Equation \eqref{taylorexpamdmd} and its Taylor coefficients enter the full symbol of $Q$ as in Theorem \ref{firstfirstofz}. To study the behavior at the boundary, we first reduce to the model case that $X=\partial X \times [0,\infty)$, as a domain in $\partial X \times \R$. We remark that $\partial X\times [0,\infty)$ is not compact, but we shall later on only use the constructed operators in a form localized to near the compact boundary.

\begin{prop}
\label{extendingtocylinder}
Let $X$ be a compact manifold with boundary with a distance function $\rd$ whose square is regular at the diagonal, embedded into a manifold $M$ as in the preceding paragraph. Consider the compact manifold $Y=\partial X$ and choose a tubular neighborhood $U\subseteq M$ of $\partial X$ and a diffeomorphism $\varphi:U\to Y\times (-1,1)$, with $\partial X=\varphi^{-1}(Y\times \{0\})$. Then there exists a classical elliptic pseudodifferential operator with parameter $Q^\partial\in \Psi^{-n-1}_{\rm cl}(Y\times \R;\C_+)$ and a number $\epsilon>0$ such that 
\begin{itemize}
\item $Q^\partial$ is translation invariant outside a compact subset in the sense that there exists a $t_0>0$ such that if $f\in C^\infty_c(Y\times \R)$ is supported in $\{(y,t): \pm t>t_0\}$ then $[Q^\partial f](\cdot\mp s)=Q^\partial [f(\cdot\mp s)]$ for all $s>0$.
\item For all $f\in C^\infty_c(Y\times (-\epsilon,\epsilon))$ it holds that $Q(\varphi^*f)$ is supported in $U$ and 
$$Q(\varphi^*f)=\varphi^*(Q^\partial f).$$
\item the principal symbol of $Q^\partial$ is given by 
$$\sigma_{-n-1}(x,\xi,R)=c(R^2+g(\xi,\xi))^{-\mu},$$
where $g$ is a Riemannian metric on $Y\times \R$ which is translation invariant in the $\R$-direction outside a compact and coincides with $\varphi^*g_{\rd^2}$ on $Y\times (-1,1)$.
\end{itemize}
\end{prop}

\begin{proof}
Construct $g$ from interpolating between $g_{\rd^2}$ near $0$ and something at infinity. Construct $Q^\partial$ from interpolation along the real line between $Q$ near $0$ and $(R^2+\Delta_{g})^{-\mu}$ at infinity. The second part follows from that $Q$ by construction has small propagation.
\end{proof}

Henceforth, we shall fix a choice of $Q^\partial$ and $g$ as in Proposition \ref{extendingtocylinder}.

\begin{remark}
To fix a choice of a diffeomorphism $\varphi:U\to Y\times (-1,1)$ is (up to a self-diffeomorphism of $Y$) equivalent to choosing a vector field defined near $Y=\partial X$ which is transversal to the boundary. This choice of vector field, or equivalently, the last entry of $\varphi:U\to Y\times (-1,1)$, gives rise to a coordinate that we denote by $x_n:U\to (-1,1)$. We remark that it is always possible to use the transversal vector field to be the metric normal to the boundary, in which case we have that on $\partial X$:
$$g_{\rd^2}=\rd x_n^2+g_{\partial X,\rd^2},$$
where $g_{\partial X,\rd^2}$ is the induced Riemannian metric on $\partial X$. For computational purposes, it becomes clumsy to restrict to the the case when the transversal vector field is orthogonal to the boundary but for some considerations it simplifies the formulas. 

Following the notation of Subsection \ref{subsec:mexedallada}, we write $x_n$ for this transversal coordinate and $x'$ for coordinates on $\partial X$. Similarly, $\xi_n$ denotes the cotangent variable in the transversal direction and $\xi'$ denotes the cotangent variables along $\partial X$.
\end{remark}

We let $q^\partial\in S^{-2\mu}_{\rm cl}(Y\times \R;\Gamma)$ denote the full symbol of $Q^\partial$. We note that $q^\partial\sim \sum_{j=0}^\infty q^\partial_j$ in $S^{-2\mu}(Y\times \R;\C)$ where each $q^\partial_j\in S^{-2\mu-j}(Y\times \R;\C)$ is a homogeneous symbol in $(\xi,R)$ of order $-2\mu-j=-n-1-j$ and near $t=0$, we have in any local coordinates on $Y$ that $q^\partial_j=q_j$ where $q_j$ is computed as in Theorem \ref{firstfirstofz} using the coordinates induced from $Y$ and $\varphi$. 

\begin{prop}
\label{rootsofmetric}
There are unique homogeneous degree $1$ symbols 
$$h_\pm=h_\pm(x,\xi',R)\in S^1(T^*Y\times \R, Y\times \R;\C),$$ 
that determine the complex solutions to the equation $R^2+g_x((\xi',\xi_n),(\xi',\xi_n))=0$ for fixed $(x,\xi',R)$ with $\pm \mathrm{Im}(\xi_n)>0$. Furthermore, there is a unique function $h_0\in C^\infty(Y\times\R,\R_{>0})$ such that 
$$R^2+g(\xi,\xi)=h_0(x)(\xi_n-h_+(x,\xi',R))(\xi_n-h_-(x,\xi',R)).$$
Moreover, we have that 
$$h_\pm(x,\xi',R)=-\frac{\xi'(b(x))}{h_0(x)}\pm i\frac{\sqrt{R^2+g_Y(\xi',\xi')-(\xi'(b))^2}}{\sqrt{h_0(x)}},$$
for suitable $b$ and $g_Y$ determined from the metric.
\end{prop}

\begin{proof}
It is not hard to see that $h_0$ and $h_\pm$ are well defined and unique, but let us construct them explicitly. We can decompose 
\begin{equation}
\label{decompsomdmee}
g=\begin{pmatrix} h_0& b\\ b^T& g_Y\end{pmatrix},
\end{equation}
where $g_Y$ is a metric on $Y$ on each slice, and $b\in C^\infty(Y\times \R, TY)$. Since $g$ is translation invariant outside a compact, $h_0$, $b$ and $g_Y$ are translation invariant outside a compact. We have that 
$$R^2+g(\xi,\xi)=R^2+h_0\xi_n^2+2\xi'(b)\xi_n+g_Y(\xi',\xi').$$
We see that $h_0$ in the above definition is the $h_0$ in Equation \eqref{decompsomdmee} and that the complex roots are as prescribed. The proposition follows.
\end{proof}

\begin{thm}
\label{symbofdldl}
Let $q^\partial\in S^{-2\mu}_{\rm cl}(Y\times \R;\Gamma)$ be as above. Then there exists $q^\partial_\pm\in S^{-\mu,0}(Y\times\R;\Gamma)$ that are translation invariant outside a compact such that 
\begin{enumerate}
\item $q^\partial_\pm\in S^{-\mu,0}(Y\times\R;\Gamma)$ admits asymptotic expansions (in $S^{-\mu,0}$ in the sense of Definition \ref{genasmdef} on page \pageref{genasmdef})
$$q^\partial_\pm\sim \sum_{j=0}^\infty q^\partial_{\pm, j},$$
where $q^\partial_{\pm,0}\in S^{-\mu,0}$ and for $j>0$, 
$$q^\partial_{\pm, j}(x,\xi,R)=\sum_{k=-1}^{j-1} b_{\pm, j,k}(x,\xi',R) (\xi_n-h_\pm(x,\xi',R))^{-\mu-j+k}\in S^{-\mu-1,-j+1},$$ 
where $b_{\pm, j,k}$ is homogeneous of degree $-k$ in $(\xi',R)$ and can be computed by an iterative scheme of partial fraction decompositions as a homogeneous rational function in derivatives of $h_0$, $h_+$ and $h_-$. The first terms are given by 
\begin{align}
\label{qpmzerodefedd}
q^\partial_{+,0}=&n!\omega_n(\xi_n-h_+)^{-\mu},\\
\nonumber
q^\partial_{-,0}&= h_0^{-\mu}(\xi_n-h_-)^{-\mu},
\end{align}
and $q_{\pm,1}$ are computed in Proposition \ref{compuatontaad} below.
\item The mixed-regularity symbols $q^\partial_\pm\in S^{-\mu,0}(Y\times\R;\Gamma)$ and $q^\partial_{\pm, j}\in S^{-\mu-1,-j+1}$ admit holomorphic extensions to $\mp \mathrm{Im}(\xi_n)>0$.
\item It holds that 
$$q^\partial=\sum_{\alpha}\frac{1}{\alpha!} \partial_\xi^\alpha q_{-}^\partial D^\alpha_x q_+^\partial \mod S^{-2\mu,-\infty}.$$
\end{enumerate}
\end{thm}

\begin{proof}
Let us first massage the statements of the theorem. We want to construct $q^\partial_\pm\sim \sum_{j=0}^\infty q^\partial_{\pm, j}\in S^{-\mu,0}$ admitting holomorphic extensions to $\mp \mathrm{Im}(\xi_n)>0$ and satisfying $q^\partial=\sum_{\alpha}\frac{1}{\alpha!} \partial_\xi^\alpha q_{-}^\partial D^\alpha_x q_+^\partial \mod S^{-2\mu,-\infty}$. We remark that to ensure item (2), i.e. the holomorphic extension of $q^\partial_\pm$, it suffices to construct each $q_{\pm, j}^\partial $ so that it admits a holomorphic extension to $\mp \mathrm{Im}(\xi_n)>0$. We also note that the requirement on the composition is equivalent to 
\begin{equation}
\label{pmcompineachedgeef}
q^\partial_j=\sum_{k+l+|\alpha|=j}\frac{1}{\alpha!} \partial_\xi^\alpha q_{-,k}^\partial D^\alpha_x q_{+,l}^\partial.
\end{equation}
We take the formula \eqref{qpmzerodefedd} as a definition, and note that $q^\partial_{\pm,0}$ satisfy the structural statement in item (1), extends holomorphically to  $\mp \mathrm{Im}(\xi_n)>0$ and $q^\partial_0=q^\partial_{-,0}q^\partial_{+,0}$. Using an idea described in \cite{hornotes}, Equation \eqref{pmcompineachedgeef} is for $j>0$ equivalent to 
$$\frac{q^\partial_{+,j}}{q^\partial_{+,0}}+\frac{q^\partial_{-,j}}{q^\partial_{-,0}}=\frac{q^\partial_j}{q^\partial_0}-\frac{1}{q^\partial_0}\sum_{\substack{k+l+|\alpha|=j\\k,l<j}}\frac{1}{\alpha!} \partial_\xi^\alpha q_{-,k}^\partial D^\alpha_x q_{+,l}^\partial$$

We proceed by induction. Assume that we have constructed $q_{\pm, k}$ for $k<j$ satisfying the statements of items (1), (2), and (3) in the relevant degrees. Using Lemma \ref{firstofzcor} and item (1) for $q_{\pm, k}$ for $k<j$, we can use Lemma \ref{lem:xipf} to uniquely partial fraction decompose 
$$\frac{q^\partial_j}{q^\partial_0}-\frac{1}{q^\partial_0}\sum_{\substack{k+l+|\alpha|=j\\k,l<j}}\frac{1}{\alpha!} \partial_\xi^\alpha q_{-,k}^\partial D^\alpha_x q_{+,l}^\partial=\mathfrak{q}_{+,j}+\mathfrak{q}_{-,j},$$
where $\mathfrak{q}_{\pm,j}$ by Proposition \ref{maxesforidkd} and Theorem \ref{firstfirstofz} takes the form
$$\mathfrak{q}_{\pm, j}(x,\xi,R)=\sum_{k=-1}^{j-1} \mathfrak{b}_{\pm, j,k}(x,\xi',R) (\xi_n-h_\pm(x,\xi',R))^{-j+k}\in S^{-1,-j+1},$$
where $\mathfrak{b}_{\pm, j,k}$ is homogeneous of degree $-k$ in $(\xi',R)$ and can be explicitly computed from the results of Appendix \ref{partialappa} and Theorem \ref{firstfirstofz}. We now define 
$$q^\partial_{\pm,j}:=q^\partial_{\pm,0}\mathfrak{q}_{\pm, j},$$
and note that it by construction satisfies items (1), (2), and (3) in the relevant degrees. 

\end{proof}

\begin{remark}
As noted in Proposition \ref{rootsofmetric}, the symbols $h_+$ and $h_-$ are directly determined from the metric and the choice of transversal to the boundary. Moreover, as the proof of Theorem \ref{symbofdldl} shows, each of the symbols $b_{\pm, j,k}=b_{\pm, j,k}(x,\xi',R)$ depends
\begin{enumerate}
\item polynomially on $(C^{(\gamma)}_G)_{\gamma\in \cup_{k\leq j}I_k}$ and its derivatives contracted by $g_G$ and $\iota_ng_G$ and its derivatives 
\item polynomially on $h_+$, $h_-$ and rationally on $h_+-h_-$ and $h_0^{-1/2}$,
\end{enumerate}
The total degree is $j$ where each $C^{(\gamma)}_G$, $\gamma \in I_k$, has degree $k$, the metric and $h_0$ have degree zero, $h_+$ and $h_-$ have degree $1$ and $x$-derivatives increase the order by $1$. 
\end{remark}

{ 
\begin{prop}
\label{compuatontaad}
For $n>1$, the terms $q_{\pm,1}$ appearing in the expansion of $q_\pm$ in Theorem \ref{symbofdldl} are given by:
\begin{align*}
q^\partial_{+,1}=&n!\omega_n\mathfrak{a}_{0,+}(x,\xi',R)(\xi_n-h_+)^{-\mu-1}+n!\omega_n\mathfrak{a}_{1,+}(x,\xi',R)(\xi_n-h_+)^{-\mu-2},\\
q^\partial_{-,1}=&\mathfrak{a}_{0,-}(x,\xi',R)h_0^{-\mu}(\xi_n-h_-)^{-\mu-1}+\mathfrak{a}_{1,-}(x,\xi',R)h_0^{-\mu}(\xi_n-h_-)^{-\mu-2},
\end{align*}
where  
\begin{align*}
\mathfrak{a}_{0,+}(x,\xi',R)=&-\frac{3i\mathfrak{c}_{1,n} (n^2-1)}{n!\omega_n}\left(C^3(x,g\otimes \iota_ng)h_+ + C^3(x,g\otimes\iota_{\xi'}g)\right)\frac{h_0^{-1}}{h_+-h_-}+\\
\nonumber
&+\frac{i\mathfrak{c}_{1,n}(n+3)_{3,-2}}{{n!}\omega_n}\frac{h_0^{-2}}{(h_+-h_-)^3}\\
\nonumber
& \qquad\qquad\bigg[C^3(x, \iota_n g\otimes \iota_n g\otimes \iota_n g)h_+^2(h_+-3h_-) - 6h_+h_-C^3(x, \iota_n g\otimes \iota_n g\otimes \iota_{\xi'} g) -\\
\nonumber
&\qquad\qquad\qquad- 3(h_++h_-)C^3(x, \iota_n g\otimes \iota_{\xi'} g\otimes \iota_{\xi'} g) - 2C^3(x, \iota_{\xi'} g\otimes \iota_{\xi'} g\otimes \iota_{\xi'} g) \bigg]-\\
\nonumber
&+\frac{i(n+1)^2}{4}\frac{(\nabla_{\xi'}h_-\cdot\nabla_{x'}h_+ - \partial_{x_n}h_+)}{h_+-h_-},
\end{align*}
\begin{align*}
\mathfrak{a}_{1,+}(x,\xi',R)=&\frac{i\mathfrak{c}_{1,n}(n+3)_{3,-2}}{{n!}\omega_n}\frac{h_0^{-2}}{(h_+-h_-)^2}
\\
\nonumber
&\qquad\qquad\bigg[C^3(x, \iota_n g\otimes \iota_n g\otimes \iota_n g)h_+^3 + 3h_+^2C^3(x, \iota_{\xi'} g\otimes \iota_n g\otimes \iota_n g) + \\
\nonumber
&\qquad\qquad\qquad\qquad\qquad+3h_+C^3(x, \iota_{\xi'} g\otimes \iota_{\xi'} g\otimes \iota_n g)+ C^3(x, \iota_n g\otimes \iota_{\xi'} g\otimes \iota_{\xi'} g)\bigg],
\end{align*}
\begin{align*}
\mathfrak{a}_{0,-}(x,\xi',R)=&
\frac{3i\mathfrak{c}_{1,n} (n^2-1)}{n!\omega_n}\left(C^3(x,g\otimes \iota_ng)h_- + C^3(x,g\otimes\iota_{\xi'}g)\right)\frac{h_0^{-1}}{h_+-h_-}-\\
\nonumber
&-\frac{i\mathfrak{c}_{1,n}(n+3)_{3,-2}}{{n!}\omega_n}\frac{h_0^{-2}}{(h_+-h_-)^3}\\
\nonumber
& \qquad\bigg[C^3(x, \iota_n g\otimes \iota_n g\otimes \iota_n g)h_-^2(h_--3h_+) - 6h_+h_-C^3(x, \iota_n g\otimes \iota_n g\otimes \iota_{\xi'} g) -\\
\nonumber
&\qquad\qquad\qquad- {3}(h_++h_-)C^3(x, \iota_n g\otimes \iota_{\xi'} g\otimes \iota_{\xi'} g) - 2C^3(x, \iota_{\xi'} g\otimes \iota_{\xi'} g\otimes \iota_{\xi'} g) \bigg]+\\
\nonumber
&-\frac{i(n+1)^2}{4}\frac{(\nabla_{\xi'}h_-\cdot\nabla_{x'}h_+ - \partial_{x_n}h_+)}{h_+-h_-},
\end{align*}
\begin{align*}
\mathfrak{a}_{1,-}(x,\xi',R)=&\frac{i\mathfrak{c}_{1,n}(n+3)_{3,-2}}{{n!}\omega_n}\frac{h_0^{-2}}{(h_+-h_-)^2}
\\
\nonumber
&\qquad\qquad\bigg[C^3(x, \iota_n g\otimes \iota_n g\otimes \iota_n g)h_-^3 + 3h_-^2C^3(x, \iota_{\xi'} g\otimes \iota_n g\otimes \iota_n g) + \\
\nonumber
&\qquad\qquad\qquad\qquad+3h_-C^3(x, \iota_{\xi'} g\otimes \iota_{\xi'} g\otimes \iota_n g)+ C^3(x, \iota_n g\otimes \iota_{\xi'} g\otimes \iota_{\xi'} g)\bigg].
\end{align*}
\end{prop}

The proof is computational and based on Lemma \ref{lem:xipf} using the explicit forms in Corollary \ref{compofodemdmdm}. The details can be found in the arXiv version of this paper \cite{gimpgoffloucaarXiv}.}

\begin{lem}
The operators $Q^\partial_\pm:=Op(q^\partial_\pm)\in \Psi^{-\mu,0}(Y\times \R;\Gamma)$ satisfy that 
\begin{enumerate}
\item $Q-Q_-Q_+\in \Psi^{-2\mu,-\infty}(Y\times \R;\Gamma)$;
\item $Q_+$ restricts to a well defined operator $\dot{H}^{-\mu}(Y\times[0,\infty)\to L^2(Y\times [0,\infty))$ which is invertible for large $R$;
\item $Q_-$ restricts to a well defined operator $L^2(Y\times[0,\infty))\to \overline{H}^\mu(Y\times [0,\infty))$ which is invertible for large $R$.
\end{enumerate}
\end{lem}

\begin{proof}
Part (1) follows from Proposition \ref{compososfofrmm} and Theorem \ref{symbofdldl}. Parts (2) and (3) follow from Theorem \ref{continuityandwhatnot} and Theorem \ref{restrictingtohalfspaces}.
\end{proof}

\begin{deef}
\label{thewopsdef}
Define $w_{\pm,j}\in S^{\mu,-j}$ inductively by 
$$w_{\pm,0}(x,\xi,R):=(q^\partial_{\pm,0})^{-1}=
\begin{cases} 
\frac{1}{n!\omega_n} (\xi_n-h_+(x,\xi',R))^\mu, \;&\mbox{for $+$},\\
{}\\
h_0(x)^\mu(\xi_n-h_-(x,\xi',R))^\mu, \;&\mbox{for $-$},\end{cases}$$
and then 
$$w_{\pm,j}:=-w_{\pm, 0}\sum_{k+l+|\alpha|=j, \, l<j}\frac{1}{\alpha!} \partial_\xi^\alpha q_{\pm, k}^\partial D_x^\alpha w_{\pm,l}.$$
We also define $w_\pm:=\sum_j w_{\pm, j}\in S^{\mu,0}(Y\times \R;\Gamma)$ and $W_\pm \in \Psi^{\mu,0}(Y\times \R;\Gamma)$ is defined as a properly supported modification of $Op(w_\pm)$ with the same full symbol which is translation invariant outside a compact subset.
\end{deef}

\begin{lem}
\label{wsymdexcpsps}
Let $w_\pm\in S^{\mu,0}(Y\times \R;\Gamma)$ be as above. Then
\begin{enumerate}
\item The asymptotic expansion (in $S^{\mu,0}$ in the sense of Definition \ref{genasmdef} on page \pageref{genasmdef}) of $w_\pm\in S^{\mu,0}(Y\times\R;\Gamma)$
$$w_\pm\sim \sum_{j=0}^\infty w_{\pm, j},$$
can for $j>0$ be expanded in a finite sum
$$w_{\pm, j}(x,\xi,R)=\sum_{k=0}^{j-1} \mathfrak{w}_{\pm, j,k}(x,\xi',R) (\xi_n-h_\pm(x,\xi',R))^{\mu-j+k}\in S^{\mu-1,-j+1},$$ 
where $\mathfrak{w}_{\pm, j,k}$ is homogeneous of degree $-k$ in $(\xi',R)$ and can be computed by an iterative scheme as a rational function of derivatives of $h_0$, $h_+$ and $h_-$.
\item The mixed-regularity symbols $w_\pm\in S^{\mu,0}(Y\times\R;\Gamma)$ and $w_{\pm, j}\in S^{\mu-1,-j+1}$ admit holomorphic extensions to $\mp \mathrm{Im}(\xi_n)>0$.
\item The symbols $ \mathfrak{w}_{\pm, j,k}= \mathfrak{w}_{\pm, j,k}(x,\xi',R)$ depend
\begin{enumerate}
\item polynomially on $(C^{(\gamma)}_G)_{\gamma\in \cup_{k\leq j}I_k}$ and its derivatives contracted by $g_G$ and $\iota_ng_G$ and its derivatives 
\item polynomially on $h_+$, $h_-$ and rationally on $h_+-h_-$ and $h_0^{-1/2}$,
\end{enumerate}
The total degree is $j$ where each $C^{(\gamma)}_G$, $\gamma \in I_k$, has degree $k$, the metric and $h_0$ have degree zero, $h_+$ and $h_-$ have degree $1$ and $x$-derivatives increase the order by $1$. 
\end{enumerate}
\end{lem}

\begin{proof}
Items (1) and (2) follow from a short induction argument with the construction (in Definition \ref{thewopsdef}) and Theorem \ref{symbofdldl}. 
\end{proof}

We now compute $w_{\pm,1}$. By definition, we have that 
$$w_{\pm,1}=-w_{\pm,0}^2q_{\pm,1}^\partial-w_{\pm,0}\sum_{|\alpha|=1} \partial_\xi^\alpha q_{\pm,0}^\partial D_x^\alpha w_{\pm, 0}.$$ 
A short algebraic manipulation with the computation of $q_{\pm,1}$ from Proposition  \ref{compuatontaad} gives the following formulas.

\begin{prop}
\label{cpomutororw1}
For $n>1$, the terms $w_{\pm,1}$ appearing in the expansion of $w_\pm$ in Lemma \ref{wsymdexcpsps} are given by:
\begin{align*}
w_{+,1}(x,\xi',\xi_n,R)=&-\frac{1}{n!\omega_n}\mathfrak{a}_{0,+}(x,\xi',R)(\xi_n-h_+)^{\mu-1}-\frac{1}{n!\omega_n}\mathfrak{a}_{1,+}(x,\xi',R)(\xi_n-h_+)^{\mu-2}-\\
&-\frac{i(n+1)^2}{4\cdot n!\omega_n}(\partial_{x_n}h_+-\nabla_{\xi'}h_+\cdot\nabla_{x'}h_+)(\xi_n-h_+)^{\mu-2},\\
w_{-,1}(x,\xi',\xi_n,R)=&-\mathfrak{a}_{0,-}(x,\xi',R)h_0^{\mu}(\xi_n-h_-)^{\mu-1}-\mathfrak{a}_{1,-}(x,\xi',R)h_0^{\mu}(\xi_n-h_-)^{\mu-2}-\\
&-\frac{i(n+1)^2}{4}(\partial_{x_n}h_--\nabla_{\xi'}h_{-}\cdot\nabla_{x'}h_-)h_0^{\mu}(\xi_n-h_-)^{\mu-2}-\\
&-\frac{i(n+1)^2}{4}(\partial_{x_n}h_0 - \nabla_{\xi'}h_{-}\cdot\nabla_{x'}h_0)h_0^{\mu-1}(\xi_n-h_-)^{\mu-1},
\end{align*}
where the homogeneous symbols $\mathfrak{a}_{\pm,0}$ (of degree $0$) and $\mathfrak{a}_{\pm,1}$ (of degree $1$) were explicitly given in Proposition \ref{compuatontaad} above.
\end{prop}

\begin{lem}
\label{thewops}
The operators $W_\pm\in \Psi^{\mu,0}(Y\times \R;\Gamma)$ satisfy that 
\begin{enumerate}
\item $1-W_\pm Q_\pm^\partial,1-Q_\pm^\partial W_\pm \in \Psi^{0,-\infty}(Y\times \R;\Gamma)$;
\item $W_-$ preserves supports in $Y\times(-\infty,0]$ and restricts to a well defined operator $\overline{H}^\mu(Y\times [0,\infty))\to L^2(Y\times[0,\infty)$ which is invertible for large $R$;
\item $W_+$ preserves supports in $Y\times[0,\infty)$ and restricts to a well defined operator $L^2(Y\times [0,\infty))\to \dot{H}^{-\mu}(Y\times[0,\infty)$ which is invertible for large $R$.
\end{enumerate}
In particular, the operators 
\begin{align*}
S_0&:=1-W_+W_-Q^\partial:\dot{H}^{-\mu}(Y\times [0,\infty))\to \dot{H}^{-\mu}(Y\times [0,\infty)),\quad\mbox{and}\\ S_1&:=1-Q^\partial W_+W_-:\overline{H}^{\mu}(Y\times [0,\infty))\to \overline{H}^{\mu}(Y\times [0,\infty)),
\end{align*}
are normbounded by $O(R^{-\infty})$. 
\end{lem}

\begin{proof}
Part (1) follows from Proposition \ref{compososfofrmm} and Lemma \ref{wsymdexcpsps}. Parts (2) and (3) follow from Theorem \ref{continuityandwhatnot} and Theorem \ref{restrictingtohalfspaces}.
\end{proof}

\subsection{Decomposition of the inverse magnitude operator}

\begin{thm}
\label{resolventstructure}
Let $X$ be an $n$-dimensional compact manifold with boundary and $\rd$ a distance function whose square is regular at the diagonal. Set $\mu:=(n+1)/2$. Let $Q_X:\dot{H}^{-\mu}(X)\to \overline{H}^\mu(X)$ denote the restriction of $Q_{\chi,\rd^2}$ to $X$ and let $A\in \Psi^{n+1}_{\rm cl}(X;\C_+)$ denote a parametrix of $Q_{\chi,\rd^2}$. For some $R_0\geq 0$ and any $R\in \Gamma_{\pi/(n+1)}(R_0)$, we can write 
$$Q_X^{-1}=\chi_1A\chi_1'+\chi_2 (\varphi^{-1})^*W_+W_-\varphi^*\chi_2'+S,$$
where { $\chi_2,\chi_2'\in C^\infty(X)$ are functions supported in a collar neighborhood $U_0$ of $\partial X$ in $X$ and $\chi_1,\chi_1'\in C^\infty_c(X^\circ)$ are functions} such that 
$$\chi_1+\chi_2=1\quad\mbox{and}\quad \chi_j'|_{\mathrm{supp}(\chi_j)}=1, \; j=1,2,$$ 
$\varphi:\partial X\times [0,1)\to U_0$ is a collar identification, and the operators $S$, $W_-$ and $W_+$ satisfy the following as $R\to \infty$:
\begin{enumerate}
\item $S:\overline{H}^{\mu}(X)\to \dot{H}^{-\mu}(X)$ is a continuous operator with $\|S\|_{\overline{H}^{\mu}(X)\to \dot{H}^{-\mu}(X)}=O(R^{-\infty})$.
\item $W_+:L^2(\partial X\times [0,\infty))\to \dot{H}^{-\mu}(\partial X\times [0,\infty))$ is the properly supported pseudodifferential operator of mixed-regularity $(\mu,0)$ from Definitions \ref{thewopsdef} which is invertible for large $R>0$ and in local coordinates has an asymptotic expansion modulo $S^{\mu,-\infty}$ as in Lemma \ref{thewops} and preserves support in $\partial X\times [0,\infty)\subseteq \partial X\times \R$. Moreover, for $\chi,\chi'\in \C+C^\infty_c(\partial X\times [0,\infty))$ with $\chi \chi'=0$, it holds that $\|\chi W_+ \chi'\|_{L^2(\partial X\times [0,\infty))\to H^{-\mu}(\partial X\times \R)}=O(R^{-\infty})$.
\item $W_-:\overline{H}^{\mu}(\partial X\times [0,\infty))\to L^2(\partial X\times [0,\infty))$ is the properly supported pseudodifferential operator of mixed-regularity $(\mu,0)$ from Definitions \ref{thewopsdef} which is invertible for large $R>0$ and in local coordinates has an asymptotic expansion modulo $S^{\mu,-\infty}$ as in Lemma \ref{thewops} and preserves support in $\partial X\times (-\infty,0]\subseteq \partial X\times \R$. Moreover, for $\chi,\chi'\in \C+C^\infty_c(\partial X\times \R)$ with $\chi \chi'=0$, it holds that $\|\chi W_- \chi'\|_{H^{\mu}(\partial X\times \R)\to L^2(\partial X\times \R)}=O(R^{-\infty})$.
\end{enumerate}
\end{thm}

\begin{proof}
The properties of $W_+$ and $W_-$ listed in items (2) and (3) follow from the results of Subsection \ref{subsec:mexedallada}. We note that it follows from the previous subsection that 
$$\chi_1Q_M^{-1}\chi_1'+\chi_2 (\varphi^{-1})^*W_+W_-\varphi^*\chi_2':\overline{H}^{\mu}(X)\to \dot{H}^{-\mu}(X),$$ 
is a well defined continuous operator. Pick a $\chi_3$ supported close to $\partial X$ with $\chi_3=1$ on $\mathrm{supp}(\chi_2')$. We compute that 
\begin{align*}
(\chi_1Q_M^{-1}\chi_1'+&\chi_2 (\varphi^{-1})^*W_+W_-\varphi^*\chi_2')Q_X=\\
=&\chi_1+\chi_1Q_M^{-1}(\chi_1'-1)Q_X+\chi_2 (\varphi^{-1})^*W_+W_-\varphi^*Q_X\chi_3+\\
&+\chi_2 (\varphi^{-1})^*W_+W_-\varphi^*(\chi_2'-1)Q_X\chi_3+\chi_2 (\varphi^{-1})^*W_+W_-\varphi^*\chi_2'Q_X(1-\chi_3)=\\
=&\chi_1+\chi_1Q_M^{-1}(\chi_1'-1)Q_X+\chi_2 (\varphi^{-1})^*W_+W_-Q^\partial\varphi^*\chi_3+\\
&+\chi_2 (\varphi^{-1})^*W_+W_-\varphi^*(\chi_2'-1)Q_X\chi_3+\chi_2 (\varphi^{-1})^*W_+W_-\varphi^*\chi_2'Q_X(1-\chi_3)=\\
=&\chi_1+\chi_1Q_M^{-1}(\chi_1'-1)Q_X+\chi_2 +(\varphi^{-1})^*S_0\varphi^*\chi_3+\\
&+\chi_2 (\varphi^{-1})^*W_+W_-\varphi^*(\chi_2'-1)Q_X\chi_3+\chi_2 (\varphi^{-1})^*W_+W_-\varphi^*\chi_2'Q_X(1-\chi_3)=\\
=&1+\chi_1Q_M^{-1}(\chi_1'-1)Q_X+(\varphi^{-1})^*S_0\varphi^*\chi_3+\chi_2 (\varphi^{-1})^*W_+W_-\varphi^*(\chi_2'-1)Q_X\chi_3\\
&+\chi_2 (\varphi^{-1})^*W_+W_-\varphi^*\chi_2'Q_X(1-\chi_3)=1+S_2+(\varphi^{-1})^*S_0\varphi^*\chi_3+S_3+S_4.
\end{align*}
Since $\chi_1(\chi_1'-1)=0$, $S_2$ is a smoothing operator with parameter. Similarly, since $\chi_2'(1-\chi_3)=0$, $S_4$ is a smoothing operator with parameter. Using Proposition \ref{pseudolocaformfixed} and Lemma \ref{thewops}, respectively, we conclude that $S_3:\dot{H}^{-\mu}(X)\to \dot{H}^{-\mu}(X)$ and $(\varphi^{-1})^*S_0\varphi^*\chi_3:\dot{H}^{-\mu}(X)\to \dot{H}^{-\mu}(X)$ are continuous with norms bounded by $O(R^{-\infty})$ as $R\to \infty$. In particular, 
$$S_5:=(\chi_1Q_M^{-1}\chi_1'+\chi_2 (\varphi^{-1})^*W_+W_-\varphi^*\chi_2')Q_X-1,$$
satisfies that $S_5:\dot{H}^{-\mu}(X)\to \dot{H}^{-\mu}(X)$ is continuous and $\|S_5\|_{\dot{H}^{-\mu}(X)\to \dot{H}^{-\mu}(X)}=O(R^{-\infty})$ as $R\to \infty$. We conclude that $(1+S_5)^{-1}$ exists for large $R$ and $\|1-(1+S_5)^{-1}\|_{\dot{H}^{-\mu}(X)\to \dot{H}^{-\mu}(X)}=O(R^{-\infty})$ as $R\to \infty$. We therefore have that 
$$Q_X^{-1}=(1+S_5)^{-1}(\chi_1Q_M^{-1}\chi_1'+\chi_2 (\varphi^{-1})^*W_+W_-\varphi^*\chi_2')=\chi_1Q_M^{-1}\chi_1'+\chi_2 (\varphi^{-1})^*W_+W_-\varphi^*\chi_2'+S,$$
where 
$$S:=(1-(1+S_5)^{-1})(\chi_1Q_M^{-1}\chi_1'+\chi_2 (\varphi^{-1})^*W_+W_-\varphi^*\chi_2').$$
Since $\|1-(1+S_5)^{-1}\|_{\dot{H}^{-\mu}(X)\to \dot{H}^{-\mu}(X)}=O(R^{-\infty})$ as $R\to \infty$, the same holds for $S$ and the proof is complete.  
\end{proof}

By combining Theorem \ref{firstofzx} with Theorem \ref{resolventstructure}, we arrive at the following corollary.

\begin{cor}
\label{resolventstructureforz}
Let $X$ be an $n$-dimensional compact manifold with boundary and $\rd$ a distance function with property (MR) on $\Gamma$. For some $R_0\geq 0$ and any $R\in \Gamma_{\pi/(n+1)}(R_0)\cap \Gamma$, we can write 
$$\mathcal{Z}_X^{-1}=\chi_1A\chi_1'+\chi_2 (\varphi^{-1})^*W_+W_-\varphi^*\chi_2'+\tilde{S},$$
where $A$, $W_+$, $W_-$ and $\chi_1,\chi_1,\chi_2, \chi_2'\in C^\infty(X)$ are as in Theorem \ref{resolventstructure} and $\tilde{S}:\overline{H}^{\mu}(X)\to \dot{H}^{-\mu}(X)$ is continuous with $\|\tilde{S}\|_{\overline{H}^{\mu}(X)\to \dot{H}^{-\mu}(X)}=O(\mathrm{Re}(R)^{-\infty})$, as $\mathrm{Re}(R)\to \infty$ in $\Gamma$.
\end{cor}

\section{Conditional expectations of $Q_X^{-1}$ and $\mathcal{Z}_X^{-1}$}
\label{condexpsecalald}

A large motivation for this paper is the relation of the operator $\mathcal{Z}_X$ with magnitude. For that purpose, we shall be interested in computing conditional expectations of $Q_X^{-1}$ and $\mathcal{Z}_X^{-1}$ against the constant function $1$. In an accompanying paper \cite{gimpgofflouc}, we prove that this conditional expectation of $(R\mathcal{Z}_X(R))^{-1}$ coincides with the magnitude function. The section is divided into three subsections: firstly, we study the case of no boundary, secondly we proceed to compute the asymptotic expansion of the conditional expectation of $Q_X^{-1}$ and finally we produce explicit formulas for the asymptotic expansion and consider examples. As in the previous sections, we perform computations for $Q$ that later translates into results for $\mathcal{Z}$ under assumptions of property (MR).

\subsection{Asymptotic expansions for compact manifolds}

Let us consider the case that $X=M$ is a compact manifold. Starting from Lemma \ref{restiricocoald} we compute the asymptotics of $\langle Q(R)^{-1}1,1\rangle$ as $R\to \infty$ for a pseudodifferential operator with parameter ${R}$. Let $\mathrm{vol}_\rd(M)$ denote the volume of $M$ in the Riemannian metric defined from the transversal Hessian of $\rd^2$ at the diagonal. 

\begin{thm}
\label{magcompsclosedeld}
Let $M$ be an $n$-dimensional compact manifold with a distance function $\rd$ whose square is regular at the diagonal. Let $(a_{j,0})_{j\in \N}\subseteq C^\infty(M;\C_+)$ denote the sequence of homogeneous functions obtained from restriction to $\xi=0$ of the full symbol of $Q_M^{-1}$, as in Lemma \ref{restiricocoald}. It holds that 
$$\langle 1,Q_M(R)^{-1}1\rangle \sim \sum_{k=0}^\infty c_k(M,\rd)R^{n+1-k}+O(\mathrm{Re}(R)^{-\infty}), \quad \mbox{as $\mathrm{Re}(R)\to +\infty$},$$
where 
$$c_k(M,\rd)=\int_M a_{k,0}(x,1)\rd x.$$
Here $\rd x$ is the Riemannian volume density defined from $g_{\rd^2}$. The functions $a_{k,0}(x,1)$ depend on the Taylor expansion \eqref{taylorexpamdmd} as described in Theorem \ref{evaluationsofinterioraxizero} and can be computed inductively using Lemma \ref{evaluationsofinterioraxizeroind}. In particular, 
\begin{align*}
c_k(M,\rd)=
\begin{cases}
0,\; &\mbox{when $k$ is odd,}\\
\frac{\mathrm{vol}_\rd(M)}{n!\omega_n},\; &\mbox{when $k=0$},\\
\frac{n+1}{6\cdot n!\omega_n}\int_X s_{\rd^2}\rd x,\; &\mbox{when $k=2$},
\end{cases}
\end{align*}
where $s_{\rd^2}$ in local coordinates is computed as the polynomial in the Taylor coefficients of $\rd^2$ at the diagonal given as
\begin{align*}
s_{\rd^2}(x):=&3C^4(x,g\otimes g) - 3\frac{\mathfrak{c}_{2,n}(n+5)(n^2-9)}{\mathfrak{c}_{1,n}}(C^3\otimes C^3)(x,g\otimes g \otimes g), \quad\mbox{if $n\neq 1,3$}\\
s_{\rd^2}(x):=&3\bigg(10C^{4}_G(x,g_G\otimes g_G)- \frac{\mathfrak{c}_{2,3}}{\mathfrak{c}_{1,3}}(C^{3}_G\otimes C^{3}_G)(x,g_G\otimes g_G\otimes g_G)\bigg), \quad\mbox{if $n=3$}\\
\end{align*}
\end{thm}

\begin{proof}
The asymptotic expansion $\langle Q_M^{-1}1,1\rangle\sim  \sum_{k=0}^\infty c_k(M,\rd)R^{n+1-k}+O(R^{-\infty})$ 
where $c_k(M,\rd)=\int_M a_{k,0}(x,1)\rd x$ follows directly from Lemma \ref{restiricocoald} and the fact that $Q_M^{-1}$ is of order $n+1$. It follows from Lemma \ref{evaluationsofinterioraxizeroind} that $a_k(x,0,1)=0$ for odd $k$. It follows from Theorem \ref{evaluationsofinterioraxizero} that $c_0$ and $c_2$ take the prescribed form.
\end{proof}

The justification for the notation $s_{\rd^2}$ in Theorem \ref{magcompsclosedeld} comes from Example \ref{geodesciexamokad2} which shows that for the geodesic distance on a Riemannian manifold, $s_{\rd^2}$ is the scalar curvature. We further conclude the following corollary.

\begin{cor}
Let $M$ be a compact Riemannian manifold equipped with its geodesic distance. Then 
$$\langle 1,Q_M(R)^{-1}1\rangle =\frac{\mathrm{vol}_\rd(M)}{n!\omega_n}R^{n+1}+\frac{n+1}{6\cdot n!\omega_n} \int_Ms \rd xR^{n-1}+O(\mathrm{Re}(R)^{n-3}), \quad \mbox{as $\mathrm{Re}(R)\to +\infty$},$$
where $s$ denotes the scalar curvature of $M$.
\end{cor}

Combining Theorem \ref{firstofz} with Theorem \ref{magcompsclosedeld} we arrive at the following corollary.

\begin{cor}
\label{magcompsclosedeldforz}
Let $M$ be an $n$-dimensional compact manifold with a distance function $\rd$ with property (MR) on $\Gamma$. It holds that 
$$\langle 1,\mathcal{Z}_M(R)^{-1}1\rangle \sim \sum_{k=0}^\infty c_k(M,\rd)R^{n+1-k}+O(\mathrm{Re}(R)^{-\infty}), \quad \mbox{as $\mathrm{Re}(R)\to +\infty$ in $\Gamma$},$$
where $c_k(M,\rd)$ is as in Theorem \ref{magcompsclosedeld}.
\end{cor}

\subsection{A lengthy exercise in integration by parts}

To study the asymptotic expansions of $\langle 1,Q_X^{-1}1\rangle$ in the presence of a boundary, we need a series of smaller lemmas. The reader should recall the notation from Theorem \ref{resolventstructure}.

\begin{lem}
\label{decomplem}
Let $X$ be an $n$-dimensional compact manifold with boundary and $\rd$ a distance function whose square is regular at the diagonal. It holds that 
$$\langle Q_X^{-1}1,1\rangle_{L^2(X)}= \langle A1,\chi_1\rangle_{L^2(X)}+\langle W_-1, (W_+)^*(\chi_2\circ \varphi)\rangle_{L^2(\partial X\times[0,\infty))}+O(R^{-\infty}).$$
\end{lem}

\begin{proof}
We first note that $W_-$ preserves supports in $\partial X\times(-\infty,0]\subseteq \partial X\times \R$ by Lemma \ref{thewops} and $(W_+)^*$ preserves supports in $\partial X\times(-\infty,0]\subseteq \partial X\times \R$ since $W_+$ preserves supports in $\partial X\times [0,\infty)$ by Lemma \ref{thewops}. Therefore, viewing $1$ as an element of  $\overline{H}^\mu_{\rm loc} (\partial X\times [0,\infty))$ and $\chi_2\circ \varphi$ as an element of $\overline{H}^\mu_c(\partial X\times [0,\infty))$, the images $W_-1\in L^2_{\rm loc}(\partial X\times [0,\infty))$ and $(W_+)^*(\chi_2\circ \varphi)\in L^2_c(\partial X\times [0,\infty))$ are well defined and $\langle W_-1, (W_+)^*(\chi_2\circ \varphi)\rangle_{L^2(\partial X\times [0,\infty))}$ is well defined. By the same token, $\langle A1,\chi_1\rangle_{L^2(X)}$ is defined as the inner product of $\chi_1\in L^2(X)$ with the restriction of $A1_M\in L^2(M)$ to $X$.

Since  $\chi_j'|_{\mathrm{supp}(\chi_j)}=1$, for $j=1,2$, it follows that 
\begin{align*}
\langle A1,\chi_1\rangle_{L^2(X)}&=\langle A\chi_1',\chi_1\rangle_{L^2(X)}+O(R^{-\infty}) \quad\mbox{and}\\ 
\langle W_-1, (W_+)^*(\chi_2\circ \varphi)\rangle_{L^2(\partial X\times[0,\infty))}&=\langle W_-(\chi_2'\circ \varphi), (W_+)^*(\chi_2\circ \varphi)\rangle_{L^2(\partial X\times[0,\infty))}+O(R^{-\infty}).
\end{align*}
The last equality follows from Proposition \ref{pseudolocaformfixed}. Therefore, Theorem \ref{resolventstructure} reduces the statement of the theorem to the property that $\langle S1,1\rangle_{L^2(X)}=O(R^{-\infty})$. This is clear from the property of $S$ that $\|S\|_{\overline{H}^{\mu}\to \dot{H}^{-\mu}}=O(R^{-\infty})$.
\end{proof}

\begin{lem}
\label{interiorlem}
Let $X$ be an $n$-dimensional compact manifold with boundary and $\rd$ a distance function whose square is regular at the diagonal. Let $(a_{j,0})_{j\in \N}\subseteq C^\infty(M;\C_+)$ denote the sequence of homogeneous functions obtained from restriction to $\xi=0$ of the full symbol of $A$, as in Lemma \ref{restiricocoald}. It holds that  
$$\langle A1,\chi_1\rangle{ \sim}\sum_{k=0}^\infty c_{k,\chi_1}(X,\rd)R^{n+1-k}+O(R^{-\infty}),$$
where $c_{k,\chi_1}(M,\rd)=\int_X \chi_1(x)a_{k,0}(x,1)\rd x$ and $\rd x$ denotes the Riemannian volume density defined from $g_{\rd^2}$. 
\end{lem}

\begin{proof}
The lemma follows immediately from Lemma \ref{restiricocoald} since $\chi_1$ has compact support in $X^\circ$.
\end{proof}

\begin{lem}
\label{asymtpfoiad}
Let $a=a(x',\xi)\in C^\infty(\R^{n-1}\times \R^n)$ be a polynomially bounded smooth function with compact support in $x'$, and $\chi\in\mathcal{S}(\R^n)$ a real even Schwartz function. Then as $R\to +\infty$,
\begin{align*}
\frac{1}{(2\pi)^n}\int_{\R^{n-1}}\int_{\R^n}&a(x',\xi)R^{n}\hat{\chi}(-R\xi) \e^{-Rix'\xi'}\rd \xi\rd x=\\
&=\sum_{\alpha\in \N^n} \frac{D^\alpha_x\chi(0)}{\alpha!} \int_{\R^{n-1}}D^\alpha_{\xi=0}\left(a(x',\xi) \e^{-ix'\xi'}\right)\rd x R^{-|\alpha|}+O(R^{-\infty}).
\end{align*}
In particular, if $\chi$ is locally constant near $0$, then 
$$\frac{1}{(2\pi)^n}\int_{\R^{n-1}}\int_{\R^n}a(x',\xi)R^{n}\hat{\chi}(-R\xi) \e^{-Rix'\xi'}\rd \xi\rd x=\chi(0)\int_{\R^{n-1}}a(x',0)\rd x+O(R^{-\infty}).$$
\end{lem}

\begin{proof}
Consider the distribution $u_R(\xi):=R^n\hat{\chi}(-R\xi)$. For any test function $\varphi\in \mathcal{S}(\R^n)$, we compute that 
\begin{align*}
(u_R,\varphi)=\int_{\R^n} \hat{\chi}(-\xi)\varphi(\xi/R)\rd \xi&=\sum_{\alpha\in \N^n} \frac{D^\alpha_x\varphi(0)}{\alpha!}\int_{\R^n}\hat{\chi}(-\xi)\xi^\alpha\rd \xi R^{-|\alpha|}+O(R^{-\infty})=\\
&=(2\pi)^n\sum_{\alpha\in \N^n} \frac{D^\alpha_x\chi(0)}{\alpha!}R^{-|\alpha|}(\delta^{\alpha},\varphi)+O(R^{-\infty})
\end{align*}
We conclude that in $\mathcal{S}'(\R^n)$, we have an asymptotic expansion 
$$u_R=(2\pi)^n\sum_{\alpha\in \N^n} \frac{D^\alpha_x\chi(0)}{\alpha!}R^{-|\alpha|}\delta^{\alpha}+O(R^{-\infty}).$$
Using standard methods for oscillatory integrals, we see that the same expansion holds also for in the weak topology against polynomially bounded smooth functions.

We compute that 
 \begin{align*}
\frac{1}{(2\pi)^n}\int_{\R^{n-1}}\int_{\R^n}&a(x',\xi)R^{n}\hat{\chi}(-R\xi) \e^{-Rix'\xi'}\rd \xi\rd x=\frac{1}{(2\pi)^n}\int_{\R^{n-1}}\int_{\R^n}a(x',\xi)u_R(\xi) \e^{-Rix'\xi'}\rd \xi\rd x\\
&=\sum_{\alpha\in \N^n} \frac{D^\alpha_x\chi(0)}{\alpha!}R^{-|\alpha|}\int_{\R^{n-1}}(\delta^{\alpha}_\xi, a(x',\xi) \e^{-ix'\xi'})\rd x+O(R^{-\infty})=\\
&=\sum_{\alpha\in \N^n} \frac{D^\alpha_x\chi(0)}{\alpha!} \int_{\R^{n-1}}D^\alpha_{\xi=0}\left(a(x',\xi) \e^{-ix'\xi'}\right)\rd x R^{-|\alpha|}+O(R^{-\infty}).
\end{align*}
\end{proof}

\begin{lem}
\label{bodunaodcont}
Let $X$ be an $n$-dimensional compact manifold with boundary and $\rd$ a distance function whose square is regular at the diagonal. We denote the symbols of $W_\pm$ by $w_\pm$ (as in Definition \ref{thewopsdef} and Lemma \ref{thewopsdef}). Then it holds that 
$$\langle W_-1, W_+^*\chi_2\rangle{ \sim} \sum_{k=0}^\infty c_{k,\chi_2}(X,\rd)R^{n+1-k}+O(R^{-\infty}),$$
where 
\small
\begin{align*}
c_{k,\chi_2}(X,\rd)=&\int_X \chi_2(x)a_{k,0}(x,1)\rd x+\\
&+\sum_{\substack{k=|\beta|+\gamma_n+j+l\\\gamma_n>0}}\frac{i^{|\beta|+|\gamma_n|}(-1)^{|\beta|+1}}{\beta'!(\beta_n+\gamma_n)!} \int_{\partial X}\partial_{x}^{\beta}w_{-,j}(x',0,0,1)\partial_{x_n}^{\gamma_n-1} \partial_\xi^{\beta+(0,\gamma_n)} w_{+,l}(x',0,0,1)\rd x'.
\end{align*}
\normalsize
Here $\rd x$ is the Riemannian volume density on $X$ defined from $g_{\rd^2}$ and $\rd x'$ the induced Riemannian volume density on $\partial X$.
\end{lem}

\begin{proof}
The computation can be reduced to one in local coordinates, so we can assume that $w_+$ and $w_-$ are symbols of mixed-regularity $(-\mu,0)$ in $\R^n$, and up to $O(R^{-\infty})$ we can treat $W_+$ and $W_-$ as compactly based. As such, we replace $M$ by $\R^n$ and $X$ by $\R^n_+$ in all computations. Let $w^*_+$ denote the symbol of $W_+^*$. By the same arguments as in \cite[Chapter I.3]{shubinbook}, we have that 
\begin{equation}
\label{asymadjotin}
w_+^*(x,\xi,R)\sim \sum_\alpha\frac{1}{\alpha!} \partial_\xi^\alpha D_x^\alpha \overline{w_+(x,\xi,R)},
\end{equation}
in the sense of Definition \ref{genasmdef}. We note that Equation \eqref{asymadjotin} only identifies $w_+^*$ up to $S^{\mu,-\infty}$ but this suffices as symbols from $S^{\mu,-\infty}$ will only contribute to the conditional expectation with $O(R^{-\infty})$.

Using that $W_-$ and $W_+^*$ preserves supports in $\R^n_-$, we can consider $\chi_2$ as an element of $C^\infty_c(\R^n)$, and write 
$$\langle W_-1, W_+^*\chi_2\rangle_{L^2(\R^n_+)}=\int_{\R^n_+}[W_-1](x)\overline{[W_+^*\chi_2](x)}\rd x,$$
where 
$$W_+^*\chi_2(x):=\frac{1}{(2\pi)^{n}}\int_{\R^n} w_+^*(x,\xi,R)\hat{\chi_2}(\xi)\rd \xi,$$ 
is computed from the action of $(Q_+^{-1})^*$ on $\chi\in C^\infty_c(\R^n)$. We compute that 
\small
\begin{align*}
\langle W_-1&, W_+^*\chi_2\rangle_{L^2(\R^n_+)}=\frac{1}{(2\pi)^{n}}\int_{\R^n_+}\int_{\R^n}w_-(x,0,R)\overline{w_+^*(x,\xi,R)\hat{\chi}_2(\xi)\e^{ix\xi}}\rd \xi\rd x=\\
=&\frac{1}{(2\pi)^{n}}\int_{\R^n_+}\int_{\R^n}w_-(x,0,R)\overline{w_+^*(x,\xi,R)}\hat{\chi}_2(-\xi)\e^{-ix\xi}\rd \xi\rd x=\\
{ \sim}&\sum_\alpha\frac{(-1)^{|\alpha|}}{(2\pi)^{n}\alpha!} \int_{\R^n_+}\int_{\R^n}w_-(x,0,R)\partial_\xi^\alpha D_x^\alpha w_+(x,\xi,R)\hat{\chi}_2(-\xi)\e^{-ix\xi}\rd \xi\rd x+O(R^{-\infty})=\\
{ \sim}&\sum_\alpha\frac{(-1)^{\alpha_n}}{(2\pi)^{n}\alpha!} \int_{\R^n_+}\int_{\R^n}D_{x'}^{\alpha'}(w_-(x,0,R) \e^{-ix\xi})D_{x_n}^{\alpha_n} \partial_\xi^\alpha w_+(x,\xi,R)\hat{\chi}_2(-\xi)\rd \xi\rd x+O(R^{-\infty})=\\
{ \sim}&\sum_\alpha\!\!\sum_{\gamma'+\beta'=\alpha'}\frac{(-1)^{\alpha_n}}{(2\pi)^{n}\beta'!\gamma'!\alpha_n!} \int_{\R^n_+}\int_{\R^n}D_{x'}^{\beta'}w_-(x,0,R)D_{x_n}^{\alpha_n} \partial_\xi^\alpha w_+(x,\xi,R)(-i\xi')^{\gamma'}\hat{\chi}_2(-\xi) \e^{-ix\xi}\rd \xi\rd x+O(R^{-\infty})=\\
{ \sim}&\sum_{\gamma,\beta}\frac{1}{(2\pi)^{n}\beta!\gamma!} \int_{\R^n_+}\int_{\R^n}D_{x}^{\beta}w_-(x,0,R)\partial_\xi^{\gamma+\beta} w_+(x,\xi,R)(-i\xi)^\gamma\hat{\chi}_2(-\xi) \e^{-ix\xi}\rd \xi\rd x+\\
&\!\!\!\!\!\!\!\!+\sum_{\substack{\gamma,\beta,\\ \gamma_n>0}}\sum_{k=0}^{\beta_n}{b}_{\gamma,\beta,k}\int_{\R^{n-1}}\int_{\R^n}D_{x}^{\beta-(0,k)}w_-(x',0,0,R)D_{x_n}^{\gamma_n-1} \partial_\xi^{\beta+\gamma} w_+(x',0,\xi,R)(i\xi)^{(\gamma',k)}\hat{\chi}_2(-\xi) \e^{-ix'\xi'}\rd \xi\rd x+\\
&+O(R^{-\infty}).
\end{align*}
\normalsize
where
$$b_{\gamma,\beta,k}=\frac{i(-1)^{|\gamma_n|+1}\beta_n!}{(2\pi)^{n}\beta'!\gamma'!(\beta_n+\gamma_n)!k!(\beta_n-k)!} $$

By the composition formula for pseudodifferential operators (see \cite[Chapter I.3]{shubinbook}), we have that 
\begin{align*}
\sum_{\gamma,\beta}\frac{1}{(2\pi)^{n}\beta!\gamma!}& \int_{\R^n_+}\int_{\R^n}D_{x}^{\beta}w_-(x,0,R)\partial_\xi^{\gamma+\beta} w_+(x,\xi,R)(-i\xi)^\gamma\hat{\chi}_2(-\xi) \e^{-ix\xi}\rd \xi\rd x{ \sim}\\
&=\int_{\R^n_+}\chi_2(x)a(x,0,R)\rd x+O(R^{-\infty})=\langle Q_M^{-1}1,\chi_2\rangle_{L^2(\R^n_+)}+O(R^{-\infty}).
\end{align*}

We can therefore continue our calculation
\small
\begin{align*}
\langle W_-1&, W_+^*\chi_2\rangle_{L^2(\R^n_+)}{ \sim}\langle Q_M^{-1}1,\chi_2\rangle_{L^2(\R^n_+)}\\
&+\sum_{\substack{\gamma,\beta,\\ \gamma_n>0}}\sum_{k=0}^{\beta_n}{b}_{\gamma,\beta,k} \int_{\R^{n-1}}\int_{\R^n}D_{x}^{\beta-(0,k)}w_-(x',0,0,R)D_{x_n}^{\gamma_n-1} \partial_\xi^{\beta+\gamma} w_+(x',0,\xi,R)(i\xi)^{(\gamma',k)}\hat{\chi}_2(-\xi) \e^{-ix\xi}\rd \xi\rd x+\\
&+O(R^{-\infty})
\end{align*}
\normalsize

To obtain an asymptotic expansion, we expand $w_\pm$ in its defining homogeneous expansion $w_\pm\sim \sum w_{\pm,l}$ from Definition \ref{thewopsdef} (see also Lemma \ref{thewops}). We see that 
\small
\begin{align*}
&\sum_{\substack{\gamma,\beta,\\ \gamma_n>0}}\sum_{k=0}^{\beta_n}{b}_{\gamma,\beta,k} \int_{\R^{n-1}}\int_{\R^n}D_{x}^{\beta-(0,k)}w_-(x',0,0,R)D_{x_n}^{\gamma_n-1} \partial_\xi^{\beta+\gamma} w_+(x',0,\xi,R)(i\xi)^{(\gamma',k)}\hat{\chi}_2(-\xi) \e^{-ix\xi}\rd \xi\rd x+O(R^{-\infty}){ \sim}\\
&=\sum_{i=0}^\infty \sum_{\substack{i=|\beta|+|\gamma|+j+l,\\ \gamma_n>0}}\sum_{k=0}^{\beta_n}{b}_{\gamma,\beta,k} \int_{\R^{n-1}}\int_{\R^n}D_{x}^{\beta-(0,k)}w_{-,j}(x',0,0,R)D_{x_n}^{\gamma_n-1} \partial_\xi^{\beta+\gamma} w_{+,l}(x',0,\xi,R)(i\xi)^{(\gamma',k)}\hat{\chi}_2(-\xi) \e^{-ix'\xi'}\rd \xi\rd x
\end{align*}
\normalsize
Let us consider each of the terms
\small
\begin{align*}
&\sum_{\substack{i=|\beta|+|\gamma|+j+l,\\ \gamma_n>0}}\sum_{k=0}^{\beta_n}{b}_{\gamma,\beta,k} \int_{\R^{n-1}}\int_{\R^n}D_{x}^{\beta-(0,k)}w_{-,j}(x',0,0,R)D_{x_n}^{\gamma_n-1} \partial_\xi^{\beta+\gamma} w_{+,l}(x',0,\xi,R)(i\xi)^{(\gamma',k)}\hat{\chi}_2(-\xi) \e^{-ix'\xi'}\rd \xi\rd x=\\
&=R^{2\mu-i} \sum_{\substack{i=|\beta|+|\gamma|+j+l,\\ \gamma_n>0}}\sum_{k=0}^{\beta_n}{b}_{\gamma,\beta,k} \int_{\R^{n-1}}\int_{\R^n}D_{x}^{\beta-(0,k)}w_{-,j}(x',0,0,1)D_{x_n}^{\gamma_n-1} \partial_\xi^{\beta+\gamma} w_{+,l}(x',0,\frac{\xi}{R},1)(i\xi)^{(\gamma',k)}\hat{\chi}_2(-\xi) \e^{-ix'\xi'}\rd \xi\rd x.
\end{align*}
\normalsize
We can compute each of the terms using Lemma \ref{asymtpfoiad} which implies that 
\begin{align*}
\frac{1}{(2\pi)^n}\int_{\R^{n-1}}&\int_{\R^n}D_{x}^{\beta-(0,k)}w_{-,j}(x',0,0,1)D_{x_n}^{\gamma_n-1} \partial_\xi^{\beta+\gamma} w_{+,l}(x',0,\xi,1)(i\xi)^{(\gamma',k)}R^{n}\hat{\chi}_2(-R\xi) \e^{-Rix'\xi'}\rd \xi\rd x=\\
{}\\
=&\begin{cases}
O(R^{-\infty}), \quad \mbox{if} \quad(\gamma',k)\neq 0,\\
\int_{\R^{n-1}}D_{x}^{\beta}w_{-,j}(x',0,0,1)D_{x_n}^{\gamma_n-1} \partial_\xi^{\beta+(0,\gamma_n)} w_{+,l}(x',0,0,1)\rd x+O(R^{-\infty}), \quad \mbox{if} \quad(\gamma',k)= 0\end{cases}
\end{align*}
We conclude that 
\small
\begin{align*}
\langle Q_-^{-1}1&, (Q_+^{-1})^*\chi_2\rangle_{L^2(\R^n_+)}-\langle Q_M^{-1}1,\chi_2\rangle_{L^2(\R^n_+)}=\\
&{ \sim} R^{2\mu}\sum_{i=0}^\infty\sum_{\substack{i=|\beta|+\gamma_n+j+l, \\\gamma_n>0}}\frac{i(-1)^{|\gamma_n|+1}R^{-i}}{\beta'!(\beta_n+\gamma_n)!} \int_{\R^{n-1}}D_{x}^{\beta}w_{-,j}(x',0,0,1)D_{x_n}^{\gamma_n-1} \partial_\xi^{\beta+(0,\gamma_n)} w_{+,l}(x',0,0,1)\rd x+O(R^{-\infty}).
\end{align*}
\normalsize
After using $D=-i\partial$, the boundary contributions have been computed.
The lemma now follows from Lemma \ref{restiricocoald} giving the asymptotic expansion  
$$\langle Q_M^{-1}1,\chi_2\rangle_{L^2(\R^n_+)}{ \sim}\sum_k R^{n+1-k}\int_{\R^n_+} \chi_2(x)a_k(x,{0,1})\rd x.$$
\end{proof}

\subsection{Asymptotic expansions for compact manifolds with boundary}
\label{asexpbound}

We now study asymptotic expansions of $\langle 1,Q_X^{-1}1\rangle$ for a compact manifold with boundary, and give a procedure to compute the coefficients. An important difference to the case of empty boundary is the boundary contributions: we identify the from Lemma \ref{bodunaodcont} as follows.

\begin{deef}
\label{defininededdebblblbl}
If $X$ is an $n$-dimensional compact manifold with boundary and $\rd$ a distance function whose square is regular at the diagonal, then we define the sequence of functions $(B_{\rd^2,k})_{k>0}\subseteq C^\infty(\partial X)$ by
$$B_{\rd^2,k}(x'):=\sum_{\substack{k=|\beta|+\gamma_n+j+l\\\gamma_n>0}}\frac{i^{|\beta|+|\gamma_n|}(-1)^{|\beta|+1}}{\beta'!(\beta_n+\gamma_n)!}\partial_{x}^{\beta}w_{-,j}(x',0,0,1)\partial_{x_n}^{\gamma_n-1} \partial_\xi^{\beta+(0,\gamma_n)} w_{+,l}(x',0,0,1).$$
For notational simplicity, we set $B_0:=0$.
\end{deef}

\begin{prop}
\label{evaluationsbbbbbbb}
Let $X$ be an $n$-dimensional manifold with boundary, $\rd$ a distance function whose square is regular at the diagonal, and $(B_{\rd^2,k})_{k>0}\subseteq C^\infty(\partial X)$ as in Definition \ref{defininededdebblblbl}. Then for each $k>0$, $B_{\rd^2,k}$ is a polynomial in $(C^{(\gamma)}_G)_{\gamma\in \cup_{k\leq j}I_k}$ and its derivatives contracted by the metric $g_G$, its contraction along the normal to the boundary, and its derivatives of total degree $j$ where each $C^{(\gamma)}_G$, $\gamma \in I_k$, has degree $k$, the metric has degree zero and $x$-derivatives increase the order by $1$. 
\end{prop}

Computing from the results of Appendix \ref{boundacompaap}, we can describe the special cases $k=1$ and $k=2$.

\begin{prop}
\label{firstbterm}
Let $(X,\rd)$ be as in Proposition \ref{evaluationsbbbbbbb}. Then 
$$B_{\rd^2,1}(x')=\frac{(n+1)}{2\cdot n!\omega_n\sqrt{h_0(x')}}.$$
In particular, if $x_n$ is the transversal coordinate defined from the unit normal to $\partial X$ (in $g_{\rd^2}$), then 
$$B_{\rd^2,1}(x')=\frac{(n+1)}{2\cdot n!\omega_n}.$$
\end{prop}

\begin{proof}
By definition, we have that 
\begin{align*}
B_{\rd^2,1}(x')=&-iw_{-,0}(x',0,0,1)\partial_{\xi_n} w_{+,0}(x',0,0,1)=\\
=&-\frac{i(n+1)}{2\cdot n!\omega_n}h_0(x') (-h_-(x',0,1)))=\frac{(n+1)}{2\cdot n!\omega_n\sqrt{h_0(x)}}.
\end{align*}
Here we have used the identities from Appendix \ref{boundacompaap}. 
If $x_n$ is the transversal coordinate defined from the unit normal to $\partial X$, then $h_0=1$.
\end{proof}

\begin{prop}
\label{secondbterm}
Let $(X,\rd)$ be as in Proposition \ref{evaluationsbbbbbbb} and assume that $x_n$ is the transversal coordinate defined from the unit normal $\partial_n$ to $\partial X$ (in $g_{\rd^2}$). Then there are universal polynomials $\alpha_1$ and $\alpha_2$ with rational coefficients such that 
$$B_{\rd^2,2}(x')=\frac{\alpha_1(n)}{n!\omega_n}C^3(x', \partial_n\otimes \partial_n\otimes \partial_n)+\frac{\alpha_2(n)}{n!\omega_n}C^3(x',  g\otimes \partial_n).$$
\end{prop}

\begin{proof}
By definition, we have that 
\begin{align*}
B_{\rd^2,2}(x)=&\frac{1}{2}w_{-,0}(x',0,0,1)\partial_{x_n}\partial_{\xi_n}^2w_{+,0}(x',0,0,1)-\frac{1}{2}\partial_{x_n}w_{-,0}(x',0,0,1)\partial_{\xi_n}^2w_{+,0}(x',0,0,1)-\\
&-iw_{-,1}(x',0,0,1)\partial_{\xi_n}w_{+,0}(x',0,0,1)-iw_{-,0}(x',0,0,1)\partial_{\xi_n}w_{+,1}(x',0,0,1)+\\
&+\nabla_{x'}w_{-,0}(x',0,0,1)\cdot \nabla_{\xi'}\partial_{\xi_n}w_{+,0}(x',0,0,1).
\end{align*}
These expressions were computed in Section \ref{boundacompaap} of the appendix. If $x_n$ is the transversal coordinate defined from the unit normal to $\partial X$, then $h_0=1$ and $b=0$. The lemmas of Section \ref{boundacompaap}, for $h_0=1$ and $b=0$, shows that $B_{\rd^2,2}(x')$ is in the linear span of $C^3(x', \partial_n\otimes \partial_n\otimes \partial_n)$ and $C^3(x',  g\otimes \partial_n)$, and carefully inspecting the computations imply the existence of the universal polynomials $\alpha_1$ and $\alpha_2$.
\end{proof}

\begin{prop}
\label{meancurvisb2}
Let $X\subseteq \R^n$ be a domain with smooth boundary equipped with the Euclidean distance. Then for a universal polynomial $\beta$, it holds that 
$$B_{\rd^2,2}=\frac{\beta(n)}{ n!\omega_n}H,$$
where $H$ denotes the mean curvature of the boundary.
\end{prop}

\begin{proof}
Fix a point $x_0\in \partial X$ and choose coordinates as in Example \ref{euxcomdmoda1}; in other words we write $\partial X$ as the graph of a function $\varphi=\varphi(x')$ with $\nabla \varphi(x_0)=0$. The computations in Example \ref{euxcomdmoda1} show that $h_0(x')=1+|\nabla\varphi(x')|^2$ is $x_n$-independent and that $b(x_0)=\nabla\varphi(x_0)=0$. Moreover, $C^3(x,v)$ is a first order order polynomial in $v_n$, so $C^3(x, \iota_n g\otimes \iota_n g\otimes \iota_n g)=0$. A short computation using Equation \eqref{defecomecomcomada} gives us that 
$$C^3(x_0,  g\otimes \iota_n g)=-2g_{x_0}(\nabla^2\varphi(x_0))=-(n-1)H(x_0),$$
where $H$ denotes the mean curvature. The result follows from Proposition \ref{secondbterm}.
\end{proof}

Combining Lemmas \ref{decomplem}, \ref{interiorlem} and \ref{bodunaodcont} we arrive at the following theorem:

\begin{thm}
\label{asyofqsexoeod}
Let $X$ be an $n$-dimensional compact manifold with boundary and $\rd$ a distance function whose square is regular at the diagonal. Denote the Riemannian volume density on $X$ defined from $g_{\rd^2}$ by $\rd x$ and the induced Riemannian volume density on $\partial X$ by $\rd x'$. It holds that 
\begin{equation}
\label{asexpforqx}
\langle 1, Q_X^{-1}1\rangle_{L^2(X)}{ \sim}\sum_{k=0}^\infty c_{k}(X,\rd)R^{n+1-k}+O(R^{-\infty}),\quad\mbox{as $R\to +\infty$},
\end{equation}
where the coefficients $c_k(X,\rd)$ are given as 
$$c_k(X,\rd)=\int_Xa_{k,0}(x,1)\rd x+\int_{\partial X}B_{\rd^2,k}(x)\rd x',$$
where 
\begin{enumerate}
\item $a_{k,0}(\cdot,1)\in C^\infty(X)$ is an invariant polynomial in the entries of the Taylor expansion \eqref{taylorexpamdmd} as described in Theorem \ref{evaluationsofinterioraxizero} and can be computed inductively using Lemma \ref{evaluationsofinterioraxizeroind}, with $a_{k,0}=0$ if $k$ is odd; and
\item $B_{\rd^2,k}\in C^\infty(\partial X)$ is an invariant polynomial in the entries of the Taylor coefficients of $\rd^2$ at the diagonal in $X$ near $\partial X$ as described in Proposition \ref{evaluationsbbbbbbb} and can be inductively computed using Lemma \ref{wsymdexcpsps}.
\end{enumerate} 
In particular, we have that 
\begin{align}
\label{expc0bound}
c_0(X,\rd)=&\frac{\mathrm{vol}(X)}{n!\omega_n}, \quad c_{1}(X,\rd)=\frac{(n+1) \mathrm{vol}(\partial X)}{2n!\omega_n},\\
\label{expc2bound}
c_{2}(X,\rd)=&\frac{n+1}{6\cdot n!\omega_n}\int_X s_{\rd^2}\rd x+\frac{(n-1)(n+1)^2}{8\cdot n!\omega_n}\int_{\partial X}H_{\rd^2}\rd x'.
\end{align}
where the scalar curvature $s_{\rd^2}$ is defined as in Theorem \ref{magcompsclosedeld} and the mean curvature $H_{\rd^2}$ is an explicit function in the linear span of $C^3(x', \partial_n\otimes \partial_n\otimes \partial_n)$ and $C^3(x',  g\otimes \partial_n)$ that coincides with the usual mean curvature for domains in $\R^n$.
\end{thm}

Our notation $H_{\rd^2}$ in Theorem \ref{asyofqsexoeod} is justified by Proposition \ref{meancurvisb2} showing that $H_{\rd^2}=H$ is the mean curvature if $X$ is a domain in Euclidean space.

\begin{proof}
The expression in Equation \eqref{asexpforqx} follows from Lemma \ref{decomplem} by adding together the computation of Lemma \ref{interiorlem} with that in Lemma \ref{bodunaodcont}.
The computation \eqref{expc0bound} follows from the fact that $a_{0,0}(x,1)=\frac{1}{n!\omega_n}$ (see Theorem \ref{evaluationsofinterioraxizero}) and Proposition \ref{firstbterm}. The computation \eqref{expc2bound} is a consequence of Theorem \ref{evaluationsofinterioraxizero} (computing the interior contribution) and Proposition \ref{secondbterm} (computing the boundary contribution).
\end{proof}

Combining Theorem \ref{firstofzx} with Theorem \ref{asyofqsexoeod} we arrive at the following corollary.

\begin{cor}
\label{asyofqsexoeodforz}
Let $X$ be an $n$-dimensional compact manifold with boundary and $\rd$ a distance function with property (MR) on $[R_0,\infty)$, for some $R_0\geq 0$. It holds that 
$$\langle 1, \mathcal{Z}_X^{-1}1\rangle_{L^2(X)}{ \sim} \sum_{k=0}^\infty c_{k}(X,\rd)R^{n+1-k}+O(R^{-\infty}),\quad\mbox{as $R\to +\infty$},$$
where the coefficients $c_k(X,\rd)$ are as in Theorem \ref{asyofqsexoeod}.
\end{cor}

\begin{cor}
\label{andpakndpasknd}
If $X\subseteq \R^n$ is a compact domain with smooth boundary, then 
\begin{align*}
c_0(X,\rd)=\frac{\mathrm{vol}(X)}{n!\omega_n},\quad c_{1}(X,\rd)=\frac{\mu \mathrm{vol}(\partial X)}{n!\omega_n},\quad
c_{2}(X,\rd)=\frac{\mu^2(n-1)}{2\cdot n!\omega_n}\int_{\partial X}H\rd S.
\end{align*}
\end{cor}

The computation of Corollary \ref{andpakndpasknd} is compatible with the computations of \cite{gimpgoff} for $\mu\in \N$. We note that the precise proportionality constant in $c_2$ follows from the computation from  \cite{gimpgoff} for $\mu\in \N$ since the pre-factor by Proposition \ref{meancurvisb2} is determined as a universal polynomial in $n$.

\begin{appendix}

\section{Conormal distributions and parameter dependent calculus}
\label{subsec:ftofcodnsos}

Pseudodifferential operators with parameters will play an important role in our study of the operators $Q$ and $\mathcal{Z}$, both to prove meromorphic extensions and to compute asymptotic expansions. We use an approach to parameter dependence described in terms of conormal distributions to which the operator $Q$ is susceptible. More details on conormal distributions can be found in \cite[Chapter 18.2]{horIII} or \cite{simanca}, or a summary thereof in the arXiv version of this paper \cite{gimpgoffloucaarXiv}. We here review the necessary computational tools. We write $CI^m(Z,Y)$ for the space of distributions on $Z$ conormal to $Y$ that admit a classical expansion near $Y$ of order $m$; so that we have a symbol isomorphism $CI^m(Z,Y)/CI^{-\infty}(Z,Y)\cong CS^m(N^*Y)/CS^{-\infty}(N^*Y)$.

Take an $\alpha\in \C$ and define $u_{\alpha,0}\in C^\infty(\R^m\setminus \{0\})$ by $u_{\alpha,0}(y,z):=|z|^{\alpha}$. If $\alpha\notin -m-\N$, \cite[Theorem 3.2.3]{horI} guarantees that $u_{\alpha,0}$ has a unique extension to a distribution $u_\alpha\in \mathcal{D}'(\R^m)$ homogeneous of degree $\alpha$ and depending holomorphically on $\alpha\in \C\setminus (-m-\N)$. If $f=f(w)$ is a function depending holomorphically on $w$ in a punctured neighborhood of $\alpha$, we write $\mathrm{F.P.}_{w=\alpha} f$ for the constant coefficient in the Laurent expansion of $f$ around $w=\alpha$. For $\alpha\in \C$, we use the notation $\mathrm{F.P.}|z|^{-\alpha}$ to denote the distribution
$$\langle \mathrm{F.P.}|z|^{-\alpha},\varphi\rangle:= \mathrm{F.P.}_{w=\alpha}\langle u_w,\varphi\rangle.$$

\begin{prop}
\label{ftoffpa}
Let $\alpha\in \C$. Then $\mathrm{F.P.}|z|^{\alpha}$ is a tempered distribution on $\R^{m}$ and for $\xi\neq0$, we have that 
\begin{align*}
\mathcal{F}\mathrm{F.P.}|z|^{\alpha}=
\begin{cases}
\pi^{m/2}2^{\alpha+m}\frac{\Gamma\left(\frac{\alpha+m}{2}\right)}{\Gamma\left(-\frac{\alpha}{2}\right)}|\xi|^{-m-\alpha}, \quad &\alpha\in \C\setminus (-m-2\N),\\
\frac{\pi^{m/2}(-1)^l}{2^{2l}l!\Gamma\left(\frac{m}{2}+l\right)} |\xi|^{2l}(-\log|\xi|^2+\beta_{l,m}), \quad &\alpha=-m-2l, \; l\in \N,
\end{cases}
\end{align*}
where 
$$\beta_{l,m}:=2\log(2)+\frac{1}{2}\psi(m/2+l)-H_l-\gamma,$$ 
and $\psi(z):=\frac{\Gamma'(z)}{\Gamma(z)}$, $H_0=0$ and $H_l:=\sum_{j=1}^l\frac{1}{j}$ for $l>0$, and $\gamma$ is the Euler-Mascheroni constant.
\end{prop}

This computation can be found in \cite[Lemma 25.2]{samko} par the value of $\beta_{l,m}$, which can be found from a Laurent expansion. See also \cite[Exemple 5, Chapitre VII.7]{schwartz} and \cite{gimpgoffloucaarXiv}.

For $g\in C^\infty(Z)$ and a submanifold $Y\subseteq Z$ with a prescribed tubular neighborhood $U\cong NY$, we define the transversal Hessian of $g$ in $y\in Y$ as the symmetric bilinear form on $(NY)_y$ defined from the Hessian at the zero section of $g$ restricted to $(NY)_y$ along the tubular neighborhood.

\begin{prop}
\label{oneonfandlog}
Assume that $Y\subseteq Z$ is an $k$-dimensional smooth submanifold of an $N$-dimensional smooth manifold. Let $\tilde{G}\in C^\infty(Z)$ be a smooth function such that 
\begin{itemize}
\item $\tilde{G}$ and $\rd \tilde{G}$ vanishes on $Y$;
\item $\tilde{G}(x)> 0$ for $x\notin Y$; and
\item for each $y\in Y$, the transversal Hessian $H_{\tilde{G}}$ of $\tilde{G}$ (defined as in Definition \ref{transversehessdefi}) is a positive definite quadratic form on the transversal tangent bundle of $Y\subseteq Z$.
\end{itemize}
Then $\log \tilde{G}\in CI^{-N+k}(Z;Y)$ and its principal symbol $\overline{\sigma}_{-N+k}\left(\log \tilde{G}\right)\in C^\infty(N^*Y\setminus Y)$ is given by 
$$\overline{\sigma}_{-N+k}\left(\log \tilde{G}\right)(y,\xi)=-2\pi(N-k-1)!\omega_{N-k-1}|g_{\tilde{G}}(\xi,\xi)|^{-(N-k)/2}, \quad\xi\neq 0,$$
where $g_{\tilde{G}}$ is the metric dual to the transversal Hessian $H_{\tilde{G}}$.
\end{prop}

\begin{proof}
The statement is local, so we can assume that there is an open set $U\subseteq \R^n$ containing $0$ such that $Z=U$ and $Y=U\cap \R^k$. Since $\rd g$ vanishes on $TY$, we can consider its restriction to $Y$ to be a section $\rd g|_Y:Y\to N^*Y$. Under the assumption that the transversal Hessian of $g$ is non-degenerate in all points of $Y$, we can assume that $U$ is taken small enough to be able to choose coordinates $(\tilde{y},\tilde{z})$ on $U$ such that $g(\tilde{y},\tilde{z})=|\tilde{z}|^{2}$ for $\tilde{z}\neq 0$. For notational simplicity, we assume that $g(y,z)=|z|^2$ and that $Z=U=\R^N$ and $Y=\R^k$. The proposition now follows from Proposition \ref{ftoffpa}.
\end{proof}

This work makes heavy use of parameter dependent pseudodifferential operators. This subject is well explained in \cite{shubinbook}. We only require classical symbols with parameters. For clarity, for a conical set $\Gamma\subseteq \C$ and $m\in \C$, a classical symbol with parameter of order $m$ is a symbol with parameter $a$ that admits a sequence $(a_j)_{j\in \N}$ of functions homogeneous of degree $m-j$ ($a_j(x,t\xi,tR)=t^{m-j}a_j(x,\xi,R)$ for $t>0$) such that $a\sim \sum_j a_j$ away from $(\xi,R)=0$. We write $CS^m$ for the space of classical symbols with parameter of order $m$ and $\Psi^m_{\rm cl}(M;\Gamma)$ for the space of classical pseudodifferential operators with parameter $R\in \Gamma$ of order $m$. An important feature of the parameter dependent calculus is that there is a Gårding equality, we omit its proof as it extends ad verbatim from the classical setting \cite{horIII}.

\begin{thm}[G\aa rding inequality]
\label{gardingwithpara}
Let $\Gamma\subseteq \Gamma_\alpha(0)\cup -\Gamma_\alpha(0)$ be a bisector with opening angle $<\alpha\in [0,\pi/2)$, and $A\in \Psi^m_{\rm cl}(M;\Gamma)$ a formally self-adjoint operator with strictly positive principal symbol, i.e. for some $\epsilon>0$, $\overline{\sigma}(A)(x,\xi,R)\geq \epsilon$ for all $|\xi|^2+|R|^2=1$ and $x\in M$. 

Then for large $R$, the quadratic form $f\mapsto \langle Af,f\rangle_{L^2}$ is continuous, positive and coercive on $H^{\mathrm{Re}(m)/2}_{|R|}$. To be precise, there is an $R_0>0$ and a $C>0$ such that 
$$\frac{1}{C} \|f\|_{H^{\mathrm{Re}(m)/2}_{|R|}}\leq\langle Af,f\rangle_{L^2}\leq C\|f\|_{H^{\mathrm{Re}(m)}_{|R|}}, \quad\forall f\in H^s(M), \; |R|>R_0.$$
\end{thm}

One of the reasons for introducing conormal distributions above was that it will give us a direct way of verifying that the magnitude operator is an elliptic pseudodifferential operator with parameter. We let $Z=U\times \R\subseteq M\times M\times \R$ denote an $\R$-invariant tubular neighborhood of the diagonal of $M$. Assume that $K\in CI^m(Z; M\times \{0\})$ in exponential coordinates has a classical asymptotic expansion of the form $K\sim \sum_{j=0}^\infty \chi K_j$ where $\chi\in C^\infty_c(U)$ is a function with $\chi=1$ near the diagonal and $K_j$ is a smooth function on $TM\oplus \R\setminus (M\times \{0\})$ such that if $m\notin \Z$, $K_j$ is homogeneous of degree $-m-j-n-1$. If $m\in \Z$, $K_j=u_j+p_j\log(|v|^2+\eta^2)$ where $u_j$ is homogeneous of degree $-m-j-n-1$ and $p_j$ is a homogeneous polynomial in $(v,\eta)$ of degree $-m-j-n-1$ (in particular $p_j=0$ if $j<-m-n-1$).

\begin{deef}
\label{unidoasissma}
We shall say that $K\in CI^m(Z; M\times \{0\})$ has a {\bf uniform asymptotic expansion} if for any $\alpha\in \N^p$, $\beta\in \N^n$, $k, N\in \N$ there is a constant $C>0$ and an $N_0\in \N$ such that for $R\neq 0$ 
$$\left|\partial_x^\alpha\partial_v^\beta\partial_\eta^k \left(K- \sum_{j=0}^{N_0}\chi K_j\right)\right|\leq C(1+|v|+|\eta|)^{-N}.$$
\end{deef}

For $x\in M$, we write $\exp_x:T_xM\to M$ for the exponential map. Note that under our assumption on the injectivity radius, $\exp_x:B_xM\to M$ is a diffeomorphism onto its range.

\begin{prop}
\label{compsuppfotkr}
Let $M$ be a smooth $n$-dimensional manifold and $Z$ is as in the preceding paragraphs. Assume that $K\in CI^m(Z;M\times \{0\})$ admits a uniform asymptotic expansion. Then there is a pseudodifferential operator with parameter $A\in \Psi^m(M;\R)$ such that for any $R$, the Schwartz kernel of $A(R)$ is given by $\mathcal{F}_{\eta\to R}K$. In other words, for $f\in C^\infty(M)$, $A(R)f$ is defined as the oscillatory integral
$$A(R)f(x):=\int_{B_xM\oplus \R} K(x,v,\eta)f(\exp_x(v))\mathrm{e}^{i\eta R}\rd v\rd \eta.$$
In particular, the full symbol of $A$ in $CS^m(M;\R)/S^{-\infty}(M;\R)$ coincides with the full symbol of $K$ in $CS^m(N^*(\mathrm{Diag}_M\times \{0\}\subseteq Z))/S^{-\infty} (N^*(\mathrm{Diag}_M\times \{0\}\subseteq Z))$.
\end{prop}

\begin{proof}
If the uniform asymptotic expansion of $K$ only contains one term, i.e. $K=K_0$, the statement of the proposition holds by homogeneity properties of the Fourier transform. As such, the proposition is in fact a statement concerning the asymptotic completeness of the space of pseudodifferential operators with parameters. This is proven just as in the usual setting (see \cite[Proposition 18.1.3]{horIII}).
\end{proof}

\section{Partial fraction decompositions of symbols}
\label{partialappa}

We note the following structural result from basic calculus:
\begin{lem}
\label{lem:xipf}
For all $l,m \in \mathbb{N}$ with $m<2l$, there exists homogeneous rational functions (with rational coefficients) of $(h_+,h_-)$  denoted by $\beta_{l,m,0,\pm},\beta_{l,m,1,\pm}, \ldots, \beta_{l,m,l-1,\pm}$, where each $\beta_{l,m,j,\pm}=b_{l,m,j,\pm}(h_+,h_-)$ has homogeneous degree $m-j -l$ in $(h_+,h_-)$, such that
\begin{align*}
\xi_n^m(\xi_n - h_+)^{-l} (\xi_n - h_-)^{-l}=\sum_{j=0}^{l-1} \beta_{l,m,j,+}(h_+,h_-)(\xi_n-h_+)^{j-l}+\sum_{j=0}^{l-1} \beta_{l,m,j,-}(h_+,h_-)(\xi_n-h_-)^{j-l}.
\end{align*}
\end{lem}

More generally, by differentiating with respect to $h_\pm$ we obtain: 
\begin{lem}
For all $l,m \in \mathbb{N}$ and two distinct complex numbers $h_+,h_-\in \C$, consider the rational function:
$$K_{m,l}(\xi_n) := (\xi_n - h_+)^{-m} (\xi_n - h_-)^{-l}.$$
This rational function can be decomposed as
\begin{align*}
&K_{m,l}(\xi_n)\\ & = \sum_{j=0}^{m-1} \frac{(-1)^{j}}{(h_+-h_-)^{l+j}} \binom{l+j-1}{j}(\xi_n - h_+)^{-m+j} + \sum_{j=0}^{l-1} \frac{(-1)^{m}}{(h_+-h_-)^{m+j}} \binom{m+j-1}{j}(\xi_n - h_-)^{-l+j}\ .
\end{align*}
\end{lem}

These following formulas allow us to obtain explicit partial fraction decompositions for the terms relevant to the factorization of $\mathcal{Z}_R$.

\begin{cor}
\label{compofodemdmdm}
For all $l \in \mathbb{N}$ and two distinct complex numbers $h_+,h_-\in \C$,
\begin{align*}
\xi_n (\xi_n - h_+)^{-l} (\xi_n - h_-)^{-l} =& K_{l-1,l}(\xi_n) + h_+ K_{l,l}(\xi_n)\ ,\\
\xi_n^2 (\xi_n - h_+)^{-l} (\xi_n - h_-)^{-l} =&K_{l-1,l-1}(\xi_n)+h_-K_{l-1,l}(\xi_n)+ h_+K_{l,l-1}(\xi_n)+h_+h_- K_{l,l}(\xi_n)\ ,\\
\xi_n^3 (\xi_n - h_+)^{-l} (\xi_n - h_-)^{-l} =&K_{l-2,l-1}(\xi_n) + (2h_++h_-) K_{l-1,l-1}(\xi_n)+\\
&+(h_-^2+h_+^2+h_+h_-)K_{l-1,l}(\xi_n)+h_+^2h_- K_{l,l}(\xi_n)
\end{align*}
\end{cor}

\section{Evaluation of some boundary symbols at zero}
\label{boundacompaap}

For the purpose of computations in Subsection \ref{asexpbound}, we are interested in evaluating some symbols at $\xi=0$ and $R=1$. We will use the notations from the Sections \ref{structurofinversesec} and  \ref{condexpsecalald} freely. We tacitly assume that $n>1$ to avoid limit cases.

Recall from Proposition \ref{rootsofmetric} that  
$$R^2+g(\xi,\xi)=h_0(\xi_n-h_+)(\xi_n-h_-), $$
where $h_\pm=h_\pm(x,\xi',R)\in S^1(T^*Y\times \R, Y\times \R;\C)$ are of the form
$$h_\pm(x,\xi',R)=-\frac{\xi'(b(x))}{h_0(x)}\pm i\frac{\sqrt{R^2+g_Y(\xi',\xi')-(\xi'(b))^2}}{\sqrt{h_0(x)}}.$$
Here we use the splitting of the metric
$$
g=\begin{pmatrix} h_0& b\\ b^T& g_Y\end{pmatrix},
$$
For $x_n=0$, i.e. on $\partial X$, we write $x'$ instead of $(x',0)$. We conclude the following lemma from elementary computations, which are included in \cite{gimpgoffloucaarXiv}. 
We use the notation $\mu=\frac{n+1}{2}$.

\begin{lem}
The following identities hold on $\partial X$:
\begin{align*}
w_{-,0}(x',0,0,1)\partial_{x_n}\partial_{\xi_n}^2w_{+,0}(x',0,0,1)&=\frac{\mu(\mu-1)(\mu-2)}{2n!\omega_n}\partial_{x_n}h_0(x'),\\
\partial_{x_n}w_{-,0}(x',0,0,1)\partial_{\xi_n}^2w_{+,0}(x',0,0,1)&=-\frac{\mu^2(\mu-1)}{2\cdot n!\omega_n}\partial_{x_n}h_0(x'),\\
\nabla_{x'}w_{-,0}(x',0,0,1)\cdot \nabla_{\xi'}\partial_{\xi_n}w_{+,0}(x',0,0,1)&=-\frac{\mu^2(\mu-1)}{2\cdot n!\omega_n}\frac{b(x')}{h_0(x')}\cdot \nabla_{x'}h_0(x'),\\
w_{-,1}(x',0,0,1)\partial_{\xi_n}w_{+,0}(x',0,0,1)&=\\
&\hspace*{-3.5cm} =\frac{i\mu\mathfrak{c}_{1,n} (n^2-1)}{(n!\omega_n)^2}\bigg(\frac{3}{2}C^3(x',g\otimes \iota_ng)+\frac{17(n+3)}{4h_0(x)}C^3(x', \iota_n g\otimes \iota_n g\otimes \iota_n g)\bigg)-\\
&\hspace*{-3.0cm}  -\frac{7i\mu^3}{4\cdot n!\omega_n}\bigg(\partial_{x_n}h_0(x') + \frac{b(x')}{h_0(x')}\cdot \nabla_{x'}h_0(x')\bigg),\\
w_{-,0}(x',0,0,1)\partial_{\xi_n}w_{+,1}(x',0,0,1)&=\\
&\hspace*{-3.5cm} =\frac{i\mathfrak{c}_{1,n} (n^2-1)}{(n!\omega_n)^2}\bigg(-\frac{3(\mu-1)}{2}C^3(x',g\otimes \iota_ng)+\frac{\mu(n+3)}{4h_0(x')}C^3(x, \iota_n g\otimes \iota_n g\otimes \iota_n g) \bigg)+\\
&\hspace*{-3.0cm} +\frac{i\mu^2(3\mu-5)}{4\cdot n!\omega_n}\bigg(\partial_{x_n}h_0(x') + \frac{b(x')}{h_0(x')}\cdot \nabla_{x'}h_0(x')\bigg).
\end{align*}
\end{lem}

\end{appendix}


\begin{thebibliography}{100}

\bibitem{antil} H. Antil, T. Berry, J. Harlim, \emph{Fractional diffusion maps}, Applied and Computational Harmonic Analysis 54 (2021), pp.~145--175.

\bibitem{barcarbs} J. A. Barcel\'o, and A. Carbery, \emph{On the magnitudes of compact sets in Euclidean spaces}, Amer. J. Math. 140 (2018), pp.~449--494.

\bibitem{bermanwitt} R. Berman, S. Boucksom, D. Witt Nyström, \emph{Fekete points and convergence towards equilibrium measures on complex manifolds}, Acta Math. 207 (2011), no. 1, pp.~1--27.

\bibitem{bermandet} R. J. Berman, \emph{Determinantal point processes and fermions on polarized complex manifolds: bulk universality}, Algebraic and analytic microlocal analysis, pp.~341--393, Springer Proc. Math. Stat., 269, Springer, Cham, 2018. 

\bibitem{bishcut} R. L. Bishop, \emph{Decomposition of cut loci}, Proc. Amer. Math. Soc. 65 (1977), pp.~133--136.

\bibitem{bunch1} E. Bunch, J. Kline, D. Dickinson, S. Bhat, G. Fung, \emph{Weighting vectors for machine learning: numerical harmonic analysis applied to boundary detection}, arXiv:2106.00827.

\bibitem{bunch2} E. Bunch, J. Kline, D. Dickinson, G. Fung, \emph{Practical applications of metric space magnitude and weighting vectors},  arXiv:2006.14063.

\bibitem{chocho} S. Cho, \emph{Quantales, persistence, and magnitude homology}, arXiv:1910.02905.


\bibitem{eskinbook} G. I. Eskin, \emph{Boundary value problems for elliptic pseudodifferential equations}, Translations of Mathematical Monographs, 52. American Mathematical Society, Providence, R.I., 1981. 


\bibitem{franklieb} R. L. Frank, E. H. Lieb, \emph{A `liquid-solid' phase transition in a simple model for swarming, based on the `no flat-spots' theorem for subharmonic functions}, Indiana Univ. Math. J. 67 (2018), pp.~1547--1569.

\bibitem{gimpgoff} H. Gimperlein, M. Goffeng, \emph{On the magnitude function of domains in Euclidean space}, Amer.~J.~Math. 143 (2021), pp.~939--967.

\bibitem{gimpgoff2} H. Gimperlein, M. Goffeng, \emph{The Willmore energy and the magnitude of Euclidean domains},
Proc. Amer. Math. Soc. 151 (2023), pp.~897--906

\bibitem{gimpgofflouc} H. Gimperlein, M. Goffeng, and N. Louca, \emph{The magnitude and spectral geometry}, arXiv preprint 2201.11363.



\bibitem{gimpgoffloucaarXiv} H. Gimperlein, M. Goffeng, and N. Louca, \emph{Semiclassical analysis of a nonlocal boundary value problem related to magnitude}, arXiv preprint 2201.11357.




\bibitem{govchep} D. Govc, R. Hepworth, \emph{Persistent magnitude}, J. Pure Appl. Algebra 225 (2021), no. 3, Paper No. 106517, 40 pp.

\bibitem{g2} G. Grubb, \emph{Spectral results for mixed problems and fractional elliptic operators}, J. Math. Anal. Appl. 421 (2015), pp.~1616--1634.

\bibitem{g3} G. Grubb, \emph{Fractional Laplacians on domains, a development of H\"{o}rmander's theory of $\mu$-transmission pseudodifferential operators}, Adv. Math. 268 (2015), pp.~478--528.

\bibitem{g4}G. Grubb, \emph{Local and nonlocal boundary conditions for $\mu$-transmission and fractional elliptic pseudodifferential operators}, Anal. PDE 7 (2014), pp.~1649--1682.

\bibitem{grubb1} G. Grubb, \emph{Green's formula and a Dirichlet-to-Neumann operator for fractional-order pseudodifferential operators}, Comm. Partial Differential Equations  43 (2018), pp.~750--789.

\bibitem{grubbibp} G. Grubb, \emph{Integration by parts and Pohozaev identities for space-dependent fractional-order operators}, J. Differential Equations 261 (2016), pp.~1835--1879.

\bibitem{ggreen} G. Grubb, \emph{Functional Calculus of Pseudodifferential Boundary Problems}, Birkh\"{a}user Boston, 1996.

\bibitem{holuhl} S. Holman, G. Uhlmann, \emph{On the microlocal analysis of the geodesic X-ray transform with conjugate points}, J. Differential Geom. 108 (2018), no. 3, pp.~459--494. 

\bibitem{hornotes} L. H\"{o}rmander, \emph{Seminar notes on pseudo-differential operators and boundary problems}, Lectures at IAS Princeton 1965-66, available from Lund University, \href{https://lup.lub.lu.se/record/7c9d8485-d7d6-4d47-80a5-a5533685c4ca}{https://lup.lub.lu.se/record/7c9d8485-d7d6-4d47-80a5-a5533685c4ca}.

\bibitem{horI} L. H\"{o}rmander, \emph{The analysis of linear partial differential operators. I. Distribution theory and Fourier analysis}, Reprint of the 1994 edition. Classics in Mathematics. Springer, Berlin, 2007.

\bibitem{horIII} L. H\"{o}rmander, \emph{The analysis of linear partial differential operators. III. Pseudo-differential operators}, Reprint of the 1994 edition. Classics in Mathematics. Springer, Berlin, 2007.

\bibitem{leincobb} T. Leinster, and C. A. Cobbold, \emph{Measuring diversity: the importance of species similarity}, Ecology 93 (2012), pp.~477--489.

\bibitem{leinster} T. Leinster, \emph{The magnitude of metric spaces}, Doc. Math. 18 (2013), pp.~857--905.

\bibitem{leinmeck} T. Leinster, M. Meckes, \emph{The magnitude of a metric space: from category theory to
geometric measure theory},  in: N.~Gigli (ed.), \emph{Measure theory in non-smooth spaces}, pp.~156--193,
Partial Differ. Equ. Meas. Theory, De Gruyter Open, Warsaw, 2017. 

\bibitem{ls} T.  Leinster, M. Shulman, \emph{Magnitude homology of enriched categories and metric spaces}, Algebr. Geom. Topol. 21 (2021), pp.~2175--2221.

\bibitem{leinwill} T. Leinster, S. Willerton, \emph{On the asymptotic magnitude of subsets of Euclidean space}, Geom. Dedicata 164(1) (2013), pp.~287--310.

\bibitem{leschhab} M. Lesch, \emph{Operators of Fuchs type, conical singularities, and asymptotic methods}, Teubner-Texte zur Mathematik, 136. B. G. Teubner Verlagsgesellschaft mbH, Stuttgart, 1997.  arXiv:dg-ga/9607005.

  
\bibitem{mcleanbook} W.  McLean, \emph{Strongly elliptic systems and boundary integral equations}, Cambridge University Press, Cambridge, 2000. 

\bibitem{meckes1} M. Meckes, \emph{Positive definite metric spaces}, Positivity 17 (2013), pp.~733--757.

\bibitem{meckes} M. Meckes, \emph{Magnitude, diversity, capacities and dimensions of metric spaces}, Potential Anal. 42 (2015), pp.~549--572.

\bibitem{meckes20} M. Meckes, \emph{On the magnitude and intrinsic volumes of a convex body in Euclidean space},
Mathematika 66 (2020), pp.~343--355. 

\bibitem{ottermag} N. Otter, \emph{Magnitude meets persistence. Homology theories for filtered simplicial sets}, arXiv:1807.01540.

\bibitem{samko} S. G. Samko, A. A. Kilbas, O. I. Marichev, \emph{Fractional Integrals and Derivatives: Theory and Applications}, Gordon and Breach Science Publishers, Amsterdam, 1993.

\bibitem{schwartz} L. Schwartz, \emph{Théorie des distributions}, Hermann, Paris, 1997.

\bibitem{shubinbook} M. Shubin, \emph{Pseudodifferential operators and spectral theory}, Springer-Verlag, Berlin, second edition, 2001.

\bibitem{simanca} S. R. Simanca, \emph{Pseudo-differential operators}, Pitman Research Notes in Mathematics Series, 236. Longman Scientific \& Technical, Harlow; copublished in the United States with John Wiley \& Sons, Inc., New York, 1990.

\bibitem{summerv} V. W. Summer, \emph{Torsion in the Khovanov homology of links and the magnitude homology of graphs},
Ph.D.~thesis, North Carolina State University, 2019, available at https://repository.lib.ncsu.edu/handle/1840.20/36588

\bibitem{triebel} H. Triebel, D. Yang, \emph{Spectral theory of Riesz potentials on quasi-metric spaces}, Math. Nachr. 238 (2002), pp.~160--184.

\bibitem{warnerconj} F. W. Warner, \emph{The conjugate locus of a Riemannian manifold}, Amer. J. Math. 87 (1965), pp.~575--604. 

\bibitem{willna} S. Willerton, \emph{Heuristic and computer calculations for the magnitude of metric spaces}, arXiv:0910.5500.

\bibitem{will} S. Willerton, \emph{On the magnitude of spheres, surfaces and other homogeneous spaces}, Geom. Dedicata 168 (2014), pp.~291--310.

\bibitem{willoddb} S. Willerton, \emph{The magnitude of odd balls via Hankel determinants of reverse Bessel polynomials}, Discrete Anal. 2020, Paper No. 5, 42 pp. 

\bibitem{loggas} A. Zabrodin, P. Wiegmann, \emph{Large-N expansion for the 2D Dyson gas}, Journal of Physics A 39 (2006),  pp.~8933--8963.
\end{thebibliography}
\end{document}